\documentclass[a4paper,10pt]{article}
\usepackage{fancyhdr}
\usepackage[T1]{fontenc}
\usepackage[english]{babel}
\usepackage[utf8]{inputenc}
\usepackage{amsmath, amsfonts, amsthm,amssymb, dsfont, mathrsfs}
\usepackage{indentfirst}
\usepackage{color}
\usepackage{graphicx}
\usepackage{enumerate}
\usepackage{braket}
\usepackage{fontenc}
\usepackage{hyperref}
\usepackage{pgfplots}
\usepackage{calligra}
\usepackage{geometry}
\usepackage{math}

\newtheorem{lemma}{Lemma}[section]
\newtheorem{theorem}[lemma]{Theorem}
\newtheorem{proposition}[lemma]{Proposition}
\newtheorem{corollary}[lemma]{Corollary}

\theoremstyle{definition}

\newtheorem{definition}[lemma]{Definition}

\theoremstyle{remark}

\newtheorem{remark}[lemma]{Remark}
\numberwithin{equation}{section}

\renewcommand{\be}{\begin{equation}}
\newcommand{\ee}{\end{equation}}

\renewcommand{\bar}{\overline}

\renewcommand{\Re}{{\rm Re }\, }

\newcommand{\J}{\mathbb J}
\newcommand{\supp}{{\rm supp }\,}

\newcommand{\floor}[1]{\lfloor #1 \rfloor}

\newcommand{\Op}[1]{{\rm Op}\left(#1\right)}
\newcommand{\Opw}[1]{{\rm Op}^W\!\left(#1\right)}

\newcommand{\Opbw}[1]{{\rm Op}^{\rm BW}\!\left(#1\right)}

\renewcommand{\div}{{\rm div}}

\newcommand{\pa}{\partial}

\renewcommand{\s}{{\sigma}}

\newcommand{\ii}{{\im}}

\newcommand{\Tb}{{\breve T}}

\newcommand{\vr}{\varrho}

\newcommand{\vect}[2]{{\begin{pmatrix}#1 \\ #2\end{pmatrix}}}

\newcommand{\hol}{{W^{\vr, \infty}}}

\newcommand{\eq}{{\tm}}

\title{\bf  Local well posedness  of the \\
Euler-Korteweg equations on $ \T^d $}

\author{M. Berti, 
 A. Maspero,
 F. Murgante\footnote{
International School for Advanced Studies (SISSA), Via Bonomea 265, 34136, Trieste, Italy. 
 \textit{Emails: } \texttt{berti@sissa.it},  \texttt{alberto.maspero@sissa.it},
 \texttt{fmurgant@sissa.it}
 }}

\begin{document}
\maketitle

\begin{flushright}
\vspace{1em}
{\it Dedicated to the memory  of   Walter Craig}
\vspace{1em}
\end{flushright}

\begin{abstract}
We consider the Euler-Korteweg system with space periodic boundary conditions
$ x \in \T^d $. We prove 
a local in time existence result of 
classical solutions for irrotational velocity fields 
requiring natural minimal regularity assumptions on the initial data.  
\end{abstract}


\section{Introduction}

In this paper we consider the compressible Euler-Korteweg (EK) system 
\begin{equation}
\label{EK-complete}
\begin{cases} \partial_t \rho + {\rm div}(\rho \vec u)=0\\
\partial_t \vec u + \vec u \cdot \nabla \vec u + 
\nabla g(\rho)= \nabla \big( K(\rho)\Delta \rho+ \frac{1}{2} K'(\rho)|\nabla \rho|^2 \big)  \, , 
\end{cases} 
\end{equation}
which is a modification of the 
Euler equations for compressible fluids to include capillary  effects,
under space periodic boundary conditions $ x \in \T^d := (\R/2 \pi \Z)^d $.   
The scalar variable $ \rho (t,x)  > 0 $ is the density of the fluid and 
 $ \vec u (t,x) \in \R^d $ is the time dependent velocity field. 
The functions $ K (\rho )  $, $ g (\rho )  $ are defined 
on $ \R^+ $, smooth,  and $ K (\rho )  $ is  positive.

The quasi-linear equations \eqref{EK-complete} appear in a variety of physical contexts 
modeling phase transitions  \cite{DS},  water waves 
\cite{BDL}, quantum hydrodynamics where $ K(\rho) = \kappa / \rho $
 \cite{AM}, see also  \cite{CDCS}.

Local well posedness results for the (EK)-system  have been obtained in 
Benzoni-Gavage, Danchin and Descombes \cite{BDD}
 for initial data  sufficiently 
 localized in the space variable $ x \in \R^d $. 
Then, thanks to dispersive estimates, global in time existence results 
have been obtained for small irrotational data by Audiard-Haspot \cite{audiard1}, 
assuming the sign condition $g'( \rho)  > 0$.
The case of quantum hydrodynamics corresponds to $K(\rho) = \kappa/\rho$ and, {in this case}, the (EK)-system
 is formally equivalent, via Madelung transform, to a semilinear
Schr\"odinder equation on $ \R^d $. {Exploiting this fact,} 
global in time weak solutions have been obtained by Antonelli-Marcati \cite{AM, AM2}
also allowing $\rho(t,x) $ to become zero (see also the recent paper \cite{AM3}).

\smallskip

In this paper we prove a local in time existence result 
for the solutions of 
\eqref{EK-complete}, with space periodic boundary conditions, under
natural minimal regularity assumptions  on the initial datum in 
Sobolev spaces, see Theorem \ref{thm:main0}.
Relying on this result, in a forthcoming paper \cite{BMM},
we shall prove a set of long time existence results for the (EK)-system in 
$ 1 $-space dimension, in the same spirit of \cite{BD}, \cite{BFF}. 

\smallskip

We consider 
 {an initially irrotational  velocity field that, under the evolution of \eqref{EK-complete}, 
remains irrotational for all times}.
An irrotational vector field on $ \T^d $ reads (Helmotz decomposition)
\be\label{vecuc}
\vec u = \vec c (t) + \nabla \phi \, , \quad \vec c (t) \in \R^d  \, , \quad 
 \vec c (t) = {\frac{1}{(2\pi)^d}} \int_{\T^d} \vec u \, \di x   \, , 
\ee
where $ \phi : \T^d \to \R $ is a scalar potential. 
By the second equation in 
\eqref{EK-complete} and ${\rm rot}\, \vec u  = 0 $, we get 
$$
\pa_t \vec c (t) 
= - {\frac{1}{(2\pi)^d}} \int_{\T^d} \vec u \cdot \grad \vec u \, \di x  
= {\frac{1}{(2\pi)^d}} \int_{\T^d}  -\frac12  \grad (|\vec u|^2) \, \di x  = 0  \quad
\Longrightarrow \quad \vec c(t) = \vec c(0)
$$
is independent of time.  
Note that if the dimension $d = 1$, the average $ {\frac{1}{2\pi}} \int_{\T} u(t,x) \di x$ is
an integral of motion for \eqref{EK-complete}, and 
thus any solution $u (t,x)$, $ x \in \T $,  of the (EK)-system \eqref{EK-complete} 
has the form  \eqref{vecuc} with $c(t) = c(0) $ independent of time,  
that is $ u (t,x) = c(0) + \phi_x (t,x) $.

The (EK) system \eqref{EK-complete} 
is Galilean invariant: if
$ (\rho (t,x), \vec u (t,x)) $ solves \eqref{EK-complete} then
$$
\rho_{\vec c} (t,x) := \rho_{\vec c} (t,x + \vec c t) \, , \quad 
\vec u_{\vec c} (t,x) := \vec u (t,x + \vec c t) - \vec c 
$$ 
solve \eqref{EK-complete} as well. Thus, 
regarding the  Euler-Korteweg system in a frame moving with a 
 constant speed $ \vec c (0) $,  
we may always consider  in \eqref{vecuc} that 
$$
\vec u = \nabla \phi \, , \quad \phi : \T^d \to \R \, .
$$ 
The Euler-Korteweg equations 
\eqref{EK-complete} read,  for irrotational fluids, 
\begin{equation}
\label{iniziale}
\begin{cases} \partial_t \rho + {\rm div}(\rho \nabla \phi )=0\\
\partial_t \phi +\frac{1}{2} |\nabla \phi |^2 + g(\rho)= K(\rho)\Delta \rho+ \frac{1}{2} K'(\rho)|\nabla \rho|^2 \, . 
\end{cases} 
\end{equation}
The main result of the paper proves local well posedness for the solutions of 
\eqref{iniziale} with initial data $(\rho_0, \phi_0) $ in Sobolev spaces
$$
H^s (\T^d) := \Big\{
u (x) = \sum_{j\in \Z^d} u_j e^{\im j \cdot x} \ :  \ 
\| u \|_s^2   := \sum_{j \in \Z^d} |u_j|^2 \langle j \rangle^{2s} < + \infty \Big\}  
$$
where $  \langle j \rangle := \max\{1,|j|\} $, 
under the natural mild regularity assumption $ s > 2 + (d/2)$. 
Along the paper, 
$ H^s (\T^d) $ may denote 
either the Sobolev space 
of real valued functions $ H^s (\T^d, \R) $  or the complex valued ones $ H^s (\T^d, \C) $. 

\begin{theorem}{\bf (Local existence on $\T^d$)}
\label{thm:main0} 
Let $s > 2 +  \frac{d}{2} $. 
For any  initial data
$$
(\rho_0, \phi_0) \in H^s (\T^d,\R) \times H^s(\T^d,\R) 
 \quad {\rm with} \quad \rho_0(x) > 0  \, ,
\quad  \forall x \in \T^d \,,
$$
there exists $T:= T(\| (\rho_0, \phi_0)\|_{s_0+2}, \min_x \rho_0(x)) >0 $ and a unique solution $ (\rho, \phi)$ of \eqref{iniziale} 
such that 
\[
(\rho, \phi) \in C^{0}\Big([-T,T], H^{s}(\mathbb{T}^d,\mathbb{R})\times
 H^{s}(\mathbb{T}^d,\mathbb{R})\Big) \cap 
  C^{1}\Big([-T,T], H^{s-2}(\mathbb{T}^d,\mathbb{R})\times
 {H}^{s-2}(\mathbb{T}^d,\mathbb{R})
\Big)\ 
\]
and  $\rho(t,x) > 0 $ for any $t \in [-T, T]$. 
Moreover, for $|t| \leq T$,  
the solution map $ (\rho_0, \phi_0) \mapsto (\rho (t, \cdot), \phi (t, \cdot) ) $
is locally defined and continuous  in $ H^{s}(\mathbb{T}^d,\mathbb{R})\times
 H^{s}(\mathbb{T}^d,\mathbb{R}) $. 
\end{theorem}

We remark that it is sufficient to prove the existence of a solution of
\eqref{iniziale} on $[0,T]$ because system \eqref{iniziale} is reversible: 
the Euler-Korteweg  vector field  $ X $  defined by 
\eqref{iniziale} satisfies 
$ X\circ \cS = - \cS \circ X $, 
where $ \cS $ is the involution 
\begin{equation}\label{rev_invo}
\cS\left( \begin{matrix}
\rho \\ \phi 
\end{matrix} \right) := \left( \begin{matrix}
 \rho^\vee \\ -  \phi^\vee 
\end{matrix} \right) \, , \quad \rho^\vee (x) :=  \rho (-  x) \, .  
\end{equation}
Thus, denoting by $ (\rho, \phi)(t,x) = \Omega^t (\rho_0, \phi_0) $ the solution of \eqref{iniziale}
with initial datum $(\rho_0, \phi_0)$ in the time interval $[0,T]$, we have that 
$ \cS \Omega^{-t} (\cS (\rho_0, \phi_0) ) $  solves  \eqref{iniziale}
with the same initial datum but  in the time interval $[-T,0]$.

Let us make some comments about the phase space of 
system \eqref{iniziale}. Note that the average 
$ {\frac{1}{(2\pi)^d}} \int_{\T^d} \rho(x) \, \di x $ 
is a prime integral of \eqref{iniziale} (conservation of the mass), namely 
\begin{equation}\label{media}
{\frac{1}{(2\pi)^d}} \int_{\T^d} \rho(x) \, \di x = \eq  \, , \quad  \eq \in \R \, ,  
\end{equation}
remains constant along the solutions of \eqref{iniziale}. 
Note also that the vector field of \eqref{iniziale} 
depends only on $  \phi - \frac{1}{(2 \pi)^d}\int_{\T^d} \phi \, \di x  $. 
As a consequence, the variables $ (\rho-\eq, \phi) $ 
belong naturally to some Sobolev space
$ H^s_0(\T^d) \times \dot H^s (\T^d) $,   
where  $H^s_0(\T^d) $ 
denotes the Sobolev space of functions with zero average
$$
H^s_0(\T^d) := \Big\{
u \in H^s(\T^d) \ \colon 
\  
\int_{\T^d} u(x) \di x = 0 \Big\} 
$$
and  $\dot H^s(\T^d)$, $s \in \R$, the corresponding 
homogeneous Sobolev space, namely the quotient space obtained by identifying all 
the $H^s(\T^d)$ functions which differ only by a constant. For simplicity of notation
we  denote the equivalent class $ [u] := \{ u + c,   c \in \R \} $, just by 
$ u $.  
The homogeneous norm of $ u \in \dot H^s (\T^d) $ is 
$ \| u \|_s^2   := \sum_{j \in \Z^d \setminus \{0\}} |u_j|^2 |j |^{2s}$.
We shall denote 
by $ \| \  \|_s $  
either the Sobolev norm in $ H^s $ or that one in the homogenous space $ \dot H^s $, 
according to the context. 

\smallskip

Let us make some comments about the proof. First, in view of \eqref{media}, 
we rewrite system \eqref{iniziale} in terms of $ \rho  \leadsto \eq + \rho $ 
with $ \rho \in H^s_0 (\T^d) $, obtaining  
\begin{equation}
\label{modif0}
\begin{cases} 
\partial_t \rho  = - \eq \Delta \phi - \div(\rho \grad \phi)\\
\partial_t \phi  = -\frac{1}{2} | \grad \phi |^2 - g(\eq+\rho) +  K(\eq+\rho) \Delta \rho+ \frac{1}{2} K'(\eq+\rho)| \grad \rho|^2 \, .
\end{cases} 
\end{equation}
Then Theorem \ref{thm:main0} follows by the following result, that we are going to prove 
\begin{theorem}
\label{thm:main1} 
Let $s > 2 +  \frac{d}{2} $ and $ 0 < \eq_1 < \eq_2 $.  
For any  initial data of the form $(\eq + \rho_0, \phi_0)$ with 
$ (\rho_0, \phi_0) \in H^s_0(\T^d) \times \dot H^s(\T^d) $ and
$ \eq_1 <  \eq + \rho_0(x) < \eq_2 $, $  \forall x \in \T^d $, 
there exists {$ T= T\big(\| (\rho_0, \phi_0)\|_{s_0+2}, \min_x (\eq+ \rho_0(x)) \big) >0    $ } and a unique solution $(\eq + \rho, \phi)$ of \eqref{modif0} 
such that 
\[
(\rho, \phi) \in C^{0}\Big([0,T], H_0^{s}(\mathbb{T}^d,\mathbb{R})\times
  \dot{H}^{s}(\mathbb{T}^d,\mathbb{R})\Big) \cap 
  C^{1}\Big([0,T], H_0^{s-2}(\mathbb{T}^d,\mathbb{R})\times
  \dot{H}^{s-2}(\mathbb{T}^d,\mathbb{R})
\Big)\ 
\]
and  
$ \eq_1 <  \eq + \rho(t,x) < \eq_2 $  holds  for any $t \in [0, T]$. 
Moreover, for $|t| \leq T$,  
the solution map $ (\rho_0, \phi_0) \mapsto (\rho (t, \cdot), \phi (t, \cdot) ) $
is locally defined and continuous  in $H^s_0(\T^d) \times \dot H^s(\T^d) $.
\end{theorem}

We consider system \eqref{modif0} 
as a system on the homogeneous space $ \dot H^s \times \dot H^s $, that is we study
\begin{equation}
\label{modif00}
\begin{cases} 
\partial_t \rho  = - \eq \Delta \phi - \div((\Pi_0^\bot \rho) \grad \phi)\\
\partial_t \phi  = -\frac{1}{2} | \grad \phi |^2 - g(\eq+ \Pi_0^\bot \rho) +  K(\eq+\Pi_0^\bot \rho) \Delta \rho+ \frac{1}{2} K'(\eq+ \Pi_0^\bot \rho)| \grad \rho|^2 \, 
\end{cases} \end{equation}
{where $\Pi_0^\perp$ is the  projector onto the Fourier modes of index $\neq 0$.}
For simplicity of notation we shall not distinguish between systems \eqref{modif00} 
and \eqref{modif0}. 
In Section \ref{sec:3},  we paralinearize \eqref{modif0}, i.e. \eqref{modif00}, up to 
bounded semilinear terms
(for which we do not need Bony paralinearization formula). Then,
introducing  a suitable complex variable, we  transform it into 
a quasi-linear type Schr\"odinger equation, see system \eqref{EKcf1}, 
defined in the phase space
\begin{equation}\label{Hsdot}
\dot \bH^s := \Big\{
U = \vect{u}{\bar u} \colon \quad  u \in \dot H^s(\T^d,\C)
\Big\} \, , \quad \norm{U}_{s}^2 := \norm{U}_{\dot \bH^s}^2 = \norm{u}_s^2 + 
\norm{\bar u}_s^2  \, . 
\end{equation} 
We use paradifferential calculus in the Weyl quantization, because it is
quite convenient to prove energy estimates for this system. 
Since \eqref{EKcf1} is a quasi-linear system, in order 
to prove local well posedness  (Proposition \ref{prop:LWP})
we follow the  strategy, initiated by Kato \cite{Kato},   of constructing
inductively 
a sequence of linear problems 
 whose solutions converge to the solution  of the quasilinear equation.
Such a scheme has been widely used,  see e.g.    \cite{Met,ABZ, BDD, FIloc} and reference therein. 

The equation \eqref{iniziale} is a Hamiltonian PDE. We do not exploit 
explicitly this fact, but it is indeed responsible for the energy estimate {of Proposition 
\ref{exp}}.
The method of proof of Theorem \ref{thm:main0} is similar to the one in 
Feola-Iandoli \cite{FI1} for Hamiltonian quasi-linear Schr\"odinger equations on $ \T^d $
(and Alazard-Burq-Zuily \cite{ABZ} in the case of
gravity-capillary water waves in $ \R^d $). 
The main difference is that we aim to obtain the minimal smoothness assumption
$ s > 2 + (d/2) $. 
This requires to optimize several arguments, and, in particular, to develop
a sharp para-differential calculus for periodic functions that we report 
in the Appendix in a self-contained way. 
Some other technical differences are in the use of the modified energy (section \ref{sec:4.2}), 
the mollifiers \eqref{epsisol} which enables to prove energy estimates independent of 
$ \varepsilon $ for the regularized system, 
the argument for the continuity {of the flow} in $ H^s $. 
We expect that our approach 
would enable to extend the  local existence  result of  \cite{FI1} to initial data  fulfilling the  minimal 
smoothness assumptions $ s > 2 + (d/2) $.

\smallskip

We now set some notation that will be used throughout the paper. 
Since $ K : \R_+ \to \R $ is positive, given $ 0 < \eq_1 < \eq_2 $, 
{there exist
constants}  $c_K , C_K > 0 $ such that 
\begin{equation}
\label{ellip}
 c_K\leq K(\rho)  \leq C_K   \,  , \qquad  \forall \rho  \in (\eq_1, \eq_2)  \, .
\end{equation}
Since the velocity potential $ \phi $ is defined up to a constant, we may assume 
in \eqref{modif0} that
\begin{equation}\label{gmze}
g(\eq) = 0 \,.
\end{equation} 
From now on {\it we fix} $s_0$ so that 
\begin{equation}
\label{s0}
 \frac{d}{2} < s_0 < s-2 \, . 
\end{equation}
The initial datum $\rho_0 (x) $ 
belongs to  the open subset of $  H^{s_0}_0(\T^d)$ defined by 
\begin{equation}
\label{def:Q}
\mathcal{Q}
:= \big\{ \rho \in H^{s_0}_0 (\T^d)\ \  : \ \  \eq_1<\eq+ \rho(x)< \eq_2 \big\} 
\end{equation}
and we shall prove that, locally in time, the solution of \eqref{modif0} stays in this set.

We write
 $a\lesssim b$ with the meaning $a \leq C b$ for some constant $C >0$ which does not depend on relevant quantities.

\section{Functional setting and  paradifferential calculus}

\noindent The Sobolev norms $ \| \ \|_s $ satisfy   interpolation inequalities
(see e.g. section 3.5 in 
\cite{BB}): 
\\[1mm]
(i) for all $ s \geq s_0 > \frac{d}{2} $, $ u, v \in H^s $, 
\begin{equation}
\label{iterp1}
\| u v \|_s \lesssim \| u \|_{s_0} \| v \|_s + \| u \|_{s} \| v \|_{s_0} \, . 
\end{equation}
(ii) For all  $ 0 \leq s \leq s_0 $, $ v \in H^s  $, $ u \in H^{s_0}  $, 
\begin{equation}
\label{interp2}
\| u v \|_s \lesssim  \| u \|_{s_0} \| v \|_s \, . 
\end{equation}
(iii) For all $ s_1 < s_2 $, $ \theta \in [0,1] $ and $ u \in H^{s_2}  $, 
\begin{equation}\label{interp4}
\| u \|_{\theta s_1 + (1- \theta) s_2} \leq \| u \|_{s_1}^\theta \| u \|_{s_2}^{1-\theta} \, .
\end{equation}
(iv) For all  $a\leq \alpha\leq \beta \leq b$, $u,v\in H^b  $,
\begin{equation}\label{interp3}
\| u \|_{\alpha} \| v \|_{\beta} \leq
\| u \|_{a} \| v \|_{b} + \| u \|_{b} \| v \|_{a} \, . 
\end{equation}

\paragraph{Paradifferential calculus.}
We now introduce the notions of paradifferential calculus that 
will be used in the proof of Theorem \ref{thm:main0}.
We develop it in  the Weyl quantization since it is more convenient 
to get  the energy estimates of section \ref{sec:4}. 
The main results are the continuity Theorem \ref{thm:contS}
and the  composition Theorem \ref{thm:compS}, 
which require mild regularity assumptions of the symbols 
in the space variable (they are deduced by the sharper results proved in 
Theorems \ref{thm:cont2} and \ref{thm:comp2} in the Appendix). 
This is needed in order to prove the local existence Theorem \ref{thm:main0} 
with the natural minimal regularity on the initial datum 
$ (\rho_0, \phi_0) \in  H^s \times  H^s $ 
with $ s >  2 + \frac{d}{2} $. 

\smallskip

Along the paper $ \mathscr{W} $ may denote either 
the Banach space  $ L^\infty (\T^d)$,  or the Sobolev spaces $ H^s (\T^d)$, or 
the H\"older spaces  $ W^{\varrho,\infty} (\T^d)$, introduced in  Definition 
\ref{sec:holder}.  {Given a multi-index $ \beta \in \N_0^d $ we define
$ |\beta| := \beta_1 + \ldots + \beta_d $.}

\begin{definition}{\bf (Symbols with finite regularity)}\label{def:sfr}
Given $ m  \in \R $ and a Banach space $\mathscr{W}
\in \{   L^\infty (\T^d), H^s (\T^d), W^{\varrho,\infty} (\T^d)\} $, 
we denote by $\Gamma^m_\mathscr{W}$ the space of 
functions $ a : \T^d \times \R^d \to \C $,  $a(x, \xi)$, 
which are $C^\infty$ with respect to $\xi$ and such that, for any  $ \beta \in \N_0^d $, 
there exists a constant $C_\beta >0$ such that
\begin{equation}\label{simb-pro}
\big\| \pa_\xi^\beta \, a(\cdot, \xi) \big\|_\mathscr{W} \leq C_\beta \, \la \xi \ra^{m - |\beta|}  , \quad \forall \xi \in \R^d \,  . 
\end{equation}
We denote by $\Sigma^m_\mathscr{W}$  the subclass of symbols $a\in \Gamma^m_\mathscr{W}$ which are {\em spectrally localized}, that is 
\begin{equation}\label{spectr-loca}
\exists \,  \delta \in (0, 1) \, \colon \qquad \quad 
\wh a(j, \xi)  = 0 \, , \quad \forall |j| \geq \delta \la \xi \ra  \, , 
\end{equation}
where {$\wh a(j, \xi) := (2 \pi)^{-d} \int_{\T^d} a(x,\xi) e^{- \ii j \cdot x} \di x$, $ j \in \Z^d $, are the Fourier coefficients of the function $x \mapsto a(x, \xi)$}.

We endow $\Gamma^m_\mathscr{W}$ with the  family of norms defined, for any $n \in \N_0$, by
\begin{equation}
\label{seminorm}
\abs{a}_{m, \mathscr{W}, n}:=  \max_{|\beta| \leq n}\, 
\sup_{\xi \in \R^d} 
\, \big\| \la \xi \ra^{-m+|\beta|} \, \pa_\xi^{\beta} a(\cdot,  \xi) \big\|_\mathscr{W} \ . 
\end{equation}
When $\mathscr{W} = H^s $, we  also denote  $\Gamma^m_s \equiv \Gamma^m_{H^s}$ and  $\abs{a}_{m, s, n} \equiv \abs{a}_{m,H^s, n}$. 
We denote by $ \Gamma_s^m \otimes {\cal M}_2 (\C) $ the $ 2 \times 2 $ matrices 
$ A = \begin{pmatrix} a_{1} & a_{2} \\
a_{3} & a_{4}
\end{pmatrix}
$
of symbols  in $ \Gamma_s^m $ and $ | A |_{m, \mathscr{W}, n} := 
 \max_{i=1, \ldots, 4}\{ | a_{i} |_{m, \mathscr{W}, n}\} $. 
 Similarly we denote by 
$ \Gamma_s^m \otimes \R^d $ the $d$-dimensional vectors 
of symbols in $ \Gamma_s^m $. 
\end{definition}
Let us make some simple remarks: 
\\[1mm]
$ \bullet $ {($i$)} given a  function $ a(x)  \in \mathscr{W} $ then $a(x)  \in \Gamma^0_{\mathscr{W}} $ and
\begin{equation} \label{rem:symb}
\abs{u}_{0, \mathscr{W}, n} = \norm{u}_{\mathscr{W}} \, ,  \forall  n \in \N_0 \, . 
\end{equation}
$ \bullet $ {($ii$)} For any 
$s_0 > \frac{d}{2}$ and $0 \leq \vr' \leq \vr $, we have that 
\begin{equation}\label{simba}
\abs{a}_{m, L^\infty, n} 
\lesssim \abs{a}_{m, W^{\vr', \infty}, n}
\lesssim \abs{a}_{m, \hol, n} 
\lesssim \abs{a}_{m, H^{s_0+\vr },n}  \, , \quad \forall n \in \N_0 \, . 
\end{equation}
 $ \bullet $ {($iii$)} 
 If $ a \in \Gamma^m_{\mathscr{W}}$, 
then, for any $ \alpha \in \N_0^d $, 
we have $ \pa_\xi^\alpha a \in \Gamma^{m-|\alpha|}_{\mathscr{W}} $ and
\begin{equation}\label{sim1}
| \pa_\xi^\alpha a |_{m-|\alpha|,\mathscr{W},n} \lesssim 
| a|_{m,\mathscr{W},n+|\alpha|} \, , \quad \forall n \in \N_0 \, . 
\end{equation}
$ \bullet $ {($iv$)}  If  $ a \in \Gamma^m_{H^{s}} $, resp. 
$ a \in \Gamma^m_{W^{\vr, \infty}} $, then 
$ \pa_x^\alpha a \in \Gamma^m_{H^{s-|\alpha|}}  $, resp. 
$ \pa_x^\alpha a \in \Gamma^m_{W^{\vr-|\alpha|,\infty}}  $, and
\begin{equation}\label{sim2}
| \pa_x^\alpha a|_{m,s-|\alpha|,n} \lesssim 
|  a|_{m,s,n} \, , \quad {\rm resp.} \ 
|\pa_x^\alpha a|_{m,W^{\vr-|\alpha|,\infty},n} \lesssim |a|_{m,W^{\vr,\infty},n}  \, , \quad \forall n \in \N_0 \, . 
\end{equation}
$ \bullet $ {($v$)} If 
$ a,b \in \Gamma^m_{\mathscr{W}} $ then $ ab \in \Gamma^m_{\mathscr{W}} $
with $ |ab|_{m+m',\mathscr{W},n} \lesssim 
|a|_{m,\mathscr{W},N} |b|_{m',\mathscr{W},n}$ for any $ n \in \N_0 $.
In particular, if 
$ a,b \in \Gamma^m_{s} $ with $ s > d / 2 $ then 
$ ab \in \Gamma^{m+m'}_{s} $ and 
\begin{equation}\label{sim3}
| a b |_{m+m',s,n} \lesssim
| a  |_{m,s,n} | b  |_{m',s_0,n} + | a  |_{m,s_0,n} | b  |_{m',s,n}  \, , \quad \forall n \in \N_0 \, . 
\end{equation}
Let $\epsilon \in (0,1)$ and consider a 
$ C^\infty $, even cut-off  function $\chi\colon \R^d \to [0,1]$ such that
\begin{equation}\label{def-chi}
\chi(\xi) =  
\begin{cases}
1 & \mbox{ if } |\xi| \leq 1.1 \\
0 & \mbox{ if } |\xi| \geq 1.9 \, , 
\end{cases}  
\qquad \chi_\epsilon(\xi) := \chi \left( \frac{\xi}{\epsilon}\right) \, . 
\end{equation}
Given a symbol $ a $ in $ \Gamma^m_{\mathscr{W}} $
we define  the {\it regularized} symbol
\begin{equation}\label{achi}
a_\chi (x, \xi) := \chi_{\epsilon \la \xi \ra}(D) a(x, \xi) = \sum_{j \in \Z^d}
 \chi_\epsilon \Big( \frac{j}{\langle \xi \rangle} \Big) \, \widehat a (j, \xi) \, e^{\im j \cdot x} \, .
\end{equation}
Note that $a_\chi$ is  analytic in $ x $  (it is a trigonometric polynomial)
 and it  is spectrally localized. 

In order to define the Bony-Weyl quantization of a symbol $ a (x, \xi )$
we first remind the Weyl quantization formula 
\begin{equation}
\label{opW} 
{\rm Op}^W {(a)}[u]  := \sum_{j \in \Z^d} \Big( \sum_{k \in \Z^d}
\widehat a \Big(j-k, \frac{k + j}{2}\Big)  \, u_k \Big) e^{\im j \cdot x } \, .
\end{equation}

\begin{definition} {\bf (Bony-Weyl quantization)}
Given a symbol $a \in  \Gamma^m_{\mathscr{W}}$, we define the 
{\em Bony-Weyl paradifferential operator}
$ {\rm Op}^{BW} {(a)} =  {\rm Op}^{W} (a_\chi) $ that acts on a periodic function $ u $
as   
\begin{equation}
\begin{aligned}
\label{BW}
\left({\rm Op}^{BW} {(a)}[u] \right)(x) 
& := \sum_{j \in \Z^d} \Big( \sum_{k \in \Z^d}
\widehat a_\chi \Big(j-k, \frac{j+k}{2}\Big) 
\, u_k \Big) e^{\im j \cdot x } \\
&= 
\sum_{j \in \Z^d} \Big( \sum_{k \in \Z^d}
\widehat a \Big(j-k, \frac{j+k}{2}\Big) 
\, \chi_\epsilon\Big( \frac{j-k}{\langle j+k \rangle}\Big)
\, u_k \Big) e^{\im j \cdot x } \, . 
\end{aligned}
\end{equation}
{If $ A = \begin{pmatrix} a_{1} & a_{2} \\
a_{3} & a_{4}
\end{pmatrix} $ is a matrix of symbols in $  \Gamma_s^m $, then
$ {\rm Op}^{BW} (A) $ is defined as the matrix valued operator
$ \begin{pmatrix} {\rm Op}^{BW} (a_{1}) & {\rm Op}^{BW} (a_{2}) \\
{\rm Op}^{BW} (a_{3}) & {\rm Op}^{BW} (a_{4})
\end{pmatrix} $. 
}
\end{definition}
Given a symbol $ a(\xi) $ independent of $ x $, then
$ {\rm Op}^{BW} {(a)} $ is the Fourier multiplier operator 
$$
{\rm Op}^{BW} {(a)} u = a(D) u = \sum_{j \in \Z^d} 
a(j) \, u_j \, e^{\im j \cdot x }  \, . 
$$
Note that if 
$ \chi_\epsilon \Big( \frac{k-j}{\langle k + j \rangle}\Big) \neq 0 $
then $ |k-j| \leq \epsilon \langle j + k \rangle  $
and therefore, for $ \epsilon \in (0,1)$, 
\begin{equation}\label{rela:para}
\frac{1-\epsilon}{1+\epsilon} |k| \leq 
|j| \leq \frac{1+\epsilon}{1-\epsilon}|k| \, , \quad \forall j, k \in \Z^d \, . 
\end{equation}
This relation shows that the 
action of a para-differential operator does not spread much the Fourier support of functions.
In particular  
$ {\rm Op}^{\rm BW}(a)  $ sends a constant function into a constant function and therefore 
$ {\rm Op}^{\rm BW}(a) $ sends homogenous spaces into homogenous spaces. 

\begin{remark}
Actually, 
if $ \chi_\epsilon \big( \frac{k-j}{\langle k + j \rangle}\big) \neq 0 $, 
$ \epsilon \in (0,1/4) $,  then 
$  |j| \leq |j+k| \leq 3 |j| $, for all $ j,k \in \Z^d $. 
\end{remark}

Along the paper 
we shall use the following  results concerning the action of a paradifferential operator 
in Sobolev spaces.

\begin{theorem}{\bf (Continuity of Bony-Weyl operators)}
\label{thm:contS}
Let $ a \in \Gamma^m_{s_0} $, resp.  $ a \in \Gamma^m_{L^\infty} $, with  $ m \in \R $.
Then $\Opbw{a}$ extends to a bounded operator 
$\dot H^{s} \to \dot H^{s-m}$ for any $ s \in \R $  satisfying the estimate, for any $ u \in \dot H^s $,  
\begin{align} \label{cont00}
& \norm{\Opbw{a}u}_{{s-m}} \lesssim \, \abs{a}_{m, s_0, 2(d+1)} \, \norm{u}_{{s}}
\end{align}
Moreover, for any $\vr \geq 0$, $ s \in \R $, $ u \in \dot H^s (\T^d)$, 
\begin{equation}
\label{cont2}
\norm{\Opbw{a}u}_{{s-m- \vr}} \lesssim  \, \abs{a}_{m, {s_0-\vr}, 2(d+1)} \, \norm{u}_{{s}} \, .
\end{equation}

\end{theorem}
\begin{proof}
Since $ \Opbw{a} = \Opw{a_\chi} $,
the estimate \eqref{cont00}  follows
 by \eqref{cont0}, \eqref{achi.est} and 
 $ \abs{a}_{m, L^{\infty}, N} \lesssim \abs{a}_{m, {s_0}, N}  $.  Note that the condition on the Fourier support of $a_\chi$ in Theorem \ref{thm:cont2} is automatically satisfied provided $\epsilon$ in \eqref{def-chi} is sufficiently small. 
To prove \eqref{cont2} we use also \eqref{simbolocutrec}.
\end{proof}

The second result of symbolic calculus that we shall use
regards composition  for Bony-Weyl paradifferential operators
at the second order (as required in the paper) with 
mild smoothness assumptions for the symbols in the space variable $ x $. 
Given symbols
$a \in \Gamma^m_{s_0+\vr}$, $b \in \Gamma^{m'}_{s_0+\vr}$ with $m, m' \in \R$
and  $\vr  \in (0,2]$
we define 
\begin{equation}\label{regolarized-simbol}
a\#_\vr b :=
\begin{cases}
ab \, , & \vr \in (0,1] \\
ab + \frac{1}{2\im}\{a, b\} \, , & \vr \in (1,2] \,  , \quad 
{\rm where} \quad 
\{a,b\} :=  \nabla_{\xi} a \cdot \nabla_{x} b - \nabla_x a \cdot \nabla_\xi b  \, ,  
\end{cases} 
\end{equation}
is the Poisson bracket between  $ a (x, \xi)$ and $ b(x, \xi ) $. 
By  \eqref{sim1} and \eqref{sim3} we have that 
$ab $ is a symbol in $  \Gamma^{m+m'}_{s_0+\varrho} $ and 
$ \{a, b\} $ is in $  \Gamma^{m+m'-1}_{s_0+\varrho-1} $. 
 The next result
follows directly by  Theorem \ref{thm:comp2} and \eqref{simba}. 

\begin{theorem}{\bf (Composition)}
\label{thm:compS}
Let $a \in \Gamma^m_{s_0+\vr}$, $b \in \Gamma^{m'}_{s_0+\vr}$ with $m, m' \in \R$ and $\vr  \in (0,2]$.  Then  
\begin{align}
\label{comp01}
\Opbw{a}\Opbw{b} 
& =  \Opbw{a\#_\vr b} + R^{-\vr}(a,b)
\end{align}
where the linear operator $R^{-\vr}(a,b)\colon \dot H^s \to \dot H^{s-(m+m')+\vr}$, 
$\forall s \in \R$,  satisfies, for any $ u \in \dot H^s $, 
\begin{align}
\label{compS}
&\norm{R^{-\vr}(a,b)u}_{{s -(m+m') +\vr}} \lesssim \left(\abs{a}_{m, {s_0+\vr}, N} \, \abs{b}_{m', {s_0}, N} + \abs{a}_{m, {s_0}, N} \, \abs{b}_{m', {s_0+\vr}, N}  \right) \norm{u}_{s}
\end{align}
where $ N \geq 3d + 4 $.
\end{theorem}

A useful corollary of Theorems \ref{thm:compS} and \ref{thm:contS} 
(using also \eqref{sim1}-\eqref{sim3}) is the following:
\begin{corollary}\label{symweyl}
Let 
$a\in \Gamma^{m}_{s_0+2}$,  $b\in \Gamma^{m'}_{s_0+2}$, 
$c \in \Gamma^{m''}_{s_0+2}$ with $m, m', m'' \in \R$.  Then  
\begin{equation}\label{autoag}
\Opbw{a}\circ \Opbw{b}\circ \Opbw{ c}= \Opbw{abc} + R_1(a,b,c)+ R_0(a,b,c),  
\end{equation}
where
\begin{equation}\label{R1abc}
R_1(a,b,c):= 
{\rm Op}^{BW} \big( \{ a,c\}b+ \{b,c\}a+\{ a,b\}c \big) 
\end{equation}
satisfies 
$ R_1(a,b,c)=-R_1(c,b,a) $ and $R_0(a,b,c)$  is a bounded operator  $\dot H^s \to \dot H^{s-(m+m'+m'')+2}$, $\, \forall s\in \R$, satisfying, for any 
$ u \in \dot H^s  $,
\begin{equation}\label{restoR0}
\|R_0(a,b,c)\|_{s-(m+m'+m'') + 2} \lesssim  |a|_{m,s_0+2,N} \ 
|b|_{m',s_0+2,N}\
|c|_{m'',s_0+2,N} \  \| u \|_s
\end{equation}
where $ N \geq 3d+ 5 $.
\end{corollary}

We now  provide the Bony-paraproduct decomposition for the product of Sobolev functions in the Bony-Weyl quantization. Recall that $ \Pi_0^\bot $ denotes the projector on 
the subspace $ H^s_0 $. 

\begin{lemma}{\bf (Bony paraproduct decomposition)}
\label{bony}
Let  $u  \in  H^s$, $v \in H^r$  with  
$s + r \geq  0$. Then 
\begin{equation}\label{bonyeq}
uv = \Opbw{u}v + \Opbw{v} u + R(u,v)
\end{equation}
where the bilinear operator  $R\colon  H^s \times  H^r \to  H^{s+r-s_0}$ is symmetric and satisfies 
the estimate 
\begin{equation}
 \label{bonyeq2}
\norm{R(u,v)}_{s+ r - s_0} \lesssim \norm{u}_s \, \norm{v}_{r} \, . \\
\end{equation}
Moreover
$ R(u,v) = R(\Pi_0^\bot u, \Pi_0^\bot v) - u_0 v_0 $ and then
\begin{equation} \label{bonyeq2o} 
\| \Pi_0^\bot R(u,v) \|_{s+ r - s_0} \lesssim 
\| \Pi_0^\bot u \|_s \, \| \Pi_0^\bot v \|_{r} \, . 
\end{equation}
\end{lemma}
\begin{proof}
Introduce the  function $\theta_\epsilon(j,k)$ by 
\begin{equation}
\label{partz}
1 =\chi_\epsilon\Big( \frac{j-k}{\langle j+k \rangle}\Big) + 
\chi_\epsilon\Big( \frac{k}{\langle 2j-k \rangle}\Big) + 
\theta_\epsilon(j,k) \, . 
\end{equation}
Note that $\abs{\theta_\epsilon(j, k )} \leq 1$. 
Let $ \Sigma := \{ (j, k) \in \Z^d \times \Z^d \, : \,  \theta_\epsilon (j,k) \neq 0  \} $
 denote the support 
of $\theta_\epsilon $. 
  We claim that  
\begin{equation}
\label{theta.supp}
(j,k) \in \Sigma \qquad \Longrightarrow \qquad 
|j| \leq  C_\epsilon \min(  |j-k| , \  |k |)  \, .
\end{equation}
Indeed, recalling the definition of the cut-off function $ \chi $ in \eqref{def-chi},
we first note  that\footnote{For $\delta$ sufficiently small, 
if $ |j-k| \leq \delta  \la j + k \ra $ and  
$ |k| \leq \delta  \la 2j - k  \ra  $ then $ (j,k) = (0,0)$. }  
$$
\Sigma = \{(0,0)\} \cup  
\Big\{  |j-k| \geq \epsilon \la j + k \ra \, , \ |k| \geq \epsilon \la 2j - k \ra \Big\} \, . 
$$
Thus, for any $ (j,k) \in \Sigma $,  
\begin{align*}
&|j| \leq \frac12 |j-k| + \frac12 |j+k| \leq \left(\frac{1}{2} + \frac{1}{2\epsilon} \right) \, |j-k|  \, ,
\quad  |j| \leq \frac12 |2j - k| + \frac12 |k| \leq 
\left(\frac{1}{2} + \frac{1}{2\epsilon} \right) \, |k| 
\end{align*}
proving \eqref{theta.supp}.
Using \eqref{partz} we decompose 
\begin{align*}
uv 
& = \sum_{j,k } \wh u_{j-k} \, \chi_\epsilon\Big( \frac{j-k}{\langle j+k \rangle}\Big)  \, \wh v_{k} \, e^{\im j \cdot x}
 +
  \sum_{j,k}  \wh v_{k} \, \chi_\epsilon\Big( \frac{k}{\langle 2j-k \rangle}\Big)\, \wh u_{j-k} \,  e^{\im j \cdot x} + \sum_{j, k}   \theta_\epsilon (j, k)\wh u_{j-k}\,  \whv_{k}   \,  e^{\im j \cdot x} \\
& =  \Opbw{u} v + \Opbw{v} u + R(u,v) \, . 
\end{align*}
By   \eqref{theta.supp},    $ s+ r \geq 0 $, and 
 the Cauchy-Schwartz inequality, we get 
\begin{align*}
\norm{R(u,v)}_{s+r-s_0}^2 & \leq 
\sum_j \la j \ra^{2(s+r-s_0)} \Big| \sum_{k}  \theta_\epsilon (j, k)\wh u_{j-k}\,  \whv_{k} \Big|^2 \\
&  \lesssim \sum_j \la j \ra^{-2s_0} \Big| \sum_{k}  \la j-k \ra^{s} \, |\wh u_{j-k}|\, \la k \ra^{r }\,  |\whv_{k}| \Big|^2 
 \lesssim \norm{u}_s^2 \, \norm{v}_r^2
\end{align*}
 proving \eqref{bonyeq2}. Finally, since on the support of $ \theta_\epsilon $
 we have or $ (j,k) = (0,0)$ or $ j - k \neq 0 $ and $ k \neq 0 $, we deduce that  
$$ 
R(u,v) = 
 \theta_\epsilon (0,0)\wh u_{0}\,  \whv_{0}  + 
 \sum_{j-k \neq 0, k \neq 0}   \theta_\epsilon (j, k)\wh u_{j-k}\,  \whv_{k}   \,  e^{\im j \cdot x}
 = - \wh u_{0}\,  \whv_{0}  + R(\Pi_0^\bot u,\Pi_0^\bot v) 
$$
and we deduce \eqref{bonyeq2o}.
\end{proof}

\paragraph{Composition estimates.}

We will use the following Moser estimates for composition of functions 
in Sobolev spaces.

\begin{theorem}\label{thm:moser}
Let $I \subseteq \R$ be an open interval and  $F\in C^\infty(I; \C)$ a smooth function.
Let $J \subset I$ be a compact interval.  
For any function $u, v \in H^s(\T^d, \R)$, $s>\frac{d}{2} $,  
with values in $ J $,  we have
\begin{equation}\label{moser}
\begin{aligned}
& \| F(u)\|_{s} \leq C({s,F,J}) \left( 1+ \| u \|_{s}\right) \, , \\
& \| F(u)- F(v)\|_{s} \leq C({s,F,J})  \left(
\| u-v \|_{s} + (\| u \|_s+ \| v \|_s)\| u - v \|_{L^\infty}\right)\, \\
& \| F(u)\|_{s} \leq C({s,F,J})   \| u \|_{s}  \quad {\rm if} \quad  F(0) = 0  \, . 
\end{aligned}
\end{equation}
\end{theorem}

\begin{proof}
Take an extension $\tilde F\in C^\infty(\R;\C)$ such that $\tilde F_{| I}= F $. 
Then $  F(u) = \tilde F(u)$ for any $u  \in H^s(\T^d;\R)$ with values in $ J $, 
and apply 
the usual Moser estimate,  
see e.g. \cite{AG},  replacing the   Littlewood-Paley decomposition on $\R^d$ with the one   on  $\T^d$ in  \eqref{LP}.
\end{proof}

\section{Paralinearization of (EK)-system and complex form}\label{sec:3}

In this section we paralinearize  the Euler-Korteweg system  \eqref{modif0} and write it in terms of the complex  variable
\begin{equation}\label{defu}
u := \frac{1}{\sqrt 2}  \left(\frac{\eq}{K(\eq)}\right)^{-1/4} \rho + \frac{\im}{\sqrt 2}\, \left(\frac{\eq}{K(\eq)}\right)^{1/4} \phi \, , \quad \rho \in \dot H^s \, , \ \phi \in \dot H^s \, . 
\end{equation}
The variable $ u \in \dot H^s $. 
We denote this change of coordinates in $ \dot H^s \times \dot H^s $ by 
\begin{equation}\label{scala}\begin{aligned}
& \vect{u}{\bar u}= \boldsymbol{C}^{-1} \vect{\rho}{\phi},\\ & \boldsymbol{C}:=\frac{1}{\sqrt 2} \begin{pmatrix} \left(\frac{\eq}{K(\eq)}\right)^{\frac14}&\left(\frac{\eq}{K(\eq)}\right)^{\frac14}\\ -\im \left(\frac{\eq}{K(\eq)}\right)^{-\frac14}&\im \left(\frac{\eq}{K(\eq)}\right)^{-\frac14}\end{pmatrix},\quad \boldsymbol{C}^{-1}=\frac{1}{\sqrt 2} \begin{pmatrix} \left(\frac{\eq}{K(\eq)}\right)^{-\frac14}&\im\left(\frac{\eq}{K(\eq)}\right)^{\frac14}\\ \left(\frac{\eq}{K(\eq)}\right)^{-\frac14}&-\im \left(\frac{\eq}{K(\eq)}\right)^{\frac14}\end{pmatrix} \, . 
\end{aligned}
\end{equation}
We also define the matrices
\begin{equation}
\label{def:E}
J := 
 \begin{bmatrix}
0& 1 \\
-1 & 0
\end{bmatrix}, \qquad 
\J:=  \begin{bmatrix}
-\im & 0 \\
0 & \im
\end{bmatrix} , \qquad 
\uno :=  \begin{bmatrix}
1& 0 \\
0 & 1
\end{bmatrix} . 
\end{equation}

 \begin{proposition}{\bf (Paralinearized Euler-Korteweg equations in  complex coordinates)}
The (EK)-system  \eqref{modif0} can be 
written in terms of the complex variable 
$ U := \vect{u}{\bar u} $ with $ u $ defined in \eqref{defu},  
in the paralinearized form 
\begin{equation}
\label{EKcf1}
\begin{aligned}
\pa_t U  & =  \ 
\J \, \Big[ \Opbw{
A_2(U; x, \xi)  + A_1(U; x, \xi)  }    \Big] U+ R(U)
\end{aligned}
\end{equation} 
where, for any function $U\in \dot {\bf H}^{s_0+2} $ such that 
\begin{equation}\label{defrho}
\rho (U) := \frac{1}{\sqrt 2} \left(\frac{\eq}{K(\eq)}\right)^{1/4} \Pi_0^\bot (u + \bar u) \in \cQ
\ \  {\rm (see \ \eqref{def:Q})}  \, ,  
\end{equation}
(i) 
$A_2(U;x,\xi) \in \Gamma_{s_0+2}^2 \otimes \cM_2(\C) $ is  
the matrix of symbols  
\begin{equation}
\label{A2}
A_2(U; x, \xi) := 
\sqrt{\eq K(\eq)} |\xi|^2
 \begin{bmatrix}
1 + \ta_+(U; x) & \ta_-(U; x) \\
\ta_-(U; x) & 1+ \ta_+(U; x)
\end{bmatrix}
\end{equation}
where  $\ta_\pm(U; x) \in \Gamma_{s_0+2}^0$ are the 
$\xi$-independent functions 
\begin{equation}
\label{tatb}
\begin{aligned}
&\ta_\pm(U; x) := \frac{1}{2} \left( \frac{K(\rho + \eq)- K(\eq) }{K(\eq)} \pm  \frac{\rho}{\eq} \right) \, . 
\end{aligned}
\end{equation}
\\[1mm]
(ii) 
$A_1(U; x, \xi) \in \Gamma_{s_0+1}^1 \otimes \cM_2(\C) $ is
the diagonal matrix of symbols 
\begin{equation}
\label{A1}
A_1(U; x, \xi) := 
\begin{bmatrix}
\tb(U; x) \cdot \xi  & 0 \\
0 & -\tb(U; x)   \cdot \xi 
\end{bmatrix} ,\qquad 
\tb(U; x) :=  \grad \phi \in  \Gamma_{s_0+1}^0\otimes \R^d \, .
\end{equation}
{Moreover for any $\sigma\geq 0 $ there exists a non decreasing function
$\tC ( \ ) : \R_+ 
\to \R_+ $ (depending on $K$)  such that, for any 
$U,V\in \dot {\bf H}^{s_0} $ with 
$\rho(U), \rho (V) \in \cQ $, $W \in \dot \bH^{\s+2}$ and $j=1,2$, we have 
\begin{align}
& \| \Opbw{ A_j(U)}W\|_{\s} \leq \tC\left(\| U\|_{s_0}\right)\| W\|_{\s+2}  \label{ba1}\\
& \| \Opbw{ A_j(U)-A_j(V)} W\|_{\s}\le \tC\left(\|U\|_{s_0}, \| V\|_{s_0}\right) \| W\|_{\s+2} \| U-V\|_{s_0} \label{la1}
\end{align}}
where in \eqref{la1}  we denoted by $\tC(\cdot, \cdot):= \tC\left(\max\{\cdot,\cdot\}\right)$.
\\[1mm]
(iii) 
The vector field $R(U)$ satisfies the following ``semilinear'' estimates: 
for any $\sigma \geq s_0 > d / 2 $ there exists a non decreasing function
$\tC ( \ ) : \R_+ 
\to \R_+ $ (depending also  on $g,K$)  such that, for any 
$U,V\in \dot {\bf H}^{\sigma+2} $ such that 
$\rho(U), \rho (V) \in \cQ $, we have 
\begin{align}
\label{restoparalin0}
& \| R(U) \|_{ \sigma}  \leq  \tC \left(\| U\|_{ {s_0+2}} \right)\| U\|_{ \sigma} , \qquad 
 \| R(U)\|_{ {\sigma}} \leq \tC \left(\| U\|_{ {s_0}}\right) \| U\|_{ {\sigma+2}}\, , \\
\label{restoparalin1}
& \|R(U)-R(V) \|_{ \sigma}  \leq \tC \left(  \|U\|_{ {s_0+2}}, \, \|V\|_{ {s_0+2}}\right) \|U-V\|_{{\sigma}} 
+ \tC \left(  \|U\|_{ {\sigma}}, \, \|V\|_{ {\sigma}}\right) \|U-V\|_{{s_0+2}} \\
\label{restoparalin2}
& \|R(U)-R(V) \|_{ s_0}  \leq \tC \left(  \|U\|_{ {s_0+2}}, \, \|V\|_{ {s_0+2}}\right) \|U-V\|_{{s_0}} \,,
\end{align}
where in \eqref{restoparalin1} and \eqref{restoparalin2} we denoted again by $\tC(\cdot, \cdot):= \tC\left(\max\{\cdot,\cdot\}\right)$.
 \end{proposition}

\begin{proof} 
We first  paralinearize the original equations \eqref{modif0}, then we switch to complex coordinates. 
\\[1mm]
{\bf Step 1: paralinearization of \eqref{modif0}.}
We apply several times  the paraproduct Lemma \ref{bony} and 
the composition Theorem \ref{thm:compS}. 
In the following we  denote by 
$R^p$  the remainder that comes from Lemma \ref{bony}, and 
 by $R^{-\vr}$, $\vr=1,2$,  the remainder that comes from Theorem \ref{thm:compS}.  
We shall adopt the following convention:  given  $\R^d$-valued  
symbols $ a = (a_j)_{j=1,\ldots, d} $, $ b = (b_j)_{j=1,\ldots, d} $
in some class $\Gamma^m_s\otimes \R^d$,  we denote 
$ R^p(a , b) := \sum_{j=1}^d R^p(a_j, b_j)  $, 
\begin{align*}
R^{-\vr}(a, b) := \sum_{j=1}^d R^{-\vr}(a_j, b_j) \quad {\rm and} \quad 
\Opbw{a}\cdot \Opbw{b} := \sum_{j=1}^d \Opbw{a_j} \Opbw{b_j} \, . 
\end{align*}
We paralinearize the terms in the first line of \eqref{modif0}. 
We  have $  \Delta \phi = - \Opbw{ |\xi |^2} \phi $
and 
$ \div (\rho \nabla \phi) = \nabla \rho \cdot \nabla \phi + \rho \Delta \phi $ can be written as
\begin{align}
 \notag
\rho \Delta  \phi & = -\Opbw{\rho |\xi|^2 + \grad \rho \cdot \im \xi }\phi \\
& \quad + \Opbw{\Delta\phi}\rho + R^p(\rho, \Delta \phi)+ R^{-2}(\rho, |\xi|^2)\phi \, ,\label{paraline1}  \\
\notag
\grad \rho \cdot \grad \phi & = \Opbw{\grad \rho \cdot \im \xi }\phi + \Opbw{\grad \phi \cdot \im \xi }\rho\\
& \quad + R^p(\nabla \rho, \nabla \phi) + R^{-1}(\nabla \rho, \im\xi)\phi+R^{-1}(\nabla \phi, \im\xi)\rho\label{paraline2} \, .  
\end{align}
Then we paralinearize the terms in the second line of \eqref{modif0}. 
We have
\begin{align}
\frac12  | \grad \phi |^2  &= \Opbw{\grad \phi\cdot  \im \xi } \phi \notag \\
& \quad + \frac12 R^p(\nabla \phi,\nabla \phi)+R^{-1}(\nabla \phi, \im\xi)\phi\label{paraline3} \, . 
\end{align}
Using \eqref{gmze} we regard the semilinear term  
\begin{equation}\label{paraline4}
g(\eq + \rho)  = g(\eq + \rho)  - g(\eq) =:  R(\rho) 
\end{equation}
directly as a remainder.  Moreover, writing $  \Delta \rho = - \Opbw{ |\xi |^2} \rho $, we get 
\begin{align}
K(\eq + \rho)\Delta \rho  & = \Opbw{K(\eq+\rho)}\Delta \rho+\Opbw{\Delta \rho} K(\eq+\rho)+R^p(\Delta \rho, K(\eq +\rho)) \nonumber \\
& = - \Opbw{K(\eq+\rho)|\xi|^2 + K'(\eq+\rho) \grad \rho \cdot \im \xi }\rho \nonumber \\
&\quad +\Opbw{\Delta \rho} K(\eq+\rho)+ R^p(\Delta \rho, K(\eq +\rho)) -  R^{-2}(K(\eq+\rho),|\xi|^2)\rho \, .\label{paraline5}
\end{align}
Finally, using for $ \frac12  | \grad \rho |^2  $ the expansion \eqref{paraline3} for $\rho $ instead of $ \phi $, we obtain 
\begin{align}
\notag
\frac12 K'(\eq + \rho)| \grad \rho | ^2  
&  = \frac12 \Opbw{K'(\eq+\rho)}|\nabla \rho|^2+ \frac12 \Opbw{|\nabla \rho|^2}K'(\eq+\rho) 
\nonumber  \\
& \quad + \frac12 R^p(|\nabla  \rho|^2, K'(\eq +\rho)) 
= \Opbw{K'(\eq  + \rho) 	\grad \rho \cdot \, \im \xi } \rho  + {\mathtt R}(\rho) \nonumber 
\end{align}
where
 \begin{align}
\label{bfgr}  {\mathtt R} (\rho) & := \frac12\Opbw{|\nabla \rho|^2}K'(\eq+\rho)+
\frac12 R^p(|\nabla  \rho|^2, K'(\eq +\rho))\\
 &\label{bfgr2} \quad + \frac12 \Opbw{K'(\eq+\rho)}R^p(\nabla \rho, \nabla \rho)\\
 &\label{bfgr3} \quad 
 +\Opbw{K'(\eq+\rho)} R^{-1}(\nabla \rho, \im \xi ) \rho +  R^{-1}(K'(\eq+\rho), \im \nabla \rho\cdot \xi)\rho \, . 
  \end{align}
Collecting all the above  expansions and recalling the definition of 
the symplectic matrix $ J $ in  \eqref{def:E},  the system \eqref{modif0} can be written in the paralinearized form
\begin{equation}
\label{EKpara}
\begin{aligned}
\pa_t \vect{\rho}{\phi} & =  
J \Opbw{
\begin{bmatrix} K(\eq+\rho)|\xi|^2  &  \grad \phi \cdot \,\im \xi \\ 
-\grad \phi\cdot  \,\im \xi  &  (\eq+\rho)|\xi|^2
\end{bmatrix}}\vect{\rho}{\phi} +  R(\rho,\phi) 
 \end{aligned}
 \end{equation}
 where  we collected   in  $R(\rho, \phi)$ all the terms in lines \eqref{paraline1}--\eqref{bfgr3}.
 \\[1mm]
 {\bf Step 2: complex coordinates.} 
 We now write system \eqref{EKpara} in the 
 complex coordinates
  $U = \boldsymbol{C}^{-1}\vect{\rho}{\phi}$. Note that $ \boldsymbol{C}^{-1}$ 
conjugates the Poisson tensor  $ J $ to
$\J$ defined in \eqref{def:E}, i.e. 
$ \boldsymbol{C}^{-1} \, J  =  \J  \, \boldsymbol{C}^* $ and 
therefore system
 \eqref{EKpara} is conjugated to
  \begin{equation}
\begin{aligned}
\label{EKcf10}
\pa_t U& =
\J  
\boldsymbol{C}^{*}
\Opbw{
\begin{bmatrix} K(\eq+\rho)|\xi|^2  &  \grad  \phi \cdot  \,\im \xi \\ 
- \grad  \phi \cdot \,\im \xi &  \rho \, |\xi|^2
\end{bmatrix}}
\boldsymbol{C}U 
 + \boldsymbol{C}^{-1} R(\boldsymbol{C}U)  \, . 
 \end{aligned}
 \end{equation}
Using \eqref{scala}, system \eqref{EKcf10} reads as system \eqref{EKcf1}-\eqref{A1} with
$ R(U) := \boldsymbol{C}^{-1} R(\boldsymbol{C}U) $. 

We note also that estimates \eqref{ba1} and \eqref{la1} for $j = 2$ follow by \eqref{cont00} and \eqref{moser}, whereas in case  $j = 1$  follow by \eqref{cont2} applied with $m=1$, $\vr=1$. 
\\[1mm]
{\bf Step 3: Estimate of the remainder $ R(U) $}. 
We now prove  \eqref{restoparalin0}-\eqref{restoparalin2}.
Since $ \| \rho \|_\sigma,  \| \phi \|_\sigma \sim \| U \|_\sigma $ for any $ \sigma \in \R $
by \eqref{scala}, the estimates 
\eqref{restoparalin0}-\eqref{restoparalin2} directly follow from  
those of $ R(\rho, \phi) $ in \eqref{EKpara}. 
We now estimate each term in \eqref{paraline1}--\eqref{bfgr3}. 
In the sequel $\sigma \geq s_0 > d / 2 $.
\\[1mm]
{\sc Estimate of the terms in line \eqref{paraline1}.}  Applying first  \eqref{cont00} 
with $ m = 0 $,  and then \eqref{cont2} with $\vr=2$,  we have 
\begin{equation}\label{ess1}
\| \Opbw{\Delta \phi} \rho \|_{ \sigma}\lesssim \| \phi\|_{ {s_0+2}} \| \rho\|_{ \sigma} \,  , 
\quad \| \Opbw{\Delta \phi} \rho \|_{ \sigma}\lesssim 
\| \phi\|_{ {s_0}} \| \rho\|_{ {\sigma+2}} \, . 
\end{equation}
By  \eqref{bonyeq2}, 
the smoothing remainder in line \eqref{paraline1} satisfies the estimates
\begin{equation}\label{ess2}
\|R^p(\rho,\Delta \phi)\|_{ \sigma} \lesssim \| \phi\|_{ {s_0+2}}\| \rho\|_{ \sigma} \, , 
\quad 
\|R^p(\rho,\Delta \phi)\|_{ \sigma} \lesssim \| \phi\|_{ {s_0}}\| \rho\|_{ {\sigma+2}} \, ,
\end{equation}
and, by \eqref{compS} with $ \varrho = 2 $,
and the interpolation estimate \eqref{interp3}, 
\begin{equation}\label{ess3}
\| R^{-2}(\rho, |\xi|^2)\phi \|_{\sigma}\lesssim \| \rho\|_{{s_0+2}} \| \phi\|_{\sigma}
\lesssim  \| \phi\|_{{s_0}} \| \rho\|_{{\sigma+2}}+\| \rho\|_{{s_0}} \| \phi\|_{{\sigma+2}} \, .
\end{equation}
By \eqref{ess1}-\eqref{ess3} and $ \| \rho \|_\sigma,  \| \phi \|_\sigma \sim \| U \|_\sigma $
we deduce that the terms  in line \eqref{paraline1}, written in function of $ U  $, 
satisfy \eqref{restoparalin0}.
Next we write
$$
\Opbw{\Delta \phi_1} \rho_1- \Opbw{\Delta \phi_2} \rho_2 = 
\Opbw{\Delta \phi_1} [\rho_1-\rho_2] + 
\Opbw{\Delta \phi_1-\Delta \phi_2}\rho_2 
$$
and, applying  \eqref{cont00} with $ m = 0$, and 
\eqref{cont2} with $\vr = 2$ to $\Opbw{\Delta \phi_1-\Delta \phi_2}\rho_2$,  we get
\begin{equation}\label{ess4}
\begin{aligned}
 \| \Opbw{\Delta \phi_1} \rho_1- \Opbw{\Delta \phi_2} \rho_2\|_{ \sigma}
 &\lesssim \| \phi_1\|_{ {s_0+2}}\| \rho_1-\rho_2\|_{ \sigma}+  \| \phi_1-\phi_2\|_{ {s_0+2}}\| \rho_2\|_{ \sigma}  \\
  \| \Opbw{\Delta \phi_1} \rho_1- \Opbw{\Delta \phi_2} \rho_2\|_{ \sigma}
 &\lesssim  \| \phi_1\|_{ {s_0+2}}\| \rho_1-\rho_2\|_{ \sigma}+  \| \phi_1-\phi_2\|_{ {s_0}}\| \rho_2\|_{ {\sigma+2}} \, .
 \end{aligned}
 \end{equation}
Concerning the remainder $R^p(\rho, \Delta \phi)$,  we write $R^p(\rho_1,\Delta \phi_1)-R^p(\rho_2,\Delta \phi_2) = 
 R^p(\rho_1-\rho_2,\Delta \phi_1) +R^p(\rho_2,\Delta \phi_2- \Delta \phi_1)$ and, applying 
\eqref{bonyeq2}, we get 
\begin{equation}\label{ess5}
 \begin{aligned}
 \|R^p(\rho_1,\Delta \phi_1)-R^p(\rho_2,\Delta \phi_2)\|_{ \sigma} 
 &\lesssim  \| \phi_1\|_{ {s_0+2}} \| \rho_1-\rho_2\|_{ \sigma} + \| \rho_2\|_{ \sigma} \| \phi_1-\phi_2\|_{ {s_0+2}} \\
 \|R^p(\rho_1,\Delta \phi_1)-R^p(\rho_2,\Delta \phi_2)\|_{ \sigma} 
 &\lesssim \| \phi_1\|_{ {s_0+2}} \| \rho_1-\rho_2\|_{ \sigma} + \| \rho_2\|_{ {\sigma+2}} \| \phi_1-\phi_2\|_{ {s_0}} \, .
 \end{aligned}
 \end{equation}
 Finally  we write  
 $
 R^{-2}(  \rho_1, |\xi|^2)\phi_1- R^{-2}( \rho_2, |\xi|^2)\phi_2 = R^{-2}(  \rho_1-\rho_2, |\xi|^2)\phi_1 + 
 R^{-2}( \rho_2, |\xi|^2)[\phi_1- \phi_2] $. 
Using   \eqref{compS} we get 
\begin{equation}\label{ess50}
 \| R^{-2}(  \rho_1, |\xi|^2)\phi_1- R^{-2}( \rho_2, |\xi|^2)\phi_2\|_{\sigma}
 \lesssim \| \phi_1\|_{\sigma}\| \rho_1-\rho_2\|_{s_0+2}
 +  \| \phi_1-\phi_2\|_{\sigma}\| \rho_2\|_{s_0+2} \, .
\end{equation}
We also claim that 
\begin{equation}\label{ess7}
 \| R^{-2}(  \rho_1, |\xi|^2)\phi_1- R^{-2}( \rho_2, |\xi|^2)\phi_2\|_{\sigma}
 \lesssim
 \| \rho_1-\rho_2\|_{s_0} \| \phi_1\|_{\sigma+2}+
  \| \phi_1-\phi_2\|_{\sigma}\| \rho_2\|_{s_0+2}  \, .
\end{equation}
Indeed, we bound 
\begin{align*}
 \| R^{-2}(  \rho_1, |\xi|^2)\phi_1- R^{-2}( \rho_2, |\xi|^2)\phi_2\|_{\sigma}
 &\lesssim \| R^{-2}(  \rho_1-\rho_2, |\xi|^2)\phi_1\|_{\sigma}+  \| \phi_1-\phi_2\|_{\sigma}\| \rho_2\|_{s_0+2} \, 
\end{align*}
and, to control  $R^{-2}(  \rho_1-\rho_2, |\xi|^2)\phi_1$,  we use that, by  definition, it equals 
\begin{align*}
 \Opbw{\rho_1-\rho_2}\Opbw{|\xi|^2}\phi_1- \Opbw{ (\rho_1-\rho_2)|\xi|^2}\phi_1 -\Opbw{ \nabla(\rho_1-\rho_2)\cdot \im \xi} \phi_1 
\end{align*}
and we estimate the first two terms using   \eqref{cont2} with $\vr = 0$ and the last term   with $\vr = 1$, by 
$\| R^{-2}(  \rho_1-\rho_2, |\xi|^2)\phi_1\|_{\sigma}\lesssim 
 \| \rho_1-\rho_2\|_{s_0} \| \phi_1\|_{\sigma+2} 
 $,  
 proving \eqref{ess7}. By \eqref{ess4}-\eqref{ess7} and $ \| \rho \|_\sigma,  \| \phi \|_\sigma \sim \| U \|_\sigma $
we deduce that the terms  in line \eqref{paraline1}, written in function of $ U  $, 
satisfy \eqref{restoparalin1}-\eqref{restoparalin2}.
 \\[1mm]
The estimates \eqref{restoparalin0}-\eqref{restoparalin2}
for the terms in lines \eqref{paraline2}, \eqref{paraline3}, \eqref{paraline5} and \eqref{paraline4}, follow by similar arguments, using also \eqref{moser}. 
\\[1mm]
{\sc  Estimates of $ {\mathtt R}(\rho) $ defined in \eqref{bfgr}-\eqref{bfgr3}.}

Writing 
$ \Opbw{|\nabla \rho|^2}K'(\eq+\rho) = \Opbw{|\nabla \rho|^2}(K'(\eq+\rho) -  K'(\eq))$
(in the homogeneous spaces $ \dot H^s $), we have, 
by \eqref{cont00}, the fact that $ \rho \in \cQ $, 
Theorem \ref{thm:moser},   \eqref{bonyeq2}, \eqref{interp2},  
\eqref{compS} with $\vr=1$, 
$$
   \| {\mathtt R} (\rho)\|_{ \sigma} \leq \tC \big(\| \rho\|_{ {s_0+2}}\big) \| \rho\|_{ \sigma} \, .  
$$
  Thus $ {\mathtt R} (\rho) $, written as a function of $ U $, satisfies \eqref{restoparalin0}.
  The estimates \eqref{restoparalin1}-\eqref{restoparalin2} follow by 
  \begin{align}
 &\label{sbound}\| {\mathtt R}(\rho_1)-{\mathtt R}(\rho_2)\|_{ \sigma} \leq 
 \tC\big(\| \rho_1\|_{ {s_0+2}},\| \rho_2\|_{ {s_0+2}}\big)\| \rho_1-\rho_2\|_{ \sigma}+
 \tC \big(\| \rho_1\|_{ {\sigma}},\| \rho_2\|_{ {\sigma}}\big)\| \rho_1-\rho_2\|_{ {s_0+2}}  \\
 &\label{tbound}\| {\mathtt R}(\rho_1)-{\mathtt R}(\rho_2)\|_{ {s_0}} \leq 
 \tC \big(\| \rho_1\|_{ {s_0+2}},\| \rho_2\|_{ {s_0+2}}\big)\| \rho_1-\rho_2\|_{ {s_0}} \, .
 \end{align}
 {\sc Proof of \eqref{sbound}.}
Defining $ w := \nabla( \rho_1+\rho_2)$, $v:= \nabla (\rho_1-\rho_2)$, then we have, by
\eqref{iterp1},   
  \begin{align}
  &  \| |\nabla \rho_1|^2- |\nabla \rho_2|^2\|_{ {s_0}} = \|  w \cdot v\|_{ {s_0}}
  \lesssim \big( \| \rho_1\|_{ {s_0+1}} +\| \rho_2\|_{ {s_0+1}}\big) \| \rho_1-\rho_2\|_{ {s_0+1}}\label{nvj}
 \\
&   \| |\nabla \rho_1|^2- |\nabla \rho_2|^2\|_{ {s_0-1}} = \|  w \cdot v\|_{ {s_0-1}}
   \stackrel{\eqref{interp2}}\lesssim \big( \| \rho_1\|_{ {s_0+1}} +\| \rho_2\|_{ {s_0+1}}\big) \| \rho_1-\rho_2\|_{ {s_0}} \, .  \label{nvk}
   \end{align}
Let us prove \eqref{sbound} for the first term in \eqref{bfgr}. 
 Remind that $ \rho_1, \rho_2 $ are in $ \cQ $. 
 We have 
  \begin{align}
  \| &\Opbw{|\nabla \rho_1|^2}K'(\eq+\rho_1)-\Opbw{|\nabla \rho_2|^2}K'(\eq+\rho_2)\|_{ \sigma} \nonumber \\
  &\leq \| \Opbw{w \cdot v} \big( K'(\eq+ \rho_1)\|_{ \sigma} + \| \Opbw{|\nabla \rho_2|^2}\big[K'(\eq+\rho_1)-K'(\eq+\rho_2)\big]\|_{ \sigma} \nonumber \\
  &\stackrel{\eqref{cont00}}\lesssim \| w\cdot v\|_{ {s_0}} 
  \| K'(\eq+ \rho_1)-K'(\eq)\|_{ \sigma}+ \| \rho_2\|_{ {s_0+1}}^2 \|K'(\eq+\rho_1)-K'(\eq+\rho_2)\|_{ \sigma} \nonumber \\
  &\stackrel{\eqref{interp3},\eqref{moser},\eqref{nvj}}\lesssim 
  \| \rho_1\|_\sigma (\| \rho_1\|_{s_0+1}+ \| \rho_2\|_{s_0+1}) \| \rho_1-\rho_2\|_{s_0+1}  + \tC \big( \| \rho_1\|_{s_0+1}, \|\rho_2\|_{ {s_0+1}} \big) \| \rho_1-\rho_2\|_{ {\sigma}}
  \nonumber \\
  &\ \ \ \ \ \ \ \ \ + \| \rho_2\|_{ {s_0}}\| \rho_2\|_{s_0+2}( \| \rho_1\|_\sigma + \| \rho_2\|_\sigma ) \| \rho_1-\rho_2\|_{s_0} \nonumber \\
  &\stackrel{\eqref{interp3}}\leq \tC \big( \| \rho_1\|_\sigma, \|\rho_2\|_{ {\sigma}} \big) \| \rho_1-\rho_2\|_{ {s_0+2}}
 +
  \tC \big( \| \rho_1\|_{s_0+2}, \|\rho_2\|_{ {s_0+2}} \big) \| \rho_1-\rho_2\|_{ {\sigma}} \, . 
  \label{boundvol}
  \end{align}
  In the same way the second term in \eqref{bfgr} is bounded by \eqref{boundvol}. 
  Regarding the term in  \eqref{bfgr2}, using that $R^p(\cdot,\cdot)$ is bilinear and symmetric, we have 
  \begin{align}
  \notag \| &\Opbw{K'(\eq+\rho_1)}R^p(\nabla \rho_1, \nabla \rho_1)-\Opbw{K'(\eq+\rho_2)}R^p(\nabla \rho_2, \nabla \rho_2)\|_{ \sigma}\\
  \notag&\leq \| \Opbw{K'(\eq+\rho_1)-K'(\eq +\rho_2)}R^p(\nabla \rho_1, \nabla \rho_1)\|_{ \sigma} + \| \Opbw{K'(\eq+\rho_2)}R^p(w, v)\|_{ \sigma} 	\nonumber \\
  \notag &\stackrel{\eqref{cont00}, \eqref{bonyeq2}}\lesssim \| \rho_1\|_{ {\sigma}}\| \rho_1\|_{s_0+2} \| K'(\eq+ \rho_1)-K'(\eq+\rho_2)\|_{ {s_0}}+ \| w\|_{ {s_0+1}}\| v\|_{\sigma-1} \|K'(\eq+\rho_2)\|_{ {s_0}} \nonumber \\
 &\stackrel{\eqref{moser},\eqref{nvj}}\leq \tC \big( \| \rho_1\|_{\sigma}, \| \rho_2\|_{\sigma} \big) \| \rho_1-\rho_2\|_{s_0}+ \tC \big( \| \rho_1\|_{s_0+2}, \|\rho_2\|_{s_0+2} \big) \| \rho_1-\rho_2\|_{\sigma} \label{prontagia} \, .
  \end{align}
Also the terms in \eqref{bfgr3} are bounded by \eqref{boundvol}, 
 proving that ${\mathtt R}(\rho)$ satisfies  \eqref{sbound}.
  \\[1mm]
  {\sc Proof of \eqref{tbound}}.
Regarding the first term \eqref{bfgr},   we have 
\begin{align}
  & 
  \| \Opbw{|\nabla \rho_1|^2}K'(\eq+\rho_1)-\Opbw{|\nabla \rho_2|^2}K'(\eq+\rho_2)\|_{ s_0}
  \nonumber \\
  & \leq \| \Opbw{w \cdot v}K'(\eq+ \rho_1)\|_{ s_0} + \| \Opbw{|\nabla \rho_2|^2}\big[K'(\eq+\rho_1)-K'(\eq+\rho_2)\big]\|_{ s_0} \nonumber \\
  & \stackrel{\eqref{cont00},\eqref{cont2}}\lesssim \| w \cdot v \|_{ s_0-1}\|K'(\eq+ \rho_1)\|_{s_0+1}+ \| \rho_2\|_{ {s_0+1}}^2 \|K'(\eq+\rho_1)-K'(\eq+\rho_2)\|_{ s_0} \nonumber \\
  & \stackrel{\eqref{moser},\eqref{nvk}}\leq \tC \big( \| \rho_1\|_{s_0+1}, \|\rho_2\|_{ {s_0+1}} \big) \| \rho_1-\rho_2\|_{ {s_0}} \, . \label{329}
  \end{align}
Similarly we deduce that the second term in \eqref{bfgr} is bounded 
as in \eqref{329}. Regarding the  term in \eqref{bfgr2},   note that  the bound \eqref{tbound} follows from \eqref{prontagia}  applied for $\sigma= s_0$. The estimate for  last two terms in \eqref{bfgr3} follows in the same way so we analyze the last one. First we have
    \begin{align*}
 \ \ \ \  \|  R^{-1}& (K'(\eq+\rho_1), \im \nabla \rho_1\cdot \xi)\rho_1- R^{-1}(K'(\eq+\rho_2), \im \nabla \rho_2\cdot \xi)\rho_2\|_{ s_0}\\
 & \leq \| \big[ R^{-1}(K'(\eq+\rho_1),\nabla \rho_1\cdot\im \xi )- R^{-1}(K'(\eq+\rho_2), \im \nabla \rho_2\cdot \xi)\big] \rho_1\|_{ s_0} \\
  & +\|  R^{-1}(K'(\eq+\rho_2),\nabla \rho_2\cdot\im \xi ) (\rho_1-\rho_2)\|_{ s_0}\\
  \stackrel{\eqref{compS}, \eqref{moser}}\leq &\| \big[ R^{-1}(K'(\eq+\rho_1),\nabla \rho_1\cdot\im \xi )- R^{-1}(K'(\eq+\rho_2), \im \nabla \rho_2\cdot \xi)\big] \rho_1\|_{ s_0}+ \tC \big(\| \rho_2\|_{s_0+2}\big) \| \rho_1-\rho_2\|_{s_0}.
  \end{align*}
  On the other hand, by definition, we have 
  \begin{align}
  \big[ R^{-1}&(K'(\eq+\rho_1),\nabla \rho_1\cdot\im \xi )- R^{-1}(K'(\eq+\rho_2), \im \nabla \rho_2\cdot \xi)\big] \rho_1 \label{ulli} \\
 =   & \big[\Opbw{K'(\eq+\rho_1)}\Opbw{\nabla \rho_1\cdot\im \xi}-\Opbw{K'(\eq+\rho_2)}\Opbw{\nabla \rho_2\cdot\im \xi}\big]\rho_1 \nonumber \\
  &+ \Opbw{ K'(\eq+\rho_1)\nabla \rho_1\cdot\im \xi-K'(\eq+\rho_2)\nabla \rho_2\cdot\im \xi}\rho_1 \nonumber \\
  =  &\Opbw{K'(\eq+\rho_1)-K'(\eq+\rho_2)}\Opbw{\nabla \rho_1\cdot\im \xi}\rho_1
  \nonumber \\
  &+\Opbw{K'(\eq+\rho_2)}\Opbw{\nabla( \rho_1-\rho_2)\cdot\im \xi}\rho_1 \nonumber \\
  &+\Opbw{ \nabla ( K(\eq+\rho_1)-K(\eq+\rho_2))\cdot \im \xi}\rho_1 \, . \nonumber 
  \end{align}
 Then, applying first \eqref{cont00} to the first term and then  \eqref{cont2} with $\vr=1$, $m=1$ and  \eqref{moser} to each term, we deduce that the $ \| \ \|_{s_0} $-norm of
 \eqref{ulli}   is bounded by 
 $ \tC \big(\| \rho_1\|_{s_0+2}, \| \rho_2\|_{s_0+2}\big) \| \rho_1-\rho_2\|_{s_0}$. 
 Thus  \eqref{tbound} is proved. 
 \end{proof}

  \section{Local existence}\label{sec:4}
 
 In this section we prove the existence of a local in time solution of system \eqref{EKcf1}. For any $s \in \R$ and $T >0$, we denote
 $L^\infty_T \dot \bH^s := L^\infty ([0,T], \dot \bH^s) $. 
 For $ \delta > 0 $ we also introduce
\begin{equation}\label{def:Qdelta}
\mathcal{Q}_\delta 
:= \big\{ \rho \in H^{s_0}_0 \ \colon \ \eq_1 + \delta \leq \eq + \rho(x) \leq \eq_2 - \delta \big\}  \subset \cQ  
\end{equation}
 where $ \cQ  $ is defined in \eqref{def:Q}.
 
 \begin{proposition}{\bf (Local well-posedness in $\T^d$)} \label{prop:LWP}
 For any $s > \frac{d}{2}+2$, any initial datum
$U_0 \in \dot \bH^s $ with 
$  \rho(U_0) \in \cQ_\delta $ for some $ \delta >  0 $, 
 there exist $T := T(\| U \|_{s_0+2}, \delta) > 0 $ and a unique solution $U \in C^0\big([0, T], \dot \bH^s \big) \cap C^1\big([0, T], \dot \bH^{s-2} \big)$ of \eqref{EKcf1} satisfying 
 $ \rho (U) \in \cQ $, for any $ t \in [0,T] $. Moreover the solution depends
 continuously with respect to the initial datum in $ \dot \bH^s $.
 \end{proposition}

Proposition \ref{prop:LWP} proves  Theorem
\ref{thm:main1} and thus Theorem \ref{thm:main0}. 

The  first step is to   prove the   local well-posedness result of a linear inhomogeneous problem.

\begin{proposition}{\bf (Linear local well-posedness)} \label{LinLWP}
Let $\Theta\geq r > 0 $ and $ U $ be a function in 
$ C^0([0,T],\dot \bH^{s_0+2})  \cap C^1([0,T],\dot \bH^{s_0}) $
satisfying  
\begin{equation}
\label{U.ass}
\norm{U}_{L^\infty_T \dot \bH^{s_0+2} } +
\norm{\pa_t U}_{L^\infty_T \dot \bH^{s_0}  } \leq \Theta \, ,  \quad 
\| U\|_{L^\infty_T \dot \bH^{s_0}}\leq r \, , \quad 
\rho(U(t)) \in \cQ \, , \  \forall t \in [0, T] \, . 
\end{equation}
Let  $ \sigma  \geq 0$ and  
$ t \mapsto R(t)$ be a function in $ C^0([0, T], \dot\bH^\sigma) $. 
Then  there exists a unique solution $V\in C^0([0,T], \dot \bH^\sigma)\cap C^1([0,T], \dot \bH^{\sigma-2})$  of the linear inhomogeneous system 
\begin{equation}\label{linpro}
\pa_t V = \J \, 
\Opbw{  A_2(U(t); x, \xi)  + A_1(U(t); x, \xi)  }V +  R(t) \, ,  
\quad 
V(0,x) =  V_0(x) \in \dot \bH^\s\, ,
\end{equation}
satisfying, for some  {$ C_\Theta := C_{\Theta,\sigma} > 0 $ 
and  $C_r := C_{r,\sigma} > 0 $}, the estimate 
\begin{equation}
\label{duaest} \| V\|_{L^\infty_T \dot \bH^\sigma} \leq C_r e^{C_\Theta T} \| V_0\|_{\sigma}+ C_\Theta e^{C_\Theta T } T \| R\|_{L^\infty_T \dot \bH^\sigma}.
\end{equation}
\end{proposition}
The following two sections are devoted to the proof of Proposition \ref{LinLWP}. 
The key step is the construction of a  modified energy which is 
controlled by the $\dot \bH^\s$-norm, and whose time variation is bounded by the $\dot \bH^\s$ norm of the solution, as done e.g.  in \cite{ABZ} and \cite{MaRo} for linear systems.
In order to construct such  modified energy, the first step is to diagonalize  the matrix 
$ \J  A_2$ in \eqref{linpro}.
 \subsection{Diagonalization at highest order}
We  diagonalize the matrix of symbols 
$\J A_2(U; x, \xi)$.  The eigenvalues of the matrix 
\begin{equation}
\label{Adiag1}
\J  \begin{bmatrix}
1 + \ta_+(U; x) &  \ta_-(U; x) \\
\ta_-(U; x) & 1+ \ta_+(U; x)
\end{bmatrix}
\end{equation}
with $\ta_\pm(U;x)$ defined in \eqref{tatb} are given by $\pm \im \lambda(U; x)$ with 
\begin{equation}
\label{Adiag2}
\lambda(U; x) := \sqrt{(1+\ta_+(U;x))^2 - \ta_-(U;x)^2}  = \sqrt{\frac{(\eq+\rho(U)) \, K(\eq+\rho(U))}{\eq \, K(\eq)}} \, .  
\end{equation}
These eigenvalues are purely imaginary because $\rho(U) \in \cQ$ 
(see \eqref{def:Q}) and \eqref{ellip}, which guarantees that
$\lambda(U;x)  $ is real valued and fulfills
\begin{equation}
\label{est.lambda}
0< \lambda_{\min}:= \sqrt{\frac{\eq_1 c_K}{\eq K(\eq)}}\leq  \lambda(U;x) \leq \sqrt{\frac{\eq_2 C_K}{\eq K(\eq)}}=: \lambda_{\max}  \, . 
\end{equation}
 A matrix which diagonalizes \eqref{Adiag1} is 
 \begin{equation}
 \label{Adiag4}
 F : = 
 \begin{pmatrix}
 f (U; x )& g(U; x ) \\
 g(U; x ) & f(U; x ) 
 \end{pmatrix} 
 , \ \ 
f := \displaystyle{ \frac{1+\ta_+ + \lambda}{\sqrt{(1+\ta_+ + \lambda)^2 - \ta_-^2}} } \, , \ \ 
g:=  
\displaystyle{  \frac{ -\ta_-}{\sqrt{(1+\ta_+ + \lambda)^2 - \ta_-^2}} }\, . 
 \end{equation}
Note that $ F(U; x) $ is well defined because 
{ \begin{align}
(1+\ta_+ + \lambda)^2 - \ta_-^2 
& = \Big(\frac{K(\eq+\rho(U))}{K(\eq)}+ \lambda \Big)
 \Big(\frac{\eq+\rho(U)}{\eq}+ \lambda \Big) \nonumber \\
 & > 
 \frac{(\eq+\rho(U)  )K(\eq+\rho(U))}{\eq K(\eq)} \geq 
 \frac{\eq_1 c_K}{\eq K(\eq)}  \label{dislo2}
\end{align}
 by  \eqref{def:Q} and \eqref{ellip}}.  
 The matrix $F(U;x)$ has $ \det F(U;x) = {f^2 - g^2} = 1 $ and its inverse is 
  \begin{equation}
 \label{Adiag5}
 F(U; x )^{-1} : = 
 \begin{pmatrix}
f (U; x )& - g(U; x ) \\
-g(U; x ) & f(U; x )
 \end{pmatrix} \, . 
 \end{equation}
We have that
\begin{equation}
\label{AFAF}
F(U;x)^{-1} \,
\J  \begin{bmatrix}
1 + \ta_+(U; x) &  \ta_-(U; x) \\
\ta_-(U; x) & 1+ \ta_+(U; x)
\end{bmatrix}\,  F(U;x) 
 = 
 \J \lambda(U; x) \, .
\end{equation} 
{By \eqref{moser} {and \eqref{dislo2}} we deduce 
the following estimates: for any $ N \in \N_0 $, {$ s \geq 0 $}} and $\s>\frac{d}{2}$, 
\begin{equation}\label{simboli}
\begin{aligned}
& \norm{\ta_\pm(U)}_{\sigma}, \, \norm{f(U)}_{\s}, \, \norm{g(U)}_{\s}\leq \tC \big( \| U\|_{\s}\big) \, , 
\\
&|\lambda(U;x)|\xi|^{2s}|_{2s,\s,N}\leq \tC_N \big(\| U\|_\s\big) \, , \quad
 | \tb(U)\cdot \xi|_{1,\s,N}\leq \tC_N (\| U\|_{\s+1}\big) \, . 
\end{aligned}
 \end{equation}
For any $\varepsilon> 0 $, consider the regularized matrix symbol
\begin{equation}
\label{Aeps}
A^\varepsilon(U;x,\xi) :=  \left( A_2(U;x,\xi)+A_1(U;x,\xi) \right)  \chi 
( \varepsilon \lambda(U;x) |\xi|^2 ) \, , 
\end{equation}
where $ \chi $ is the cut-off function in \eqref{def-chi} and $\lambda(U;x)$ is the 
{function} defined in \eqref{Adiag2}. 
In what follows we will denote by $\chi_\varepsilon:= \chi( \varepsilon \lambda(U; x) |\xi|^2)$. 
Note that,  by \eqref{moser}, \eqref{est.lambda} and by the fact that the function $y \mapsto \langle \xi \rangle^{|\alpha|}\pa_\xi^\alpha \chi( \varepsilon y  |\xi|^2)$ is bounded together with its derivatives uniformly in $\varepsilon\in (0,1)$, $ \xi \in \R^d$ and $y\in [\lambda_{\min}, \lambda_{\max}]$,  the symbol $ \chi_\varepsilon $ satisfies, 
for any $ N\in \N_0 $, $ \sigma > d/2 $
\begin{equation}\label{es:chie}
| \chi_\varepsilon|_{0, \sigma, N} \leq \tC\left(\|U\|_{\sigma}\right) \, ,
\quad {\rm uniformly \ in} \ \varepsilon \, .
\end{equation}
The diagonalization \eqref{AFAF} has the following operatorial consequence.
\begin{lemma}\label{parametrica}
We have 
\begin{equation}\label{Aconj}
 \Opbw{F^{-1}} \, \J \Opbw{A^\varepsilon} \, \Opbw{F}=\J \Opbw{(\sqrt{\eq K(\eq)}\lambda |\xi|^2+ \tb\cdot \xi)\chi_\varepsilon}+ \mathcal{F}(U)
 \end{equation}
where $\mathcal{F}(U):= \mathcal{F}_\varepsilon (U)\colon \dot \bH^\sigma \to \dot \bH^\sigma$, $\forall \sigma \geq 0$,  satisfies, uniformly in $ \varepsilon $,  
\begin{equation}\label{matf}
\| \mathcal{F}(U) W\|_{\sigma}\leq {\tC ( \| U \|_{s_0+2})} \|W\|_{\sigma} \, , \quad \forall W \in \dot \bH^\sigma \, .
\end{equation}
\end{lemma}
\begin{proof}
We have that 
$$
 \Opbw{F^{-1}}\J \Opbw{A_2\chi_\varepsilon} \Opbw{F}= 
 \J \sqrt{\eq K(\eq)}
 \begin{bmatrix}
 D_2 & B_2 \\
 B_2 & D_2
 \end{bmatrix},
$$
where 
 \begin{align*}
 D_2  & =  \Opbw{f} \Opbw{ |\xi|^2(1+\ta_+) \chi_\varepsilon}\Opbw{ f}+
 \Opbw{g}\Opbw{ |\xi|^2(1+\ta_+)\chi_\varepsilon }\Opbw{g}\\
& \ \ + \Opbw{f} \Opbw{ |\xi|^2 \ta_- \chi_\varepsilon} \Opbw{ g}
+ \Opbw{g}  \Opbw{ |\xi|^2 \ta_-\chi_\varepsilon}\Opbw{ f} \\
B_2  & =   
 \Opbw{f} \Opbw{ |\xi|^2(1+\ta_+)\chi_\varepsilon}\Opbw{ g} + 
\Opbw{g} \Opbw{ |\xi|^2(1+\ta_+)\chi_\varepsilon}\Opbw{ f}\\
 &\ \ +\Opbw{ f }\Opbw{ |\xi|^2 \ta_-\chi_\varepsilon}  \Opbw{ f}
+ \Opbw{g} \Opbw{ |\xi|^2 \ta_-\chi_\varepsilon}  \Opbw{ g} \, .
 \end{align*}
{By  Corollary \ref{symweyl}} we obtain 
 \begin{align*}
 &D_2= \Opbw{\left[(f^2+g^2)(1+\ta_+)+2fg \ta_-\right]|\xi|^2\chi_\varepsilon}+\mathcal{F}_1(U)=\Opbw{\lambda(U)|\xi|^2\chi_\varepsilon}+ \mathcal{F}_1(U) \, , \\
 & B_2= \Opbw{\left[(f^2+g^2)\ta_-+2fg(1+\ta_+)\right]|\xi|^2\chi_\varepsilon}+ \mathcal{F}_2(U)=\mathcal{F}_2(U) \, ,
 \end{align*}
 where  $ \mathcal{F}_1,\mathcal{F}_2 $ satisfy \eqref{matf} {by \eqref{restoR0}}, 
 \eqref{simboli},  and {\eqref{es:chie}} and {since, by the 
 definition of $ f $ and  $ g $ in \eqref{Adiag4} and $ \lambda$ in \eqref{Adiag2}, 
 we have 
 $  (f^2+g^2)(1+\ta_+) + 2fg\ta_- =
 \lambda $ and $ (f^2+g^2)\ta_-+2fg(1+\ta_+)=0 $}. 
Moreover 
\begin{equation*}
\Opbw{F^{-1}}\J \Opbw{A_1\chi_\varepsilon} \Opbw{F}= 
 \J 
 \begin{bmatrix}
 D_1 & B_1\\
 -B_1& -D_1
 \end{bmatrix},
 \end{equation*}
 where 
 \begin{align*}
 D_1 & = \Opbw{f} \Opbw{ \tb\cdot \xi\chi_\varepsilon }\Opbw{ f}-
 \Opbw{g}\Opbw{ \tb \cdot \xi\chi_\varepsilon}\Opbw{g}\\
B_1 & = 
 \Opbw{f} \Opbw{ \tb\cdot \xi \chi_\varepsilon}\Opbw{ g} -
\Opbw{g} \Opbw{ \tb\cdot \xi \chi_\varepsilon}\Opbw{ f} \, .
 \end{align*}
 Applying Theorem \ref{thm:compS},  \eqref{simboli}, {\eqref{es:chie}}, 
 using that $f^2-g^2=1$ we obtain 
$ D_1= \Opbw{\tb \cdot \xi \, \chi_\varepsilon}+ \mathcal{F}_1(U) $ and 
$ B_1= \mathcal{F}_2(U) $ 
with $\mathcal{F}_1,\mathcal{F}_2$ satisfying \eqref{matf}. 
\end{proof}

\subsection{Energy estimate for smoothed system}\label{sec:4.2}

We first solve  \eqref{linpro} in the case $R(t)=0$ and $V_0 \in \dot C^\infty:= \cap_{\s\in \R} \dot \bH^\s$. 
Consider the  regularized Cauchy problem 
\begin{equation}\label{epsisol}
\partial_t V^\varepsilon= \J \Opbw{A^\varepsilon(U(t);x,\xi)} V^\varepsilon \, , 
\quad  
V^\varepsilon(0)=  V_0  \in \dot C^\infty \, , 
\end{equation}
where $A^\varepsilon(U; x, \xi) $ is defined in \eqref{Aeps}.
As the operator $ \Opbw{A^\varepsilon(U;x,\xi)} $ is bounded for any 
$ \varepsilon >0$, {and $ U(t)$ satisfies \eqref{U.ass}}, the differential equation \eqref{epsisol} has a unique  solution $V^\varepsilon(t)$ {which belongs to $ C^2 ([0,T],\dot \bH^\sigma)$ for any 
$  \sigma \geq 0$}.
The important fact is that it admits  the following $\varepsilon$-independent energy estimate.
\begin{proposition}\label{exp}
{\bf (Energy estimate)}
Let $U$ satisfy \eqref{U.ass}.
For any $\sigma \geq 0$, there exist 
constants  $ C_r , C_{\Theta} > 0 $ (depending also on $\sigma$), 
such that for any $ \varepsilon > 0 $, 
the {unique} solution $V^\varepsilon(t)$ of {epsisol} fulfills
\begin{equation}
\| V^\varepsilon(t)\|_{\sigma}^2\leq  C_{r}\|V_0\|_{\sigma}^2+ C_{\Theta} \int_0^t \|V^\varepsilon(\tau)\|^2_{\sigma} \, \di \tau \, , \quad \forall t\in [0,T] \,.  \label{gron}
\end{equation}
{As a consequence, there are constants $ C_r, C_\Theta  $  independent of $\varepsilon$, such that}
\begin{equation}\label{point}
{ \| V^\varepsilon(t)\|_{{\s}}\leq C_r e^{C_\Theta t} \| V_0\|_{{\s}} \, , 
\quad  \forall t\in [0,T] \, .} 
\end{equation}
\end{proposition}
In order to prove Proposition \ref{exp}, we  define, for any $\sigma \geq 0$,  
the {\em modified energy }
\begin{equation}\label{enemod}
\| V\|_{\s,U}^2 := \langle \Opbw{ \lambda^\s(U;x) |\xi|^{2\s}} \Opbw{F^{-1}(U;x)} V, \Opbw{F^{-1}(U;x)}V\rangle \, ,
\end{equation}
where we introduce the {\em real} scalar product 
\begin{equation*}
\langle V,W\rangle :=2 \Re \int_{\T^d} v(x) \bar w (x)\, \di x \, , \qquad
V=\begin{bmatrix} v \\ \bar v\end{bmatrix} , \ \ \  W=\begin{bmatrix} w \\ \bar w\end{bmatrix} \, .  
\end{equation*}
 
\begin{lemma}
Fix $\sigma \geq 0$, $r >0$.  There exists a constant $C_{r}>0$ (depending also on $\s$) such that  for any   $U \in \dot \bH^{s_0}$ with  $\norm{U}_{s_0} \leq r$ 
 and $\rho(U) \in \cQ$ 
we have
\begin{equation}\label{equ}
C_{r}^{-1}  \| V\|_{\s}^2 -\| V\|_{-2}^2 \leq \|V\|_{\s,U}^2\leq C_{r} \|V\|_{\s}^2 \,  , \quad \forall V\in \dot \bH^\s \, . 
\end{equation}
\begin{proof}
We first prove the upper bound in \eqref{equ}. We note that, by \eqref{simboli}, $\lambda^\s(U;x) |\xi|^{2\s}\in \Gamma_{s_0}^{2\s}$ and $F^{-1}(U;x) \in \Gamma_{s_0}^0\otimes \mathcal{M}_2(\C)$ and, by Theorem \ref{thm:contS} and  \eqref{simboli} we have 
\begin{align*}
\|V\|_{\s,U}^2&\leq \|  \Opbw{ \lambda^\s(U;x) |\xi|^{2\s}} \Opbw{F^{-1}(U;x)} V\|_{{-\s}}\|  \Opbw{F^{-1}(U;x)} V\|_{\s}\leq C_{r} \|V\|_{\s}^2 \, .
\end{align*}
In order to prove the lower bound, we fix $\delta \in (0,1)$ 
such that $s_0-\delta>\frac{d}{2}$ and, due to \eqref{est.lambda},
we have $ \lambda^{-\frac{\s}{2}}\in \Gamma_{s_0-\delta}^{0} $. 
 So, applying  Theorem \ref{thm:compS} and  \eqref{simboli} with $s_0-\delta$ instead of $s_0$ and with $\vr=\delta$, we have 
\begin{equation}\label{tantecomp}
\Opbw{\lambda^{-\frac{\s}{2}}}\Opbw{F} \Opbw{\lambda^{\frac{\s}{2}}}\Opbw{F^{-1}}= \uno + \mathcal{F}^{-\delta}(U) \, ,
\end{equation}
where for any $\sigma' \in \R$ there exists a constant $C_{r, \sigma'}>0$ such that 
\begin{equation}\label{stima1}
\| \mathcal{F}^{-\delta}(U)f\|_{\sigma'}\leq C_{r,\sigma'}\| f\|_{{\sigma'-\delta}} \, , \quad \forall f\in \dot \bH^{\sigma'-\delta} \, .
\end{equation}
Again, applying Theorem \ref{thm:compS} with $s_0-\delta$ instead of $s_0$ and with $\vr=\delta$, we have also 
\begin{equation}\label{pochecomp}
\Opbw{\lambda^{\frac{\s}{2}}}\Opbw{|\xi|^{2\s}}\Opbw{\lambda^{\frac{\s}{2}}}
= \Opbw{\lambda^\s|\xi|^{2\s}}+ \mathcal{F}^{2\s-\delta}(U) \, , 
\end{equation}
where for any $\sigma' \in \R$ there exists 
a constant $C_{r,\sigma'}>0$   such that 
\begin{equation}\label{stima2}
\| \mathcal{F}^{2\s-\delta}(U)f\|_{{\sigma'-2\s+\delta}}\leq C_{r,\sigma'}\| f\|_{{\sigma'}} \, , \quad \forall f\in \dot \bH^{\sigma'} \, .
\end{equation}
By \eqref{tantecomp}--\eqref{stima2}, Theorem \ref{thm:contS} and  \eqref{simboli} and using also that $\Opbw{\lambda^{\frac{\s}{2}}}$ is symmetric with respect to $\langle \cdot, \cdot \rangle$, we have 
\begin{align*}
\|V\|_{\s}^2 & \leq 2 \| \Opbw{\lambda^{-\frac{\s}{2}}}\Opbw{F} \Opbw{\lambda^{\frac{\s}{2}}}\Opbw{F^{-1}}V\|_{\s}^2+ 2\|  \mathcal{F}^{-\delta}(U)V\|_{\s}^2\\
&\leq C_{r} \left(\|\Opbw{\lambda^{\frac{\s}{2}}}\Opbw{F^{-1}}V\|_{\s}^2+ \|V\|_{{\s-\delta}}^2\right)\\
&=C_{r}\left(\langle\Opbw{ \lambda^{\frac{\s}{2}}} \Opbw{ |\xi|^{2\s}} \Opbw{ \lambda^{\frac{\s}{2}}} \Opbw{F^{-1}} V, \Opbw{F^{-1}}V\rangle+\|V\|_{{\s-\delta}}^2\right)\\
&= C_{r}\left(\|V\|_{\s,U}^2+ \langle\mathcal{F}^{2\s-\delta}(U) \Opbw{F^{-1}} V, \Opbw{F^{-1}}V\rangle+\|V\|_{{\s-\delta}}^2\right)\\
&\leq C_{r}\big( \|V\|_{\s,U}^2+\| V\|_{{\s-\frac{\delta}{2}}}^2\big) \, .
\end{align*}
Now we use {\eqref{interp4} and the asymmetric Young inequality} 
to get, for any $\epsilon > 0$, 
$$
\norm{V}_{{\s-\frac{\delta}{2}}}^2 \leq \norm{V}_{{-2}}^{\frac{\delta}{\s+2}} \  \norm{V}_{\s}^{\frac{2(\s+2)-\delta}{\s+2}} \leq 	\epsilon^{-\frac{2(\s+2)}{\delta}} \norm{V}_{{-2}}^2 + \epsilon^{\frac{2(\s+2)}{2(\s+2)-\delta}}\norm{V}_{\s}^2  ;
$$
we choose $\epsilon$ so small so that $ \epsilon^{\frac{2(\s+2)}{2(\s+2)-\delta}}C_{r}=\frac12$ and we get  $ \|V\|_{\s}^2\leq 2C_{r} \big( \|V\|_{\s,U}^2+\| V\|_{{-2}}^2 \big) $.
This proves the lower bound in  \eqref{equ}.
\end{proof}

\end{lemma}
\begin{proof}[{\bf Proof of Proposition \ref{exp}}]
The time derivative of the modified energy \eqref{enemod} along a solution 
$ V^\varepsilon(t) $ of \eqref{epsisol} is  
\begin{align}
\frac{d}{dt} \| V^\varepsilon\|_{\s,U(t)}^2  & =  \langle \Opbw{ \partial_t(\lambda^\s) |\xi|^{2\s}} \Opbw{F^{-1}} V^\varepsilon, \Opbw{F^{-1}}V^\varepsilon\rangle \label{d1}\\
&\ \ + 2\langle \Opbw{ \lambda^\s |\xi|^{2\s}} \Opbw{\partial_t F^{-1}} V^\varepsilon, \Opbw{F^{-1}}V^\varepsilon\rangle\label{d2}\\
&\ \ +2\langle \Opbw{ \lambda^\s |\xi|^{2\s}} \Opbw{F^{-1}} \partial_t V^\varepsilon, \Opbw{F^{-1}}V^\varepsilon\rangle\label{d4} \, . 
\end{align} 
By Theorem \ref{thm:contS}  and using that  
$\forall \sigma  \geq  0 $, {$ N \in \N_0 $},  
$$
\abs{\pa_t \lambda^\sigma(U) |\xi|^{2\s}}_{2\s, s_0, N} , \quad
\abs{\pa_t F^{-1}(U)}_{0, s_0, N} 
\leq \tC_N(\norm{U}_{s_0}, \norm{\pa_t U}_{s_0})
$$ 
and the assumption \eqref{U.ass}, there exists 
a constant $C_{\Theta}>0$ (depending also on $\s$) such that 
\begin{equation}\label{en1}
\eqref{d1}+\eqref{d2}\leq C_{\Theta} \| V^\varepsilon \|_{\s}^2 \, .
\end{equation}
{We now estimate \eqref{d4}.}
By Theorem \ref{thm:compS}  {with $ \varrho = 2 $ and  \eqref{U.ass}} we have
\begin{equation}\label{identita}
\Opbw{F}\Opbw{F^{-1}}= \uno+ \mathcal{F}_+^{-2}(U) \, , \quad \Opbw{F^{-1}}\Opbw{F}= \uno+\mathcal{F}_-^{-2}(U) \, ,
\end{equation}
where  $ \mathcal{F}_\pm^{-2}(U)$ are bounded operators  from $\dot \bH^{\sigma'}$ to $\dot \bH^{\sigma'+2}$,  $\forall \sigma' \in \R $, satisfying  
\begin{equation}\label{fatturo}
 \| \mathcal{F}_\pm^{-2}(U) W\|_{{\sigma'+2}}\leq C_{\Theta,\sigma'} \norm{W}_{\sigma'} , \quad \forall \, W \in \dot \bH^{\sigma'} \, . 
 \end{equation}
Thus, denoting $ \widetilde V^\varepsilon:= \Opbw{F^{-1}}V^\varepsilon $, 
by \eqref{identita}, we have 
\begin{equation}\label{ffide}
\Opbw{F}\widetilde V^\varepsilon= V^\varepsilon + \mathcal{F}_+^{-2}(U)V^\varepsilon \, .
\end{equation}
Recalling \eqref{epsisol} we have
\begin{align}
\notag
\eqref{d4}  &  =  2 \langle \Opbw{ \lambda^\s |\xi|^{2\s}} \Opbw{F^{-1}} \J \Opbw{A^\varepsilon} V^\varepsilon, \widetilde V^\varepsilon\rangle\\
\notag
 & \stackrel{\eqref{ffide}} = 2 \langle \Opbw{ \lambda^\s |\xi|^{2\s}} \Opbw{F^{-1}} \J \Opbw{A^\varepsilon} \Opbw{F} \widetilde V^\varepsilon, \widetilde V^\varepsilon\rangle\\
\notag
& \quad \ - 2 \langle \Opbw{ \lambda^\s |\xi|^{2\s}} \Opbw{F^{-1}} \J \Opbw{A^\varepsilon}\mathcal{F}_+^{-2} V^\varepsilon, \widetilde V^\varepsilon\rangle
\end{align}
and by Lemma \ref{parametrica} we get
\begin{align}
\eqref{d4} \stackrel{\eqref{Aconj}}{ =} & \langle \J \big[ \Opbw{ \lambda^\s |\xi|^{2\s}}, 
{\rm Op}^{BW} \big( {\sqrt{\eq K(\eq)}\lambda |\xi|^2 \chi_\varepsilon} \big) \big]
\widetilde V^\varepsilon, \widetilde V^\varepsilon\rangle\label{fin2}\\
&+ \langle \J 	\big[ \Opbw{ \lambda^\s |\xi|^{2\s}}, \Opbw{\tb\cdot \xi \, \chi_\varepsilon} \big]\widetilde V^\varepsilon, \widetilde V^\varepsilon\rangle\label{fin1}\\
&+2 \langle \Opbw{  \lambda^\s |\xi|^{2\s}}\mathcal{F} \widetilde V^\varepsilon, \widetilde V^\varepsilon\rangle\label{effone}\\
&  -2\langle \Opbw{ \lambda^\s |\xi|^{2\s}} \Opbw{F^{-1}} \J \Opbw{A^\varepsilon}\mathcal{F}_+^{-2} V^\varepsilon, \widetilde V^\varepsilon\rangle \label{bellissimo}
\end{align}
where in line \eqref{effone} the operator $\mathcal{F}(U)$ is the bounded remainder of Lemma \ref{parametrica}.
We estimate each contribution. 
First we consider line \eqref{fin2}. Using Theorem  \ref{thm:compS}  with 
$ \vr = 2 $, the principal symbol of the commutator is 
 $$
 {\im^{-1}} \big\{ 
 \lambda^\s |\xi|^{2\s},\sqrt{\eq K(\eq)}\lambda |\xi|^2\chi\big(\varepsilon \lambda |\xi|^2\big) 
 \big\} = 0 \, , 
 $$ 
and, {using \eqref{es:chie}, \eqref{simboli} and assumption \eqref{U.ass},} we get
\begin{equation}\label{en2}
| \eqref{fin2}| \leq 
{C_{\Theta}' \| \widetilde V^\varepsilon\|_{\s}^2} 
\leq C_{\Theta} \| V^\varepsilon\|_{\s}^2 \, .
\end{equation}
Similarly, using Theorem \ref{thm:compS} with $\vr=1$, Theorem \ref{thm:contS},  \eqref{simboli} and  estimates \eqref{fatturo} and  \eqref{matf}, we obtain 
\begin{equation}\label{en3}
|\eqref{fin1}|+ | \eqref{effone}|+ |\eqref{bellissimo}|
\leq C_{\Theta}\|V^\varepsilon\|_{\s}^2 \, .  
\end{equation}
In conclusion, by \eqref{en1}, \eqref{en2}, \eqref{en3}, we deduce the bound
$\frac{d}{dt} \| V^\varepsilon(t)\|_{\s,U(t)}^2 \leq C_{\Theta} \norm{V^\varepsilon(t)}_\s^2 $, 
that gives, for any $ t \in [0,T] $
\begin{align}
\|V^\varepsilon(t)\|_{\s,U(t)}^2 & \leq  \| V^\varepsilon(0)\|_{\s,U(0)}^2+ C_{\Theta} \int_0^t \|V^\varepsilon(\tau)\|_{\s}^2 \, \di \tau \nonumber \\
& \stackrel{\eqref{equ}}{\leq}  C_{r} \| V^\varepsilon(0)\|_{\s}^2+ C_{\Theta} \int_0^t \|V^\varepsilon(\tau)\|_{\s}^2 \, \di \tau \, . \label{en4}
\end{align}
Since $V^\varepsilon(t)$ solves  $\eqref{epsisol}$, by Theorem \ref{thm:contS},   \eqref{simboli}, \eqref{es:chie} there exists a constant $C_{\Theta}>0$ (independent on $\varepsilon$) such that 
$ \|\partial_t V^\varepsilon(t)\|_{{-2}}^2\leq C_{\Theta}\|V^\varepsilon(t)\|_{0}^2\leq C_{\Theta}\|V^\varepsilon(t)\|_{\s}^2 $ and therefore
\begin{equation} \label{booh}
\|V^\varepsilon(t)\|_{{-2}}^2\leq\|V^\varepsilon(0)\|_{{-2}}^2+C_{\Theta} \int_0^t \|V^\varepsilon(\tau)\|_{\s}^2 \, \di \tau \, , \quad \forall t \in [0,T] \, .
\end{equation}
We finally deduce \eqref{gron} by 
\eqref{en4},  the  lower bound in \eqref{equ} and  \eqref{booh}.
{The estimate \eqref{point} follows by Gronwall inequality}.
\end{proof}

\noindent
{\bf Proof of Proposition \ref{LinLWP}.}
By Proposition \ref{exp}, 
Ascoli-Arzel\'a theorem ensures that, for any $ \sigma \geq 0 $, 
$V^\varepsilon $ converges up to subsequence to a 
limit $V $ in $ C^1([0,T],\dot \bH^\s)$,  
as $\varepsilon \to 0$ that solves
\eqref{linpro} with $R(t)=0 $, initial datum $V_0\in \dot C^\infty $,  and satisfies 
$ \| V(t)\|_{{\s}}\leq C_r e^{C_\Theta t} \| V_0\|_{{\s}} $, for any $ \sigma \geq 0 $.
The case $V_0\in \dot \bH^\s$ follows by a classical approximation argument with smooth initial data. 
This shows that  the propagator of 
$\J \Opbw{A_2(U(t);x, \xi)+A_1(U(t);x, \xi)}$ is, for any 
$ \sigma \geq 0 $, a well defined bounded linear operator
$$ 
\Phi(t):\dot \bH^\s\mapsto \dot \bH^\s \, , 
\  V_0\mapsto \Phi(t)V_0:= V(t) \, , \ \forall t\in [0,T] \, ,
\ \ {\rm satisfying} \ \  \| \Phi (t) V_0 \|_\sigma \leq C_r e^{C_\Theta t} \| V_0 \|_\sigma  \, . 
$$
In the inhomogeneous case $R\not=0 $, the solutions of 
\eqref{linpro} is given by the  Duhamel formula
$ V(t)= \Phi(t)  V_0 + \Phi(t)\int_0^t \Phi^{-1}(\tau) R(\tau)\, \di \tau \, ,
$
and the estimate  \eqref{duaest} follows.

\subsection{Iterative scheme}

In order to prove that the {nonlinear system} \eqref{EKcf1}
has a {local in time} solution we 
consider the sequence of linear Cauchy problems 
\begin{equation*}
\cP_1 := 
\begin{cases}
\pa_t U_1 = - \J \sqrt{\eq K(\eq)} \, \Delta U_1  \\
U_1(0) =   U_0 \, , 
\end{cases}  
\cP_n := 
\begin{cases}
\pa_t U_n = \J \, 
\Opbw{  A(U_{n-1}; x, \xi)  }U_n  +  R(U_{n-1})\\
U_n(0) =  U_0 \, , 
\end{cases}
\end{equation*}
{for $ n \geq 2 $}, 
where $ A := A_2 + A_1 $, {cfr.  \eqref{A2}, \eqref{A1}}.
The strategy is to prove that the sequence  of solutions $U_n$ of the approximated problems $\cP_n $ {converges to a solution $U$ of system} \eqref{EKcf1}.

\begin{lemma}\label{ite}
 Let $ U_0\in  \dot \bH^s $, {$ s > 2 + \frac{d}{2}$,}  such that $ \rho (U_0) \in \cQ_\delta $  for some $\delta >0$ 
 (recall \eqref{defrho} and \eqref{def:Qdelta}) and define 
$r:= 2\| U_0\|_{{s_0}} $. 
Then there exists a time $T:=T(\| U_0\|_{{s_0+2}},\delta)>0$  such that, for any $n\in \N$:
\begin{itemize}
\item[$(S0)_n$:]  The problem $\cP_n$ admits a unique solution $U_n\in C^0([0,T], \dot \bH^s)\cap C^1 ([0,T], \dot \bH^{s-2})$.
\item[$(S1)_n$:] For any $ t\in [0,T]$,  $ \rho(U_n(t)) $  
belongs to $\cQ_{\frac{\delta}{2}}$.
\item[$(S2)_n$:] There exists a  constant ${C_r\geq 1}$  {(depending also on  $s$)} such that, defining $\Theta:=4C_r \|U_0 \|_{{s_0+2}}$ and $ M:= 4C_r \| U_0\|_{{s}}$,  for any $1\leq m\leq n$ one has 
\begin{align}
\label{S2n0}
& \| U_m\|_{L^\infty_T \dot \bH^{s_0}} \leq r 	\,  ;\\
\label{S2n}
&\|U_m\|_{L^{\infty}_T \dot \bH^{s_0+2}}\leq \Theta \, , \quad 
\|\partial_t U_m\|_{L^{\infty}_T \dot \bH^{s_0}}\leq C_r \Theta \, ;\\
\label{S2n1}
&\|U_m\|_{L^{\infty}_T \dot \bH^{s}}\leq M \, , \quad 
\|\partial_t U_m\|_{L^{\infty}_T \dot \bH^{s-2}}\leq C_r M \, .
\end{align}
\item[ $(S3)_n$:] For $1\leq m\leq n$ one has 
$$
\| U_1\|_{L^{\infty}_T\dot \bH^{s_0}} =   r / 2 \, , \qquad
\| U_m -U_{m-1}\|_{L^{\infty}_T\dot \bH^{s_0}}\leq 2^{-m} {r} \, ,  \ \ m \geq 2 \, . 
$$
\end{itemize}
\end{lemma}
\begin{proof}

We prove the statement by induction on $n\in \N$.
Given $r >0$,  we define  
$$
C_r := \max  \left\{ 1, C_{r, s_0}, \ C_{r, s_0+2}, \ C_{r,s}, \ 2 \, \tC(r) \right\} , 
$$
where $C_{r, \sigma}$ is the constant in Proposition \ref{LinLWP} (where we stress that it depends also on $\s$)  and $\tC(\cdot)$ is the function in \eqref{ba1} and \eqref{restoparalin0}. In the following  we shall denote by $ C_\Theta $ all the constants depending on $\Theta$, which can vary from line to line. \\
\noindent{\bf Proof of $(S0)_1$:} The problem  
$\cP_1$ admits a unique global solution which 
preserves Sobolev norms.
\\[1mm]
\noindent{\bf Proof of $(S1)_1$:}  
We have $\rho(U_0) \in \cQ_{\delta}$. In addition 
$$ 
 \| \rho(U_1(t) - U_0) \|_{L^\infty(\T^d)}
\lesssim  \| U_1(t)- U_0\|_{{s_0}}\lesssim T \|  U_0\|_{s_0+2} \leq \delta/2 
$$
for {$T := T(\|  U_0\|_{s_0+2}, \delta)>0$} sufficiently small, which implies $\rho(U_1(t)) \in \cQ_{\frac{\delta}{2}}$, for any $ t \in [0, T]$.

\noindent{\bf Proof of $(S2)_1$ and $(S3)_1$:} 
{The flow of $\cP_1$ 
is an isometry} and $M\geq \| U_0\|_{ s}$, $ \Theta \geq \|  U_0\|_{{s_0+2}}$.

Suppose that $(S0)_{n-1}$--$(S3)_{n-1}$ hold true. We prove $(S0)_{n}$--$(S3)_{n}$.

\noindent{\bf Proof of $(S0)_n$:} 
We apply Proposition \ref{LinLWP} with $\sigma = s $, 
$ U \leadsto U_{n-1} $ and
$R(t) := R(U_{n-1}(t))$.
By $(S1)_{n-1}$ and  $(S2)_{n-1}$, 
the function $U_{n-1}$  satisfies assumption 
 \eqref{U.ass} with $\Theta \leadsto (1+C_r)\Theta$. In addition
$ R(U_{n-1}(t))$ belongs to $C^0([0, T], \dot \bH^{s})$ thanks to  
\eqref{restoparalin1} and  $ U_{n-1} \in C^0
([0,T]; \dot \bH^s)$.  
Thus Proposition \ref{LinLWP} with $\sigma = s$ implies  $(S0)_n $.
  In particular $ U_n $ satisfies the estimate \eqref{duaest}.\\
\noindent{\bf Proof of $(S2)_n$:}   
We first prove \eqref{S2n}. The  estimate  \eqref{duaest} with $\s = s_0+2$, 
the bound \eqref{restoparalin0} of $R(U_{n-1}(t))$ 
 and \eqref{S2n} at the  step $  n - 1 $, imply
\begin{equation}\label{Unup} 
\| U_n\|_{L^\infty_T \dot \bH^{s_0+2}}\leq C_{r} e^{C_{\Theta}T} \|  U_0\|_{{s_0+2}} + T C_{\Theta}  e^{C_\Theta T}  \Theta \, .
\end{equation}
As $\Theta= 4C_r \|U_0\|_{{s_0+2}} $, we take  $T>0$  small such that 
\begin{equation}\label{tione}
C_{\Theta}T\leq 1 \, , \quad 
T C_{\Theta} e^{C_\Theta T}  \leq 1 / 4  \, , 
\end{equation}
which, by \eqref{Unup}, gives 
$\| U_n\|_{L^\infty_T\dot \bH^{s_0+2}}\leq \Theta $. 
This proves the first estimate of \eqref{S2n}. Regarding the control of $\partial_t U_n$, we use the equation $\mathcal{P}_n$,
the second estimate in \eqref{restoparalin0}  and \eqref{ba1}  with $\s=s_0$
to obtain 
\begin{equation}\label{megliodiprima}
\| \partial_t U_n(t)\|_{{s_0}}\leq \tC\left(\| U_{n-1}(t)\|_{{s_0}}\right) \| U_n(t)\|_{{s_0+2}} 
+  \tC\left(\| U_{n-1}(t)\|_{{s_0}}\right) \| U_{n-1}(t)\|_{{s_0+2}}
\leq C_r\Theta \end{equation}
which proves  the second estimate of \eqref{S2n}.
\\
Next we prove \eqref{S2n1}. 
Applying  estimate \eqref{duaest}  with $\s = s$,  we have 
\begin{equation*} 
\| U_n\|_{L^\infty_T \dot \bH^{s}}\leq C_{r} e^{C_{\Theta}T} \|  U_0\|_{s} + T C_{\Theta}  e^{C_\Theta T}  M \leq M 
\end{equation*}
for  $ M= 4C_r \| U_0\|_{s} $ and since $T>0$ is chosen as in \eqref{tione}. The estimate for $\norm{\pa_t U_n}_{s-2}$ is similar to \eqref{megliodiprima}, and we omit it.
Estimate \eqref{S2n0} is a consequence of $(S3)_n$, which we prove below. 
\\
\noindent{\bf Proof of $(S1)_n$:} We use   estimate \eqref{megliodiprima} to get 
\begin{align*}
\| \rho(U_n(t) - U_0) \|_{L^\infty(\T^d)}  \leq C  \| U_n(t)-  U_0\|_{{s_0}}&\leq C \int_0^T \| \partial_t U_n(t)\|_{{s_0}}\, \di t  \leq C\, C_r\, T\, \Theta\leq \delta / 2 
 \end{align*}
provided that $T< \delta/{(2 C C_r \Theta)}$. This shows that $\rho(U_n(t)) \in \cQ_{\frac{\delta}{2}}$.

\noindent{\bf Proof of $(S3)_n$:}  Define $V_n:= U_n-U_{n-1}$ {if $ n \geq 2 $ 
and $ V_1 := U_1 $}. 
 Note that $ V_n$,  $ n \geq 2 $, solves 
\begin{equation}\label{Vneq} 
\partial_t V_n= \J \Opbw{A(U_{n-1})} V_n + f_n \, , \quad V_n(0)=0 \, , 
\end{equation}
where $A:= A_2+A_1$ and 
$$ 
\begin{aligned}
f_n & := \J \Opbw{A(U_{n-1})- A(U_{n-2})} U_{n-1}+ R(U_{n-1}) - R(U_{n-2}) \, ,    \
{\rm for}  \ n > 2 \, , \\
f_2 & := 
\J {\rm Op}^{BW}\big({A(U_{1})-  \sqrt{\eq K(\eq)}|\xi|^2} \big) U_{1}+ R(U_{1})  \, .
\end{aligned} 
$$  
 Applying estimates \eqref{restoparalin2}, \eqref{la1}, \eqref{restoparalin0} and \eqref{S2n}
 we obtain, for $ n \geq 2 $,  
  \begin{equation}\label{fns0}
   \| f_n \|_{{s_0}}  \leq C_\Theta    \| V_{n-1}\|_{s_0} \, , \quad \forall t \in [0,T] \, . 
   \end{equation}
  We apply Proposition \ref{LinLWP} to \eqref{Vneq} 
with $\s=s_0$. Thus by \eqref{duaest} and \eqref{fns0} we get
 $$ 
 \begin{aligned}
 \| V_n\|_{L^\infty_T \dot \bH^{s_0}} 
& \leq 
 C_\Theta e^{C_\Theta T } T 
 \| f_n\|_{L^\infty_T \dot \bH^{s_0}}
\leq  
 C_\Theta e^{C_\Theta T } T   \| V_{n-1}\|_{L^\infty_T \dot \bH^{s_0}} \leq \frac{1}{2} \| V_{n-1}\|_{L^\infty_T \dot \bH^{s_0}} 
 \end{aligned}
 $$
provided  $ C_\Theta e^{C_\Theta T} T\leq \frac{1}{2}$.
The proof of Lemma \ref{ite} is complete. 
  \end{proof}
  
  \begin{corollary} \label{corobello}
  With the same assumptions of Lemma \ref{ite}, for any $s_0+2 \leq  s'<s$:
  \begin{itemize}
  \item[(i)]
  $(U_n)_{n \geq 1}$ is a Cauchy sequence in  $C^0([0,T],\dot \bH^{s'})\cap C^1([0,T],\dot \bH^{s'-2})$ with $ T = T(\| U_0\|_{{s_0+2}},\delta) $  given by 
  Lemma \ref{ite}. It converges
  to the unique solution $U(t) $  of  \eqref{EKcf1} with initial datum $U_0$, 
$ U(t) $ is in  $  C^0([0,T],\dot \bH^{s'})\cap C^1([0,T],\dot \bH^{s'-2}) $.
  Moreover $\rho(U(t)) \in \cQ$, $ \forall t \in [0, T]$.
  \item[(ii)] For any $t\in [0,T]$, $U(t) \in \dot \bH^s $ and $\norm{U(t)}_s \leq 4 C_r \norm{U_0}_s$ where $C_r$ is the constant of $(S2)_n$.
  \end{itemize}
    \end{corollary}
  \begin{proof}
$(i)$ If $s' = s_0$ it is the content of $(S3)_n$. For $ s'  \in (s_0,  s)$, we use interpolation estimate \eqref{interp4},  \eqref{S2n1} and $(S3)_n$ to get, for $ n \geq 2 $, 
  $$ 
  \| U_n-U_{n-1}\|_{L^\infty_T \dot \bH^{s'}} \leq 
  \| U_n-U_{n-1}\|_{L^\infty_T \dot \bH^{s_0}}^{\theta}
  \| U_n-U_{n-1}\|_{L^\infty_T \dot\bH^{s}}^{1-\theta} \leq 2^{-n\theta} C_M \, ,
  $$
where $\theta \in (0,1) $ is chosen so that  $s' = \theta s_0 + (1-\theta)s$.  Thus  $(U_n)_{n \geq 1}$ is a Cauchy sequence in $C^0([0,T], \dot \bH^{s'})$;
 we denote by  $U(t)\in C^0([0,T],\dot \bH^{s'})$ its  limit.
  Similarly using that $ \partial_t U_n $ solves $ \cP_n $, one proves that 
  $\partial_tU_n$ is a Cauchy sequence in $C^0([0,T],\dot \bH^{s'-2})$ that converges to $\partial_t U$ in $C^0([0,T],\dot \bH^{s'-2})$.  
  In order to prove that $ U ( t ) $ solves \eqref{EKcf1},  it is enough to show that 
  \begin{align*}
 & \cR(U, U_{n-1}, U_n) := \J\Opbw{A(U_{n-1})}U_n -  \J\Opbw{A(U)}U
 + R(U_{n-1}) - R(U) 
  \end{align*}
 converges to $ 0 $ in $ L^\infty_T \dot \bH^{s'-2} $. 
This holds true because by estimates \eqref{cont00},  \eqref{la1}, \eqref{restoparalin1}, 
 $ (S2)_n $,  and the fact that $U(t) \in C^0([0,T],\dot \bH^{s'})$,   we have 
 $$
  \|  \cR(U, U_{n-1}, U_n) \|_{L^\infty_T \dot \bH^{s'-2}}\leq C_M \left( \| U-U_{n-1}\|_{L^\infty_T \dot \bH^{s'}}+ \norm{U-U_{n-1}}_{L^\infty_T \dot \bH^{s_0+2}} + \| U-U_n\|_{L^\infty_T \dot \bH^{s'}} \right)
  $$
which converges to $0$ as $n \to \infty$.

Let us now prove the uniqueness.
 Suppose that $V_1,V_2\in C^0([0,T],\dot \bH^{s'})\cap C^1([0,T],\dot \bH^{s'-2})$ are solutions of \eqref{EKcf1} with initial datum $U_0$. 
 Then $W:= V_1-V_2$ solves 
 \begin{align*}
 \partial_tW  = \J \Opbw{A(V_1)}W + \bold{R}(t) \, ,  \quad W(0)=0 \, ,
 \end{align*}
 where
 $
 \bold R(t):=\J\Opbw{A(V_1)-A(V_2)}V_2+ R(V_1)-R(V_2) .
 $ Applying Proposition \ref{LinLWP} with $\sigma = s_0$ and $\Theta, r$ defined by 
  $$
  \Theta := \max_{j=1,2}\big( \| V_j\|_{L^\infty_T \dot \bH^{s_0+2}} +
  \|\partial_t V_j\|_{L^\infty_T \dot \bH^{s_0}}
  \big) \, ,  \quad r := \max_{j = 1,2} 
   \| V_j\|_{L^\infty_T \dot \bH^{s_0}} \, , 
 $$
 together with estimates \eqref{restoparalin2}  and \eqref{la1}   we have, 
 for any $ t \in [0,T] $, 
 \begin{align*}
  \| W\|_{{L^\infty_t \dot \bH^{s_0}}} & \leq  C_\Theta e^{C_\Theta t }t \| \bold R\|_{L^\infty_t \dot \bH^{s_0}} 
  \leq  C_\Theta e^{C_\Theta t }t  \| W\|_{L^\infty_t \dot \bH^{s_0}} \, .
 \end{align*}
Therefore, provided $t$ is so small that  $C_\Theta e^{C_\Theta t }t< 1$, we get  $V_1(\tau)=V_2(\tau)$ $\, \forall \tau \in [0,t]$. As \eqref{EKcf1} is autonomous, actually one has   $V_1(t)=V_2(t)$ for all $t\in [0,T]$. This proves the uniqueness.\\
 Finally, as $\rho(U_n(t)) \in \cQ_{\frac{\delta}{2}}$ and  $U_n(t) \to U(t)$ in $\dot \bH^{s_0}$, then $\rho(U(t)) \in  \cQ_{\frac{\delta}{2}} \subset \cQ$.
 
 $(ii)$   
Since   $ \| U_n(t)\|_s \leq  4 C_r \norm{U_0}_s $ and 
  $U_n(t) \rightarrow U(t)$ 
   in $\dot \bH^{s'}$ 
   then 
  $\| U (t)\|_s \leq  4 C_r \norm{U_0}_s  $.
  \end{proof}

\noindent
Let
  $\displaystyle{
  \Pi_{N} U := \Big({\sum_{1\leq |j|\leq N} {u}_j e^{\im j\cdot  x}},{ \sum_{1\leq |j|\leq N} \bar{{ u}_{j}} e^{-\im j \cdot x}}\Big)}
   $.
    We need below the following technical lemma.
\begin{lemma}
\label{lem:UNtoU}
Let  $ U_0\in \dot \bH^s$, $ s > 2 + \frac{d}{2} $, with $\rho(U_0) \in \cQ_\delta$ for some $\delta >0$. 
Then there exists a time $\tilde T:= \tilde T( \| U_0\|_{s_0+2}, \delta)>0$ and $N_0 >0$ such that for any $N > N_0$:
\begin{itemize}
\item[(i)] system \eqref{EKcf1} with initial datum $ 
\Pi_{N}  U_0$ has a unique solution $U_N \in C^0([0, \tilde T], \dot \bH^{s+2})$.
\item[(ii)] Let 
$U$ be the unique solution of  \eqref{EKcf1} with initial datum $U_{0}$ defined in the 
time interval $
[0,T] $ (which exists by Corollary \ref{corobello}). Then there is
$ \breve T < \min\{T, \tilde T \} $, depending on $ \| U_0 \|_s $, independent of $ N $,  such that
\begin{equation}
\label{over0}
\norm{U - U_N}_{ L^\infty_{\breve T} \dot \bH^s} \leq \tC (\norm{U_0}_s) \,  \left( \norm{U_0 -  \Pi_{N} U_{0}}_s + N^{s_0 +2 -s} \right).
\end{equation}
In particular $U_N \to U$ in $C^0([0, \breve T], \dot \bH^{s})$ when $N \to \infty$.
\end{itemize}
\end{lemma}
\begin{proof}
Clearly $ 
\Pi_{N}  U_0  \in \dot C^\infty$. Moreover,  as
$ \| \rho(U_0 - \Pi_{N}  U_0) \|_{L^\infty(\T^d)} \to 0 $ when  $N \to \infty$,  one has  
$\rho(\Pi_{N}  U_0) \in \cQ_{\frac{\delta}{2}}$ provided 
  $N \geq N_0 $ is sufficiently large. 
So we can apply  Corollary \ref{corobello} and obtain 
a time $\tilde T > 0 $, independent on $N$, and  a unique solution   $U_N \in C^0([0, \tilde T], \dot \bH^{s+2}) $ of \eqref{EKcf1} with initial datum $ \Pi_{N}  U_0 $. Moreover, by item $(ii)$ of that corollary, setting $ r = 2 \norm{\Pi_{N}  U_0}_{s_0} $,
  \begin{align}
  \label{over00}
&    \| U_N\|_{{L^\infty_{\tilde T}} \dot  \bH^{s}}\leq 4
{C_r   \| \Pi_{N}  U_0 \|_{{s}}} \leq \tC(\norm{U_{0}}_{s_0})  \, \| U_0\|_{{s}} \, ,
\\
\label{over}
 &  \| U_N\|_{{L^\infty_{\tilde T}} \dot  \bH^{s+2}}
 \leq 4 {C_r \| \Pi_{N}  U_0 \|_{{s+2}}}
\leq 
\tC(\norm{U_{0}}_{s_0}) \, N^2 \, \| U_0\|_{{s}} \, .
   \end{align}
  This proves item $(i)$. 
  In the following let  $ \Tb \leq \min \{ \tilde T, T \} $. \\
Let us prove $(ii)$.
Let  $ \Theta :=  \norm{U}_{L^\infty_\Tb \dot \bH^{s_0+2} } +
\norm{\pa_t U}_{L^\infty_\Tb \dot \bH^{s_0}  }$ 
 and $ r := \| U\|_{L^\infty_\Tb \dot \bH^{s_0}} $.
The function $W_N(t):=U(t)-U_N(t)$ satisfies
$
\norm{W_N(t)}_s \leq 
\norm{U(t)}_s + \norm{ U_N(t)}_s 
\leq  \tC(\norm{U_{0}}_{s}) 
$, 
$\, \forall t \in [0,  \Tb]$, {by Corollary \ref{corobello}-($ii$)}.
Moreover,  $W_N$ solves 
 \begin{align*}
 \partial_t W_N & = \J \Opbw{A(U)} W_N  
 +\bold R(t) \, ,  \quad W_N(0) = U_0- \Pi_{N}  U_0 
 \end{align*}
 where
 $
 \bold R(t):= \J \Opbw{A(U)-A(U_N)} U_N + R(U)-R(U_N).
 $
Applying Proposition \ref{LinLWP} {with $ \sigma = s_0 $} and  
  estimates  \eqref{la1}, \eqref{restoparalin2}, \eqref{over00}
one obtains 
 \begin{align*}
 \| W_N\|_{L^\infty_{ \Tb} \dot \bH^{s_0 }}
 &\leq  C_r e^{C_\Theta \Tb} \| U_0 - \Pi_{N}  U_0 \|_{{s_0}}
 +
 \Tb  C_\Theta e^{C_\Theta \Tb}\, \tC(\| {U_0} \|_{s_0+2}) \norm{W_N}_{L^\infty_\Tb \dot \bH^{s_0}},
 \end{align*}
which,  provided  $\Tb $  is so small that $ \Tb  C_\Theta e^{C_\Theta \Tb}\, \tC(\| U_0\|_{s_0+2}) \leq \frac12$ (eventually shrinking it),  gives 
 \begin{equation}
  \label{over1}
 \| W_N\|_{{L^\infty_\Tb \dot \bH^{s_0}}} \leq  C_r \| U_0- \Pi_{N}  U_0 \|_{{s_0}} \leq  C_r \,  N^{s_0 - s} \norm{U_0}_s \, . 
  \end{equation} 
Similarly one  estimates $\norm{W_N(t)}_{\dot \bH^s}$,  getting 
 \begin{align*}
&  \| W_N\|_{{L^\infty_\Tb \dot \bH^s }} \\
 &\leq  C_r e^{C_\Theta \Tb}\| U_0- \Pi_{N}  U_0 \|_s
 +
  C_\Theta
   e^{C_\Theta \Tb} \Tb  \tC\big( \| U_0\|_{s} \big) \, 
   \left(\| W_N\|_{L^\infty_\Tb \dot \bH^{s_0}} \| {U_N}\|_{L^\infty_\Tb \dot \bH^{s+2}} + 
 \| W_N\|_{L^\infty_\Tb \dot \bH^s}\right)\\
& \stackrel{ \eqref{over}, \eqref{over1}}{\leq}
  C_r e^{C_\Theta \Tb} \| U_0- \Pi_{N}  U_0\|_{{s}}
  +
   C_\Theta
   e^{C_\Theta \Tb} \Tb  \tC(\| U_0\|_{s})   \big( N^{s_0-s +2}  +  \| W_N\|_{L^\infty_\Tb \dot \bH^s}\big)
 \end{align*}
from which  \eqref{over0} follows provided $ \Tb $ 
(depending on $\| U_0\|_{s}$) is sufficiently small.
\end{proof}

  \begin{proof}[{\bf Proof of Proposition \ref{prop:LWP}:}] Given an initial datum $ U_0\in \dot \bH^s$ with $\rho(U_0) \in \cQ$, choose $\delta >0$ so small that  
  $ \rho(U_0) \in \cQ_\delta$. Then 
Corollary   \ref{corobello} gives us a time $T = T(\norm{U_0}_{s_0+2}, \delta ) > 0 $ and a unique solution   $U \in C^0([0,T],\dot \bH^{s'})\cap C^1([0,T],\dot \bH^{s'-2})$, $\forall s_0+2 \leq s' < s$,  of \eqref{EKcf1}  with initial datum $U_0$.
Now take an  open neighborhood $\cU\subset \dot \bH^{s}$ of $U_0$ such that 
$\forall V \in \cU$ one has $\rho(V) \in \cQ_{\frac{\delta}{2}} $ and $\norm{V}_s \leq 2 \norm{U_0}_s$. Then  there exists $\tilde T \in (0, T] $ such that 
 the flow map  of \eqref{EKcf1}, 
 $$ 
 \Omega^t:   \cU \to \dot \bH^s \cap \big\{ 
 U \in \dot \bH^s:  \, \rho(U)\in \cQ_{\frac{\delta}{4}} \big\}  \, , \quad
 U_0 \mapsto \Omega^t(U_0):= U(t) \, , 
 $$
 is well defined for any   $t \in [0, \tilde T]$, it satisfies the group property
  \begin{equation}\label{auto}
 \Omega^{t+\tau}=\Omega^t\circ \Omega^\tau \, , \quad \forall t,\tau, t+\tau \in [0,\tilde T] \, , 
 \end{equation}
  and 
$  \norm{\Omega^t(U_0)}_s \leq  \tC(\norm{U_0}_s) $ for all $ U_0 \in \cU $, 
$  t\in [0, \tilde T] $. 
For simplicity of notation in the sequel we denote by $ T $ a time, independent of $ N $,
smaller than $ \tilde T $.

\noindent{\bf Continuity of $t \mapsto U(t)$:} We show that  $ U	\in C^0([0,T], \dot \bH^s )$.
By \eqref{auto}, it is enough to prove that $t\mapsto  U(t)$ is continuous in a neighborhood of $t=0$.
This follows by Lemma \ref{lem:UNtoU}, as $U$ is the uniform limit of continuous functions.\\
\noindent {\bf Continuity of the flow map:} 
We shall follow the method by \cite{BS,BDD}.
Let $ U^n_0 \to   U_0\in \dot \bH^s$ and pick $\delta >0$ such that  $\rho(U^n_0)$, $\rho(U_0) $,
 $\rho(\Pi_{N} U^n_0)$, $\rho(\Pi_{N} U_0) \in \cQ_{\delta}$,
for any $  n \geq n_0 $, 
$ N \geq N_0  $ sufficiently large. 
 Denote by $U^n,U \in C^0([0,T], \dot \bH^s)$ the  solutions of
 \eqref{EKcf1} with initial data $ U^n_0 $, respectively $ U_0 $, and
$ U_N(t):= \Omega^t(\Pi_{ N} U_0)$, $   U^n_N(t):= \Omega^t(\Pi_{ N} U^n_0)$. 
Note that these solutions  are well defined in $\dot \bH^s$ up to a common time $T' \in (0,T]$,
depending on $ \| U_0 \|_s $,  thanks to  Lemma \ref{lem:UNtoU}.
By triangular inequality we have,  by \eqref{over0},  for any $ n \geq n_0 $, $ N \geq N_0 $,  
\begin{align}
& \| U^n-U\|_{L^\infty_T \dot \bH^s} \leq  \| U^n- U^n_N\|_{L^\infty_T \dot \bH^s}+ \| U^n_N - U_N\|_{L^\infty_T \dot \bH^s}+ \| U_N-U\|_{L^\infty_T \dot \bH^s} \label{disun1} \\
& \leq \tC (\norm{U_0}_s) \, 
 \big(\| (\uno - \Pi_{ N}) U_0^n\|_s + \| (\uno - \Pi_{ N}) U_0 \|_s+  N^{s_0+2-s}\big)+ \| U^n_N - U_N\|_{L^\infty_T \dot \bH^s} \, \nonumber \\ 
& \leq \tC (\norm{U_0}_s) \, 
 \big( \| (\uno - \Pi_{ N}) U_0 \|_s+  N^{s_0+2-s}\big) +
 \tC (\norm{U_0}_s) \, \|  U_0^n - U_0\|_s 
 + \| U^n_N - U_N\|_{L^\infty_T \dot \bH^s} 
 \,  . \nonumber
\end{align}
 For any $\varepsilon > 0 $, 
since 
$s>s_0+2$,  there exists $N_\varepsilon \in \N$  (independent of $ n$) such that 
\begin{equation} \label{disun2}
 \tC (\norm{U_0}_s)
\big(  \| (\uno - \Pi_{ N_\varepsilon}) U_0 \|_s+  N_\varepsilon^{s_0+2-s} \big) \leq \varepsilon / 2  \, . 
\end{equation}
Consider now  the term  $ \| U^n_N - U_N\|_{L^\infty_T \dot \bH^s} $.  
As $\Pi_{ N} U_0, \Pi_{ N} U_0^n \in \dot C^\infty$, the solutions $U_N(t)$, $U_N^n(t) $ belong actually to $ \dot \bH^{s+2}$. 
 By interpolation and by item $(ii)$ of Corollary \ref{corobello} applied with $ s \leadsto s+2$ one has, for $ s+ 2 = \theta s_0 + (1-\theta) (s+2) $, 
\begin{align}
\| U^n_{N_\varepsilon} - U_{N_\varepsilon}\|_{L^\infty_T \dot \bH^s}
& \leq \tC \big( \| \Pi_{N_\varepsilon}U_{0}\|_{s+2},\| \Pi_{N_\varepsilon} U^n_{0} 
\|_{s+2}\big) \| U^n_{N_\varepsilon} - U_{N_\varepsilon}\|_{L^\infty_T \dot 
\bH^{s_0}}^{\theta} \nonumber \\
& \leq\tC \big( N_\varepsilon^2\| U_{0}\|_{s}\big) \| U^n_{N_\varepsilon} - U_{N_\varepsilon}\|_{L^\infty_T \dot \bH^{s_0}}^{\theta} \, .  \label{disun3}
\end{align}
Arguing in the same way of the proof of \eqref{over1} one obtains 
\begin{equation} \label{disun4}
\| U^n_{N_\varepsilon} - U_{N_\varepsilon}\|_{L^\infty_T \dot \bH^{s_0}}\leq 
\tC \big( \| U_0\|_{s_0+2}\big) \| \Pi_{ N_\varepsilon} (U^n_{ 0}- U_{ 0})\|_{s_0} \, .
\end{equation}
By \eqref{disun1}-\eqref{disun4}, we have 
$ \limsup_{n\to \infty} \| U^n-U\|_{L^\infty_T \dot \bH^s}\leq \varepsilon $, 
$ \forall \varepsilon\in (0,1).$
\end{proof}

\appendix

\section{Bony-Weyl calculus in periodic H\"older spaces}
\label{app:par}

In this Appendix we develop in a self-contained manner paradifferential calculus for  
space periodic symbols $ a(x, \xi) $ which 
belong to the Banach scale of H\"older  spaces $ W^{\varrho,\infty} (\T^d) $.
The main results are the continuity Theorem \ref{thm:cont2} and the 
composition Theorem \ref{thm:comp2},
which require mild regularity assumptions of the symbols 
in the space variable,
and imply 
Theorems \ref{thm:contS} and \ref{thm:compS}. 
We first  provide some preliminary technical results.  

\paragraph{Technical lemmas.}
In the following we  denote by
$\pa_m$, $m= 1, \ldots, d$ the {\em discrete derivative}, defined for functions $f \colon \Z^d \to \C$ as
\begin{equation}\label{discrete-der}
(\pa_m f)(n) := f(n) - f(n - \vec{e}_m) \, , \quad n \in \Z^d \, ,
\end{equation}
where $ \vec{e}_m $ denotes the usual unit basis vector of $ \N_0^d $ 
with $ 0 $ components expect the
$ m$-th one. 
Given  a multi-index $\beta \in \N^d_0 $, we set
$ \pa^\beta f := \pa_1^{\beta_1} \cdots \pa_d^{\beta_d} f  $. 

We shall use the Leibniz rule for finite differences in the following form: 
given $ k \in \N $, $ m = 1, \ldots, d $, there exist constants 
$ C_{k_1,k_2} $ (binomial coefficients) such that  
\begin{equation}\label{finitediff}
(\pa_m^k) (f g) (n) = \sum_{k_1+k_2= k} 
C_{k_1,k_2}
(\pa_m^{k_1} f) (n- k_2) (\pa_m^{k_2} g) (n) \, . 
\end{equation}
Moreover, when using discrete derivatives, the analogous of the integration by parts formula is given by the 
{\em Abel resummation formula}:
\begin{align}
\label{Abeld}
\sum_{n \in \Z^d} e^{\ii n \cdot z} \beta (x,n) = 
- \frac{1}{e^{\ii \vec e_m \cdot z}-1} 
\sum_{n \in \Z^d} e^{\ii n \cdot z} (\pa_m \beta) (x,n)  \, 
, \quad \forall \, m = 1, \ldots, d \, . 
\end{align}

\begin{lemma}
\label{lem:K}
 Let $K \colon \T^d \to \C$ be a function satisfying, 
for  constants $A $ and $ B   $, the estimate   
\begin{equation}\label{ipok}
\abs{K(y)} \lesssim A^d B \min\left( 1, \min_{1\leq m \leq d}	\frac{1}{ |A \,  2 \sin \frac{y_m}{2}|^{(d+1)}}\right) \, , \quad \forall y \in \T^d \, .  
\end{equation}
Then   
\begin{equation}\label{okok}
\int_{\T^d} \abs{K(y)} \di y \lesssim  B \, .
\end{equation}
\end{lemma}
\begin{proof}
If $ A \leq 1 $ the bound \eqref{okok} follows trivially integrating the first inequality in
\eqref{ipok}. Then we suppose $ A > 1 $. 
We split the integral in \eqref{okok} as  
\begin{equation}
\label{okok1}
\int_{\T^d} \abs{K(y)} \di y= \int\limits_{\T^d\cap \{|y|\leq \frac{1}{A}\}} \abs{K(y)} \di y+\int\limits_{\T^d\cap \{ |y| > \frac{1}{A}\}} \abs{K(y)} \di y \, .
\end{equation}
We bound the first integral using the first inequality in \eqref{ipok}, getting 
\begin{equation}\label{pe1}
\int\limits_{\T^d\cap \{|y|\leq \frac{1}{A}\}} \abs{K(y)} \di y\lesssim A^dB \, 
{\rm meas}\left(y \in [-\pi,\pi]^d \colon \ {|y|\leq \frac{1}{A}} \right) \lesssim B \, .
\end{equation}
To bound the second integral in \eqref{okok1} we use that, for some $c >0$,  
$ \max_{1\leq m \leq d} \big| \sin \big(\frac{y_m}{2}\big) \big| \geq c|y| $, 
$ \forall y\in [-\pi,\pi]^d $,  
and therefore the second inequality in \eqref{ipok} implies 
\begin{equation}\label{pe2} 
\int\limits_{\T^d\cap \{ |y| > \frac{1}{A}\}} \abs{K(y)} \di y\lesssim 
A^dB\int\limits_{\{y \in \R^d \, : \,  |y|>\frac{1}{A}\}}  \frac{\di y}{|Ay|^{d+1}}\stackrel{z=Ay}
{\lesssim} B\int\limits_{ \{|z|>1\}} 
\frac{\di z}{|z|^{d+1}}  \lesssim B \, .
\end{equation}
The bounds \eqref{pe1}-\eqref{pe2} and \eqref{okok1} imply \eqref{okok}. 
\end{proof}

The next lemma  represents a Fourier multiplier operator 
acting on periodic functions as a convolution integral on $ \R^d $. 
The key step is the use of Poisson summation formula. 

\begin{lemma}\label{lem:FC}
Let $ \chi \in {\cal S} (\R^d) $. 
Then the Fourier multiplier
$ \chi_\theta ( D)  :=  \chi ( \theta^{-1} D) $, $ \theta \geq 1 $,
 acting on a periodic function 
$ u \in L^1 (\T^d, \C) $ can represented by
\begin{align}
 \chi_\theta ( D) u    
& =
 \int_{\R^d} u(y) \psi_{\theta} ( x- y  ) \di y = 
  \int_{\R^d} u(x-y) \psi_{\theta} ( y  ) \di y  \label{laco}
\end{align}
where 
$ \psi_{\theta} (z) := \theta^{d} \, \psi ( \theta z) $ and 
 $ \psi $ denotes the anti-Fourier transform of $ \chi $ on $\R^d $. 
\end{lemma}

\begin{proof}
For $ \theta \geq 1  $ we  write
\begin{equation}\label{opev}
\chi \left( \frac{D}{\theta}\right)   u  = 
\int_{\T^d} u(y) \,  h_{\theta} (x-y) \di y \qquad {\rm where} \qquad 
h_\theta (z) :=\frac{1}{(2\pi)^d} \sum_{j \in \Z^d} \chi \left( \frac{j}{\theta}\right) e^{\ii j \cdot z} \, . 
\end{equation} 
Then the Fourier transform 
$ \wh {\psi_\theta} (\xi) =  
\int_{\R^d} \, \theta^d \, \psi (\theta z) \, e^{- \ii z \cdot\xi} \di z = 
\int_{\R^d}  \psi (y)\,  e^{- \ii y \cdot \frac{\xi}{\theta}} \di y  =
\wh {\psi} \left(\frac{\xi}{\theta} \right) =  \chi \left(\frac{\xi}{\theta} \right) $, 
and, using {\it Poisson summation formula}, we 
write the periodic function $ h_\theta (z) $ in \eqref{opev} as 
$$
h_\theta (z)  = \frac{1}{(2\pi)^d}\sum_{j \in \Z^d} \wh {\psi_\theta} (j)  e^{\ii j \cdot z} = \sum_{j \in \Z^d}
\psi_\theta ( z + 2 \pi j ) \, . 
$$
Therefore the integral \eqref{opev} is 
\begin{align*}
\chi (\theta^{-1} D) u   &  = 
\sum_{j \in \Z^d} \int_{\T^d} u(y) \, \psi_\theta ( x- y + 2 \pi j ) \di y \nonumber 
 = \sum_{j \in \Z^d} \int_{[0, 2\pi]^d +  2 \pi j} u(y) \psi_\theta ( x- y  ) \di y \nonumber	\\	
& =
 \int_{\R^d} u(y) \psi_\theta ( x- y  ) \di y = 
  \int_{\R^d} u(x-y) \psi_\theta ( y  ) \di y  
\end{align*}
proving \eqref{laco}.
\end{proof}

We now give the definition and basic properties of the H\"older spaces
$ W^{\vr, \infty}  (\T^d )$.

\begin{definition}{\bf (Periodic H\"older spaces)}\label{sec:holder}
Given $\vr \in \N_0$, we denote by $ W^{\vr, \infty}  (\T^d ) $ the space of 
continuous functions $ u : \T^d \to \C $, $ 2 \pi$-periodic in each variable
$ (x_1, \ldots, x_d) $, 
whose derivatives of order  $\vr$ are in $L^\infty$, equipped with the norm 
$ \| u \|_{W^{\varrho,\infty}} := 
\sum_{|\alpha| \leq \vr} \| \pa_x^\alpha u \|_{L^\infty} $,  
$ \alpha \in \N_0^d $.
In case  $ \varrho >0$,  $\vr \notin \N$,  we denote  $ \floor{\varrho} $ the integer part of 
$ \varrho $, and 
we define $  W^{\vr, \infty}  (\T^d ) $ as the space of  functions 
$ u $ in $  C^{\floor{\vr}}(\T^d,\C)  $
whose derivatives of order  $ \floor{\vr} $ 
are $ (\varrho - \floor{\varrho}) $-H\"older-continuous, that is  
$$
[\pa_x^\alpha u]_{\vr} := \sup_{x\neq y}\frac{|\pa_x^\alpha u(x)- \pa_x^\alpha u(y)|}{|x-y|^{\varrho - \floor{\vr}}} < + \infty \, , \ \ \  \forall |\alpha| = \floor{\vr} \, , 
$$
equipped with the norm 
$$
\| u \|_{W^{\varrho,\infty}}  := 
{\sum_{|\alpha| \leq \floor{\vr}} \| \pa_x^\alpha u \|_{L^\infty} + \sum_{|\alpha| = \floor{\vr}} [\pa_x^\alpha u]_{\vr}}  \, . 
$$
For $ \varrho = 0 $ the norm $ \| \ \|_{W^{\varrho,\infty}} = \| \ \|_{L^{\infty}}  $.
\end{definition}

The H\"older spaces $ W^{\vr, \infty} (\T^d) $ can be described by 
the Paley-Littlewood decomposition of a function.
Consider the locally finite partition on unity 
\begin{equation}\label{PUN}
1 =  \chi (\xi ) + \sum_{k \geq 1} \varphi(2^{-k}\xi) \, , \quad 
\varphi (z) := \chi (z) -  \chi (2z) \, , 
\end{equation}
where $ \chi : \R^d \to \R  $ is the
cut-off function defined in \eqref{def-chi}. It induces the decomposition
of a distribution $ u \in {\cal S}' (\T^d) $ as
\begin{equation}
\label{LP}
u = \sum_{k \geq 0} \Delta_k u \quad {\rm where}
\quad
\Delta_0:= \chi(D) \, , 
\  \Delta_k := \varphi (2^{-k} D) = \chi_{2^k}(D)- \chi_{2^{k-1}}(D) \, , \ k \geq 1 \, . 
\end{equation}
We also set
\begin{equation}\label{S_k}
S_k  := \sum_{0 \leq j \leq k} \Delta_j   = \chi_{2^k}(D) \, . 
\end{equation}
The Paley-Littlewood theory of the  H\"older spaces $ W^{\varrho,\infty} (\T^d)$
follows as in $ \R^d $, see e.g. \cite{Met}, 
once we represent the Fourier multipliers $ \Delta_k  $ 
as  integral convolution operators on $ \R^d $,  by Lemma \ref{lem:FC}. 
In particular the following 
smoothing estimates hold: 
for any $ \alpha \in \N^d_0 $, $ \varrho \geq 0 $, 
\begin{equation}\label{sm3}
\norm{\pa_x^\alpha S_k u}_{L^\infty} \lesssim 2^{k(| \alpha| -\varrho)} \norm{ u }_{W^{\varrho,\infty}} \, , 
\end{equation}
and, for any $ \varrho  >0  $, 
\begin{equation}\label{smo1}
\| u - \chi_\theta (D) u \|_{L^\infty} \lesssim \theta^{-\varrho} \| u \|_{W^{\varrho,\infty}}  \, . 
\end{equation}
In this way it results as in $ \R^d $ 
that the H\"older norms $ \norm{ \ }_{W^{\vr, \infty}} $ satisfy   interpolation estimates.
In particular we shall use that, 
given $\vr, \vr_1, \vr_2 \geq 0 $,  
\begin{equation}\label{inter}
\begin{aligned}
&\| u v \|_{W^{\vr,\infty}} \lesssim
\| u \|_{W^{\vr,\infty}} \| v \|_{L^{\infty}}
 + \| u \|_{L^{\infty}} \| v \|_{W^{\vr, \infty}} \\
& \norm{u }_{W^{\vr, \infty}} \lesssim \norm{u}_{W^{\vr_1, \infty}}^\theta \, \norm{u}_{W^{\vr_2,\infty}}^{1-\theta} \, , \ \ \ \  \vr = \theta \vr_1 + (1-\theta) \vr_2  \, , \ 
\theta \in (0,1) \, . 
\end{aligned}
\end{equation}

\paragraph{H\"older estimates of regularized symbols.}

In order to prove 
estimates of the regularized symbol $ a_\chi $ defined in 
\eqref{achi} in H\"older spaces (Lemma \ref{lem:achi})
we represent it  
as a convolution integral on $ \R^d $, by Lemma \ref{lem:FC}, 
\begin{equation}
\label{achi.int}
a_\chi(x, \xi) = 
\int_{\R^d} a(x-y, \xi ) \, \psi_{\epsilon \la \xi \ra}(y) \, \di y 
\end{equation}
where $\psi_\theta(z) = \theta^d \psi ( \theta z ) $ 
and $ \psi $ is the anti-Fourier transform of $ \chi $.

In the proof of Lemma \ref{lem:achi} we shall use the following estimate.

\begin{lemma} 
For any $ \beta \in \N_0^d $, $ u \in L^\infty (\T^d) $, we have 
\begin{equation}\label{bound-pro}
\| \partial_\xi^\beta \chi_{\epsilon \langle \xi \rangle} (D) u \|_{L^\infty} 
\lesssim  \la \xi \ra^{-|\beta|} \|u \|_{L^\infty} \, . 
\end{equation}
\end{lemma}

\begin{proof}
By \eqref{achi.int} we have, for all $ \beta \in \N^d_0 $, 
\begin{equation}\label{der-pezzi}
\pa_\xi^\beta \chi_{\epsilon \la \xi \ra}(D) u   = 
\int_{\R^d} u(x-y) \, \pa_\xi^\beta \psi_{\epsilon \la \xi \ra}(y) \, \di y \, . 
\end{equation}
By  the definition $\psi_{\epsilon \la \xi \ra}(y) = (\epsilon \la \xi \ra)^d \, \psi(\epsilon \la \xi \ra y)$ and Fa\`a di Bruno formula, we have that
\begin{equation}
\label{psi.der}
\int_{\R^d} \big| \pa_\xi^\beta \psi_{\epsilon \la \xi \ra}(y) \big| \, \di y  \lesssim  \la \xi \ra^{-|\beta|} , \quad \forall \xi \in \R^d \, . 
\end{equation}
Then \eqref{bound-pro} follows by \eqref{der-pezzi} and \eqref{psi.der}.
\end{proof}

The next  lemma provides estimates of the regularized 
symbol $ a_\chi $ in terms of the  symbol $ a $.

\begin{lemma}
\label{lem:achi}
{\bf (Estimates on regularized symbols)}
Let $ m \in \R $, $ N \in \N_0 $. 
\begin{enumerate}
\item
If $a \in \Gamma^m_{L^\infty}$, $ m \in \R $, then $ a_\chi $ 
defined in \eqref{achi}
 belongs to $\Sigma^m_{L^\infty}$ and
\begin{equation}
\label{achi.est}
\abs{a_\chi}_{m, L^{\infty}, N} \lesssim
\abs{a}_{m, L^{\infty}, N}  \, . 
\end{equation}
\item
 If   $a \in \Gamma_{H^{s_0-\vr}}^m$, $\varrho\geq 0$, $s_0 > \frac{d}{2}$,  then
 $a_\chi$ belongs to $\Gamma^{m+\vr}_{L^\infty}$ and 
\begin{equation}\label{simbolocutrec}
| a_\chi|_{m+\varrho, L^\infty, N} \lesssim | a|_{m, {H^{s_0-\varrho}},N} \, .
\end{equation}
\item If $a \in \Gamma^m_\hol$, $\vr > 0 $, then,
for any $\beta \in \N_0^d $, $\pa_\xi^\beta a_\chi - (\pa_{\xi}^\beta a)_\chi  \in \Sigma^{m-|\beta| -\vr}_{L^\infty}$ and 
\begin{equation}
\label{achi.est1}
\big| \pa_\xi^\beta a_\chi - (\pa_{\xi}^\beta a)_\chi \big|_{m-|\beta|-\vr, L^{\infty}, N} \lesssim 
\abs{a}_{m, \hol, N+ |\beta|}  \, . 
\end{equation}
\item   If $a \in \Gamma^m_\hol$, $\vr \geq 0 $, then, for any  $\alpha \in \N_0^d$ with $|\alpha| \geq \vr $,
$ \pa_x^\alpha a_\chi  = (\pa_x^\alpha a)_\chi \in \Sigma^{m+|\alpha| -\vr}_{L^\infty}$ and
\begin{equation}
\label{achi.est3}
\abs{\pa_x^\alpha a_\chi}_{m+|\alpha| -\vr, L^{\infty}, N} \lesssim 
\abs{a}_{m, \hol, N}  \, . 
\end{equation}
\item  If $a \in \Gamma^m_\hol$, $\vr > 0 $, then, $a- a_\chi \in \Gamma^{m-\vr}_{L^\infty}$ and 
\begin{equation}
\label{achi.est2}
\abs{a- a_\chi}_{m-\vr, L^{\infty}, N} \lesssim 
\abs{a}_{m, \hol, N}  \, . 
\end{equation}
\end{enumerate}
\end{lemma}
\begin{proof}

{\sc Proof of \eqref{achi.est}.} Differentiating  
\eqref{achi} for any $ \beta \in \N_0^d $, we have
$$
\pa_\xi^\beta a_\chi (x,\xi) = \sum_{\beta_1+\beta_2 = \beta}
C_{\beta_1, \beta_2}  
\pa_{\xi}^{\beta_1} \chi_{\epsilon \langle \xi \rangle}(D) \, 
\pa_{\xi}^{\beta_2} a(\cdot, \xi )  \, .
$$
Then
\eqref{seminorm} and  \eqref{bound-pro}  directly 
imply \eqref{achi.est}.
\\[1mm]
{\sc Proof of \eqref{simbolocutrec}}
By the Cauchy-Schwartz inequality 
\begin{align*}
 \abs{a_\chi(x,\xi)} & = \Big| \sum_{n\in \Z^d} \chi_\epsilon \left(\frac{n}{\langle \xi\rangle}\right) \ \wh a(n,\xi) \, e^{\im n\cdot x} \Big|
  \leq 
  \sum_{n\in \Z^d} 
  \chi_\epsilon \Big(\frac{n}{\langle \xi\rangle}\Big) \, \frac{\la n \ra^{\vr}}{\la n \ra^{s_0 }} \, 
  \la n\ra^{s_0-\vr} \, \abs{\wh a(n,\xi)}  \\
  & \lesssim 
\Big(   \sum_{n \in \Z^d}  \chi_\epsilon^2 \Big(\frac{n}{\langle \xi\rangle} \Big)\frac{\la n\ra^{2\vr}}{\la n \ra^{2s_0}} \Big)^{1/2} \norm{a(\cdot, \xi)}_{H^{s_0-\vr}} 
  \lesssim \la \xi \ra^{m+\vr}  \abs{a}_{m, H^{s_0-\vr}, 0} \, .
 \end{align*}
  The case $N\geq 1$ follows in the same way.
\\[1mm]
{\sc Proof of \eqref{achi.est1}.} 
First, for any $ \xi \in \R^d $,  we define $ k \in \N $ such that 
$ 2^{k-1} \leq 2 \epsilon \langle \xi \rangle \leq 2^k $. Then, by the properties of 
the cut-off function $ \chi $ in \eqref{def-chi} and the projector $ S_k $ in \eqref{S_k} 
we have 
\begin{equation}
\label{chiS_k}
\pa_\xi^\beta \chi_\epsilon\left( \frac{\eta}{\langle \xi \rangle} \right)  =
\Big(\pa_\xi^\beta \chi_\epsilon \Big( \frac{\eta}{\langle \xi \rangle} \Big)\Big) \, S_k  \, , \quad  \  \forall \eta \in \R^d \ , 
\quad \forall \beta \in \N^d_0 \, . 
\end{equation}
Differentiating \eqref{achi} and using \eqref{chiS_k} we get
$$
\pa_\xi^\beta 
a_\chi - (\pa_\xi^\beta a)_\chi 
 = \sum_{\beta_1 + \beta_2 = \beta, \beta_1 \neq 0}  
C_{\beta_1 \beta_2} \,  \pa_\xi^{\beta_1} \chi_{\epsilon\la \xi \ra}(D) \, S_k \, \pa_\xi^{\beta_2} a(\cdot, \xi) \, , 
$$
and, using \eqref{bound-pro} and \eqref{sm3}
\begin{align*}
\big\| \big(\pa_\xi^\beta a_\chi - (\pa_\xi^\beta a)_\chi \big)(\cdot, \xi) \big\|_{L^\infty}  
&
\stackrel{} \lesssim
\sum_{\beta_1 + \beta_2 = \beta, \beta_1 \neq 0} 
\la \xi \ra^{-|\beta_1|}\, 2^{-k \vr}
 \big\| \pa_\xi^{\beta_2} a(\cdot, \xi) \big\|_{W^{\vr, \infty}} \\
& \lesssim \la \xi \ra^{m-|\beta|}\, 2^{-k \vr} \abs{a}_{m, W^{\vr, \infty}, |\beta|}
\lesssim \la \xi \ra^{m-|\beta|-\vr}\,  \abs{a}_{m, W^{\vr, \infty}, |\beta|}
\end{align*}
because $\la \xi \ra \lesssim 2^k$. This proves \eqref{achi.est1} for $N = 0$. For $N \geq 1$ the estimate is similar.
\\[1mm]
{\sc Proof of \eqref{achi.est3}}.
For any $ \xi \in \R^d $,  we define $ k \in \N $ such that 
$ 2^{k-1} \leq 2 \epsilon \langle \xi \rangle \leq 2^k $.
By \eqref{achi} and \eqref{chiS_k} with $\beta = 0$, we write
$ a_\chi (\cdot, \xi) = \chi_{\epsilon \langle \xi \rangle} (D) a(\cdot, \xi) =
\chi_{\epsilon \langle \xi \rangle} (D)  S_k  a(\cdot, \xi) $, 
and then 
$$
\begin{aligned}
\| \pa_x^\alpha a_\chi (\cdot, \xi) \|_{L^\infty} 
& =  
\|  \chi_{\epsilon \langle \xi \rangle} (D) \pa_x^\alpha S_k  a(\cdot, \xi)  \|_{L^\infty}  
\stackrel{\eqref{bound-pro}} \lesssim
\| \pa_x^\alpha S_k a(\cdot, \xi)  \|_{L^\infty} \\ 
& \stackrel{\eqref{sm3}} \lesssim
2^{k (|\alpha|-\varrho)} \|  a(\cdot, \xi)  \|_{W^{\varrho,\infty}} 
\stackrel{ \langle \xi \rangle \sim  2^k }\lesssim
\langle \xi \rangle^{|\alpha|-\varrho} \|  a(\cdot, \xi)  \|_{W^{\varrho,\infty}} \lesssim \langle \xi \rangle^{m+ |\alpha|-\varrho} |a|_{m,W^{\varrho,\infty},0}  
\end{aligned}
$$
by \eqref{seminorm}. 
This proves \eqref{achi.est3} with $N = 0$.  For $N \geq 1$ the estimate is similar.
\\[1mm]
{\sc Proof of \eqref{achi.est2}}. 
For any $ \beta \in \N_0^d $ we write
$ \pa_\xi^\beta (a - a_\chi) = 
\big[ \pa_\xi^\beta a - (\pa_\xi^\beta a)_\chi \big]  + \big[ (\pa_\xi^\beta a)_\chi - 
\pa_\xi^\beta a_\chi \big] $.
The first term is bounded, using \eqref{smo1} with 
$ \theta = \epsilon \langle \xi \rangle $, as 
$$
\big\| \big( \pa_\xi^\beta a - (\pa_\xi^\beta a)_\chi \big)(\cdot , \xi) \big\|_{L^\infty} 
\lesssim \langle \xi \rangle^{-\varrho}
\| \pa_\xi^\beta a (\cdot, \xi) \|_{W^{\varrho,\infty}} \lesssim
|a|_{m,W^{\varrho,\infty},|\beta|} \langle \xi \rangle^{m-\varrho-|\beta|}
$$
The second term satisfies the same bound by \eqref{achi.est1}.
This proves  \eqref{achi.est2}. 
\end{proof}

\paragraph{Change of quantization.}
In order to prove the boundedness Theorem \ref{thm:cont2}
and the composition Theorem \ref{thm:comp2},  it is convenient 
to pass from the Weyl quantization 
of a symbol $ a(x, \xi)$, defined in \eqref{opW}, 
to the standard quantization which is defined, given a symbol $ b(x, \xi) $, as 
\begin{align} 
\label{B} {\rm Op} {(b)}[u] & := \sum_{j \in \Z^d} \Big( \sum_{k \in \Z^d}
\widehat b \left(j-k, k\right)  \, u_k \Big) e^{\im j \cdot x } =
\sum_{k \in \Z^d} b (x, k)  \, u_k e^{\im k \cdot x } 
\, . 
\end{align}
We have the  change of quantization formula
\begin{equation}
\label{cambio}
\Opw{a} = \Op{b} \qquad \Leftrightarrow \qquad \wh  b(n, \xi ) := \wh a 
\big( n, \xi + \frac{n}{2} \big)  \, . 
\end{equation}
In the next lemma we estimate the norms of $ b $ in terms of those of $ a $. 
We remind that $ \Sigma^m_\mathscr{W}$ denotes the set of
{\em spectrally localized} symbols, i.e. satisfying \eqref{spectr-loca}. 

\begin{lemma}{\bf (Change of quantization)}
\label{thm:change}
Let $a \in \Sigma^m_{L^\infty}$,  $m \in \R $. 
If $ \delta > 0 $ in  \eqref{spectr-loca} is small enough, then 
 (cfr. \eqref{cambio})
\begin{equation}\label{defb}
 b(x,\xi) := \sum_{n\in\Z^d} \wh a \big( n,\xi + \frac{n}{2} \big)  \, e^{\im n\cdot x} 
\end{equation}
is a symbol in $  \Sigma^m_{L^\infty}$ satisfying 
\begin{equation}
\label{cambio2}
| b|_{m,L^\infty, N}\lesssim | a |_{m,L^\infty, N+d+1} \, , \quad \forall N\in \N_0 \, .
\end{equation}
\end{lemma}

\begin{proof}
Since $ a $  satisfies \eqref{spectr-loca} with $\delta  $ small enough, it follows that $b$ satisfies \eqref{spectr-loca}.
In order to prove \eqref{cambio2} we differentiate \eqref{defb} obtaining that, 
for any $ \beta \in \N_0^d $, 
$$
\partial_\xi^\beta  b(x,\xi)
= \sum_{n\in\Z^d} \wh{ \partial_\xi^\beta a} (n,\xi + \frac{n}{2}) \, e^{\im n\cdot x} 
 = 
 \sum_{n\in\Z^d} \wh{ \partial_\xi^\beta a} (n,\xi + \frac{n}{2})  
 \ \chi_\epsilon \left( \frac{n}{\la \xi\ra}\right)
\, e^{\im n\cdot x}  
$$
for some $ \epsilon = \epsilon(\delta') > 0 $, 
where in the last equality we used that the sum  is actually restricted over the indexes for which   $|n|\leq \delta' \langle \xi \rangle$, $\delta' \in (0,1)$.
Then we represent $\pa_\xi^\beta b$ as 
the integral 
\begin{equation} \label{defK1}
\pa_\xi^\beta b(x,\xi)=\int_{\T^d}  K(x, y)\, \di y \, , \quad
K(x,y) :=  \frac{1}{(2\pi)^d} \sum_{n \in \Z^d}  (\partial_\xi^\beta a) \big( x-y,\xi + \frac{n}{2} \big) 
\,
\chi_\epsilon \left(\frac{n}{\la \xi \ra }\right)
\, e^{\im n\cdot y}  \, . 
\end{equation}
We are going to estimate the $ L^1 $-norm of $ K(x, \cdot ) $ using 
Lemma \ref{lem:K}. First note that,  
since  $ a \in \Sigma^m_{L^\infty}$, 
we have 
$\la \xi + \frac{n}{2}\ra \sim \la \xi \ra$
on the support of  $ \wh {\partial_\xi^\beta a}  (n,\xi + (n/2)) $,   and then we bound \eqref{defK1} as
\begin{equation}
\label{cambio3}
|K(x,y)| \lesssim
 \sum_{|n| \leq \delta' \la \xi \ra  } 
 | a|_{m, L^\infty, |\beta|} \la \xi \ra^{m-|\beta| } \lesssim 
 | a |_{m, L^\infty, |\beta| }\,   \langle \xi\rangle^{d + m - |\beta|} \  ,
\end{equation}
 uniformly in $ x $. 
Moreover, using Abel resummation formula \eqref{Abeld} and the Leibniz rule \eqref{finitediff}
for finite differences, we get, for any $ h =  1, \ldots, d $, 
$$
K(x,y) = \frac{1}{\left( e^{\im y_h}-1 \right)^{d+1}}
\sum_{ k_1 +k_2   = d+1} C_{k_1, k_2} \sum_{|n| \leq \delta' \la \xi \ra } 
\pa_h^{k_1}  (\partial_\xi^\beta a) \big( x-y,\xi + \frac{n}{2} \big) 
\, \pa_h^{k_2} \chi_\epsilon \left(\frac{n}{\la \xi  \ra }\right)
\, e^{\im n\cdot y} \, .
$$
Then, using \eqref{seminorm} and that 
$
\displaystyle{ \big| \pa_h^k  \chi_\epsilon \big( \frac{n}{\la \xi\ra} \big) \big| \lesssim  
\la \xi \ra^{-k}} $, $\forall h = 1, \ldots, d$,
 we estimate
\begin{align}\label{cambio4}
| K(x,y)| \lesssim \frac{\la \xi \ra^{m - (d+1) - |\beta|}}{|2\sin(y_h/2)|^{d+1}} \  |a|_{m, L^\infty, |\beta| + d+1} \sum_{|n| \leq \delta' \la \xi \ra } 1 
\lesssim \, 
\frac{\la \xi \ra^{d+m  - |\beta|} \,  |a|_{m, L^\infty, |\beta| + d+1}}{|\la \xi \ra \, 2\sin(y_h/2)|^{d+1}} \ 
\end{align}
uniformly in $x$. 
In view of \eqref{cambio3}-\eqref{cambio4}
we apply Lemma \ref{lem:K} with 
$A = \langle \xi\rangle$ and 
$B= \la \xi \ra^{m  - |\beta|} \,  |a|_{m, L^\infty, |\beta| + d+1}$ obtaining 
$$
\big| \pa_\xi^\beta b(x,\xi) \big| \leq \int_{\T^d} |K(x,y)| \di y 
\lesssim\la \xi \ra^{m  - |\beta|} \,  |a|_{m, L^\infty, |\beta| + d+1} \, , \quad \forall (x,\xi)\in \T^d \times \R^d,$$
that proves \eqref{cambio2}.
\end{proof}

\paragraph{Continuity.}

We now prove boundedness estimates  in Sobolev spaces
of  operators with spectrally localized symbols, 
requiring derivatives in $\xi$ of the symbol and 
no derivatives in $x$. 

\begin{theorem}{\bf (Continuity)}
\label{thm:cont2}
Let $a \in \Sigma^m_{L^\infty}$ with $m \in \R$. Then 
$\Op{a}$ defined in \eqref{B} extends to a bounded operator from 
$  H^s \to  H^{s-m}$, for any $ s \in \R $,  satisfying
\begin{equation}
\label{cont}
\norm{\Op{a}u}_{{s-m}} \lesssim  \, \abs{a}_{m, L^\infty, d+1} \, \norm{u}_{{s}} .
\end{equation}
Moreover, if $a$  fulfills \eqref{spectr-loca} 
with $ \delta > 0 $ small enough, then the operator  $\Opw{a} $ defined in \eqref{opW} satisfies 
\begin{equation}
\label{cont0}
\norm{\Opw{a}u}_{s-m} \lesssim  \, \abs{a}_{m, L^{\infty}, 2(d+1)} \, 
\norm{u}_{s} \, . 
\end{equation}

\end{theorem}
\begin{proof}
We first recall the Littlewood-Paley characterization 
of  the Sobolev norm 
\begin{equation}\label{Soboeq}
\| u \|_s^2 \sim \sum_{k \geq 0} 2^{2ks} \| \Delta_k u \|_0^2 
\end{equation}
where $ \Delta_k $ are defined in \eqref{LP}. The norm $ \| \ \|_0 = \| \ \|_{L^2} $. 
We first prove \eqref{cont}.  
\\[1mm]
\underline{Step 1:} 
according to \eqref{PUN}, we perform the Littlewood-Paley decomposition of 
$\Op{a} $, 
\begin{equation}\label{PLBW}
\Op{a}v = \sum_{k \geq 0} \Op{a_k}v \, , 
\end{equation}
where
\begin{equation}
\label{lock}
a_0 (x,\xi) := a(x,\xi) \chi(\xi) \, , 
\quad  a_k(x,\xi) := a(x,\xi) \varphi(2^{-k}\xi) \, , \quad k \geq 1 \, . 
\end{equation}
In order to  prove \eqref{cont}, it is sufficient to prove that 
\begin{equation}
\label{eq:L2L2}
\norm{\Op{a_k}v}_{0} \lesssim  |a|_{m, L^\infty, d+1} \, 2^{km} \, \norm{v}_{0} , \quad \forall k \in \N_0 \, , \quad \forall v \in L^2 \, . 
\end{equation}
Indeed, decomposing $ v$ in Paley-Littlewood  packets as in \eqref{LP}, 
\begin{equation}\label{PLv}
v = \sum_{j \geq 0} \Delta_j v \, , \quad  
\Delta_0 = \chi (D) \, , \ \Delta_j = \varphi(2^{-j} D ) \, , 
\end{equation}
which are {\it almost orthogonal} in $L^2$ (namely 
$ \Delta_k \Delta_j = 0 $  for any $ |j-k| \geq 3 $),   
using the fact that $ \Op{a_k}v = \Op{a}  \Delta_k v $,  
and 
{since the action of $\Op{a_k}$ does not spread much the Fourier support of functions being $a$ spectrally localized,} according to  \eqref{rela:para},  
we have 
\begin{align*}
\| \Op{a}v \|_{s-m}^2 & 
\stackrel{\eqref{PLBW}} = \Big\| \sum_{k \geq 0} \Op{a_k}v \Big\|_{s-m}^2 
\stackrel{\eqref{PLv},\eqref{lock}} = 
\Big\| \sum_{|j-k| < 3}  \Op{a_k} \Delta_j v \Big\|_{s-m}^2 \\
& \sim \sum_{|j-k| < 3} 2^{2k(s-m)} \,  \| \Op{a_k} \Delta_j v \|_{0}^2 
 \stackrel{\eqref{eq:L2L2}}  \lesssim  |a|_{m, L^\infty, d+1}^2 \, \sum_{|j-k| < 3}  \, 2^{2ks} \,  \| \Delta_j v \|_{0}^2 \\ 
& \lesssim   |a|_{m, L^\infty, d+1}^2 \, \sum_{k \geq 0} 2^{2ks}  \norm{\Delta_k v}_{0}^2  
\stackrel{\eqref{Soboeq}} \sim 
|a|_{m, L^\infty, d+1}^2 \| v \|_{s}^2 \, . 
\end{align*}

\noindent{\underline{Step 2:}} By 
 \eqref{lock} and \eqref{B} we  write $\Op{a_k}$ as the integral operator 
\begin{equation}
\label{opak}
(\Op{a_k} v)(x) = \int_{\T^d} K_k(x, x-y) \, v(y) \, \di y
\end{equation}
with kernel 
\begin{equation}
 \label{defK1pd} 
K_k (x,z) := \frac{1}{(2 \pi)^d} \sum_{\ell  \in \Z^d} e^{\ii \ell  \cdot z} \,  a(x,\ell ) \, 
\varphi(2^{-k}\ell ) \ .  
\end{equation}
We shall deduce \eqref{eq:L2L2} by applying the Schur lemma: if 
\begin{equation}\label{c1c2}
\sup_{x \in \T^d} \int_{\T^d} |K(x,x-y)| \di y =: C_1 < + \infty \, , \quad 
\sup_{y \in \T^d} \int_{\T^d} |K(x,x-y)| \di x =: C_2 < + \infty
\end{equation}
then Schur lemma guarantees that the   integral operator  \eqref{opak} 
 is bounded on  $ L^2 (\T^d) $ and
\begin{equation}
\label{estSchur}
\| \Op{a_k} v \|_0 \leq (C_1 C_2)^{1/2} \| v \|_0 \, . 
\end{equation}
Let us prove \eqref{c1c2} and estimate the constants $C_1, C_2$. 
By \eqref{defK1pd}  we have that
\begin{equation}
\abs{K_k(x,z)} \lesssim   \sum_{\ell \in \Z^d}  |a(x,\ell)| \varphi(2^{-k}\ell)  
\stackrel{\eqref{seminorm}} \lesssim  | a |_{m, L^\infty, 0} \sum_{\ell \in \Z^d} \la \ell \ra^m  \varphi(2^{-k}\ell ) \lesssim 
  2^{k(d+m)}    | a |_{m, L^\infty, 0} \, . \label{eq:smallerpd}
\end{equation}
Then,   applying $ (d+1)$-times  Abel resummation formula \eqref{Abeld} 
to \eqref{defK1pd}, we obtain, for any $ h = 1, \ldots, d $, 
$$
K_k (x,z) = 
\frac{1}{(2 \pi)^d} \frac{1}{(e^{\ii z_h}-1)^{d+1}} 
\sum_{\ell \in \Z^d} e^{\ii \ell \cdot z} \ \pa_h^{d+1} ( a (x,\ell)\varphi(2^{-k}\ell) )
$$
and we deduce, using \eqref{seminorm}, \eqref{finitediff}, 
$
\abs{K_k(x,z)}\lesssim 
\abs{2 \sin(z_h/2)}^{-d-1}  |a|_{m, L^\infty,d+1} 2^{k(m-1)} $
for any $ h = 1, \ldots, d $, 
thus 
\begin{equation}\label{eq:greaterpd}
\abs{K_k(x,z)}\lesssim    2^{k(d+m)}  \, |a|_{m, L^\infty,d+1} \, 
\ \min_{h =1, \ldots, d} \frac{1 }{( 2^{k} 2 \,  | \sin(z_h/2)|)^{d+1}  } \, .
\end{equation}
By  \eqref{eq:smallerpd}, \eqref{eq:greaterpd}
we apply Lemma \ref{lem:K} with 
$ A = 2^{k}$ and $B= 2^{km} |a|_{m, L^\infty, d+1} $, deducing that 
\begin{align}\label{sch1}
\int_{\T^d} |K_k\bigl(x,x-y\bigr)| \di y  
&= 
\int_{\T^d} |K_k\bigl(x,z\bigr)| \di z 
   \lesssim   2^{km} \, |a|_{m, L^\infty,d+1} 
\end{align}
uniformly for $ x \in \T^d $. Similarly
\begin{align}\label{Kypd}
\int_{\T^d} | K_k\bigl(x,x-y\bigr)| \di x 
& \lesssim   2^{km} \, |a|_{m, L^\infty,d+1} 
\end{align}
uniformly for $ y \in \T^d $. 
Finally \eqref{sch1}, \eqref{Kypd},  
\eqref{estSchur} prove \eqref{eq:L2L2} completing the proof of \eqref{cont}.
\\[1mm]
{\sc Proof of \eqref{cont0}.}
By Lemma \ref{thm:change} we have $\Opw{a} = \Op{b}$ for a 
 spectrally localized symbol $b \in \Sigma^m_{L^\infty}$ which 
 fulfills estimate \eqref{cambio2}. Then \eqref{cont0} follows by \eqref{cont}.
\end{proof}

\paragraph{Composition of paradifferential operators.}
We finally prove a composition result for paradifferential operators.  
The difference with respect to Theorem 6.1.1 and 6.1.4 in 
\cite{Met} is to have periodic symbols and
the use of the Weyl quantization.  

We shall use that, in view of  the interpolation inequality \eqref{inter},  if 
$ a \in \Gamma^m_{W^{\varrho,\infty}} $ and 
$ b \in \Gamma^{m'}_{W^{\varrho,\infty}}  $ then 
$ ab \in \Gamma^{m+m'}_{W^{\varrho,\infty}} $ and, for any 
$ N \in \N_0  $,  any $0 \leq \vr_1 \leq \alpha \leq \beta \leq\vr_2 $ such that  $\vr_1 + \vr_2 = \alpha + \beta$
\begin{equation}
\label{alg.symb}
\begin{aligned}
&\abs{ab}_{m+m', {W^{\varrho,\infty}}, N} \lesssim \abs{a}_{m, {W^{\varrho,\infty}}, N} \, \abs{b}_{m', L^\infty, N} + 
\abs{a}_{m, L^\infty, N} \, \abs{b}_{m',{W^{\varrho,\infty}}, N}  \,  ,\\
&
| a|_{m, W^{\alpha, \infty}, N} \, 
|b|_{m', W^{\beta, \infty},N} \lesssim
|a|_{m, W^{\vr_1,\infty}, N} \, |b|_{m', W^{\vr_2, \infty},N} +
|a|_{m, W^{\vr_2,\infty}, N} \, |b|_{m', W^{\vr_1, \infty},N} \, . 
\end{aligned}
\end{equation}

\begin{theorem} {\bf (Composition)}
\label{thm:comp2}
Let $a \in \Gamma^m_{W^{\vr, \infty}}$, $b \in \Gamma^{m'}_{W^{\vr, \infty}}$ with $m, m' \in \R$ and $\vr  \in (0,2]$.  Then  
\begin{align}
\label{comp01A}
\Opbw{a}\Opbw{b} 
& =  \Opbw{a\#_\vr b} + R^{-\vr}(a,b)
\end{align}
where the linear operator $R^{-\vr}(a,b)\colon {\dot H}^s \to {\dot H}^{s-(m+m')+\vr}$, 
$\forall s \in \R$,  satisfies
\begin{align}
\label{comp020}
\norm{R^{-\vr}(a,b)u}_{{s -(m+m') +\vr}} \lesssim  \left(\abs{a}_{m, W^{\vr, \infty}, N} \, \abs{b}_{m', L^\infty, N} + \abs{a}_{m, L^\infty, N} \, \abs{b}_{m', W^{\vr, \infty}, N}  \right) \norm{u}_{s}
\end{align}
with $ N \geq 3d+4$.
\end{theorem}

\begin{proof}
We give the proof in the case $ \vr \in (1, 2] $. 
We first compute  $\Opbw{a}\Opbw{b}$. Recalling the definition \eqref{BW}
we obtain 
\begin{align*}
\Opbw{a}\Opbw{b}u 
& =  
\Opw{a_\chi} \Opw{b_\chi} 
=  \sum_{ j, k , \ell } \wh a_\chi\Big(j-k, \frac{j+k}{2} \Big) \, 
\wh b_\chi \Big( k-\ell, \frac{k + \ell}{2} \Big)  \,  
u_\ell  \, 
e^{\im j \cdot x} \, . 
\end{align*}
We now perform a 
Taylor expansion of $ \wh a_\chi\big(j-k, \frac{j+k}{2} \big)  $ 
in the second variable, 
around the point $\frac{j+ \ell}{2}$. Writing $j+ k = j+ \ell + (k - \ell)$, we obtain  
\begin{align*}
\wh a_\chi\Big(j-k, \frac{j+k}{2} \Big)
& = \wh a_\chi\Big(j-k, \frac{j+\ell}{2} \Big) +  
\Big( \frac{k-\ell}{2}\Big) \cdot \pa_\xi \wh a_\chi\Big(j-k, \frac{j+\ell}{2} \Big)
\\
& \quad + \sum_{\alpha \in \N_0^d, |\alpha| = 2} \Big( \frac{k-\ell}{2}\Big)^{\alpha}  \int_0^1 (1-t) \,  \pa_\xi^\alpha \wh a_\chi\Big(j-k, \frac{j+\ell + t(k- \ell)}{2} \Big) \di t \, . 
\end{align*}
We expand analogously 
$ \wh b_\chi\big(k-\ell, \frac{k+ \ell}{2} \big)  $ around the point $\frac{j+\ell}{2}$.
Writing $ k + \ell = j+ \ell - (j-k)$, we obtain   
\begin{align*}
\wh b_\chi\Big(k-\ell, \frac{k+\ell}{2} \Big)
& =  \wh b_\chi\Big(k-\ell, \frac{j+\ell}{2} \Big) -
 \Big( \frac{j-k}{2}\Big)\cdot \pa_\xi \wh b_\chi\Big(k-\ell, \frac{j+\ell}{2} \Big)\\
&\quad 
+ \sum_{\beta \in \N_0^d,  |\beta| = 2}   \Big( \frac{k-j}{2}\Big)^{\beta} \int_0^1(1-t)   \, \pa_\xi^\beta \wh b_\chi\Big(k-\ell, \frac{j+\ell + t(k- j)}{2} \Big) \di t \, . 
\end{align*}
Moreover, recalling \eqref{regolarized-simbol} and \eqref{opW}, we write
$\Opbw{a\#_\vr b  } u 
= \Opw{(ab + \frac{1}{2\im}\{a, b\})_\chi} u $ and, by the previous expansions,  
\begin{align*}
\Big(\Opbw{a}\Opbw{b} - {\rm Op}^{BW} \Big( ab + \frac{1}{2\im}\{a,b\} \Big) \Big) u 
& = \sum_{i=1}^4 R_i(a,b) u  
\end{align*}
where
\begin{align}
\label{R1}
& R_1(a,b)u := {\rm Op^W} \Big( {a_\chi b_\chi - (ab)_\chi + \frac{1}{2\im}
\big( \{a_\chi, b_\chi\} - (\{a,b\})_\chi \big)} \Big) u \\
\label{R2}
& R_2(a,b)u:= \sum_{j,k,\ell}
\wh b_\chi\Big(k-\ell, \frac{k+\ell}{2} \Big) \!\! \sum_{|\alpha| = 2}
 \Big( \frac{k-\ell}{2}\Big)^{\alpha}  \!\!
 \int_0^1 \!\!(1-t) \,  \pa_\xi^\alpha \wh a_\chi\Big(j-k, \frac{j+\ell + t(k- \ell)}{2} \Big) \di t 
 u_\ell  \, e^{\im j \cdot x}
 \\
\label{R3}
& R_3(a,b) u := \sum_{j,k,\ell} 
\!\! - \Big( \frac{k-\ell}{2}\Big)\!  \cdot \pa_\xi \wh a_\chi\Big(j-k, \frac{j+\ell}{2} \Big)   \Big( \frac{j-k}{2}\Big)\! \cdot \!\! \int_0^1 \!\! \pa_\xi\wh b_\chi\Big(k-\ell, \frac{j+\ell + t(k- j)}{2} \Big) \di t  
 u_\ell  \, e^{\im j \cdot x} \\
\label{R4}
& R_4(a,b) u := \sum_{j,k,\ell} \wh a_\chi\Big(j-k, \frac{j+\ell}{2} \Big)  
\!\! \sum_{|\beta| = 2}  
 \Big( \frac{k-j}{2}\Big)^{\beta} \!\!
 \int_0^1 \!\! (1-t)   \, \pa_\xi^\beta \wh b_\chi\Big(k-\ell, \frac{j+\ell + t(k- j)}{2} \Big) \di t  
  u_\ell  \, e^{\im j \cdot x}. 
\end{align}
We show now that the operators $R_i(a,b)$, $ i = 1, \ldots, 4$ fulfill estimate \eqref{comp020}. \\

\noindent{\underline{ Estimate of $R_1(a,b)$.}} 
By exchanging the role of $a$ and $b$ it is enough  to prove that the symbols 
$ \pa_\xi^\alpha a_\chi \pa_x^\alpha b_\chi- (\pa_\xi^\alpha a \,\pa_x^\alpha b)_\chi$, $|\alpha|\leq 1$,  belong to $\Sigma_{L^\infty}^{m+m'-\vr}$  and then apply Theorem \ref{thm:cont2}.  
The spectral localization property follows because of the cut-off $\chi_\epsilon$ and $\epsilon$ small. 
As $\pa_x^\alpha$ commutes with the Fourier multiplier $ \chi_{\epsilon \la \xi \ra}(D)  $
we have  that $ \pa_x^\alpha  b_\chi = (\pa_x^\alpha  b)_\chi $
and we write $ \pa_{\xi}^\alpha a_\chi \, \pa_{x}^\alpha b_\chi - (\pa_{\xi}^\alpha a \, \pa_{x}^\alpha b)_\chi $ as
\begin{align}
 &(\pa_{\xi}^\alpha a)_\chi\left[(\pa_{x}^\alpha b)_\chi - \pa_{x}^\alpha b\right] + \left[(\pa_{\xi}^\alpha a)_\chi - \pa_{\xi}^\alpha a\right] \pa_x^\alpha b + \left[\pa_\xi^\alpha a\, \pa_x^\alpha b - (\pa_\xi^\alpha a \, \pa_x^\alpha b)_\chi\right]\label{blunotte}\\
&+\left[ \pa_{\xi}^\alpha a_\chi - (\pa_{\xi}^\alpha a)_\chi \right] (\pa_{x}^\alpha b)_\chi \, .\label{blunotte1}
\end{align}
Consider first the term in \eqref{blunotte1}. 
By  Lemma \ref{lem:achi}, 
 $\pa_{\xi}^\alpha a_\chi - (\pa_{\xi}^\alpha a)_\chi  \in \Gamma^{m-\vr-|\alpha|}_{L^\infty}$ and $(\pa_{x}^\alpha b)_\chi \in \Gamma_{L^\infty}^{m'+|\alpha|}$ and 
 by remark ($v$) after Definition \ref{def:sfr}, for any $ n \in \N_0 $,
\begin{align}
\notag
 \abs{\left[ \pa_{\xi}^\alpha a_\chi - (\pa_{\xi}^\alpha a)_\chi \right] (\pa_{x}^\alpha b)_\chi  }_{m+m'-\vr, L^\infty, n} 
\notag
& \leq 
 \big| \pa_{\xi}^\alpha a_\chi - (\pa_{\xi}^\alpha a)_\chi 
\big|_{m-|\alpha|-\vr, L^\infty, n}
 | (\pa_{x}^\alpha b)_\chi |_{m'+ |\alpha|, L^\infty, n}   \\
\notag
& \stackrel{\eqref{achi.est1},\eqref{achi.est3}}
\lesssim \abs{a}_{m, W^{\vr, \infty}, n+|\alpha|} \, \abs{b}_{m', L^{\infty}, n}  \, . 
\end{align}
Next consider the terms in \eqref{blunotte}. By remarks ($iii$), ($iv$) after Definition \ref{def:sfr},  we have
$ \pa_{\xi}^\alpha a \in \Gamma^{m-|\alpha|}_{W^{\vr, \infty}}\subset \Gamma^{m-|\alpha|}_{W^{\vr - |\alpha|, \infty}}$, $\pa_{x}^\alpha b \in \Gamma^{m'}_{W^{\vr -|\alpha|, \infty}}$, so we can apply  Lemma \ref{lem:achi},  
property \eqref{alg.symb} and  \eqref{sim1} to obtain
\begin{align}
\notag
\abs{\eqref{blunotte}}_{m+m'-\vr, L^\infty, n} 
&
\lesssim  \abs{a}_{m, W^{\vr-|\alpha|, \infty}, n+|\alpha|} \, \abs{b}_{m', W^{|\alpha|, \infty}, N}
+
\abs{a}_{m, L^\infty, n+|\alpha|} \, \abs{b}_{m', W^{\vr, \infty}, n} \\
& \lesssim 
\abs{a}_{m, W^{\vr, \infty}, n+1} \, \abs{b}_{m', L^\infty, n+1}
+
\abs{a}_{m, L^\infty, n+1} \, \abs{b}_{m', W^{\vr, \infty}, n+1} 
\label{comp9}
\end{align}
where to pass from the first to the second line we used the second interpolation inequality in \eqref{alg.symb}.
Altogether we have proved that the symbol in \eqref{R1} belongs to $\Sigma^{m+m'-\vr}_{L^\infty}$ and its seminorms are bounded by   \eqref{comp9}.
Then Theorem \ref{thm:cont2} proves that $R_1(a,b)$ fulfills
estimate \eqref{comp020}.\\

 \noindent{\underline{Estimate of $R_{2}(a,b)$}.} First we rewrite \eqref{R2} as 
 $$
 R_{2}(a,b)u =
\frac14 \sum_{j,\ell} \Big(  \int_0^1(1-t) \,  \sum_{|\alpha|=2} 
\wh f^\alpha_t (j-\ell, \ell) \, \di t \Big)  u_\ell \, e^{\im j \cdot x} 
 $$
where 
$$
\begin{aligned}
\wh{ f^\alpha_t}(n, \ell)
& := \sum_{k \in \Z^d}  
\wh {D_x^\alpha b_\chi}\Big(k-\ell, \frac{k+\ell}{2} \Big) \pa_\xi^\alpha \wh a_\chi\Big(n+\ell -k, \, \ell + \frac{n + t(k- \ell)}{2} \Big)   \\
& 
\stackrel{j=k-\ell} = \sum_{j \in \Z^d}  
\wh {D_x^\alpha b_\chi}\Big(j, \ell + \frac{j}{2} \Big) 
\pa_\xi^\alpha \wh a_\chi\Big(n-j, \, \ell + \frac{n + tj}{2} \Big)
\end{aligned}
$$ 
and $ D_{x_n} := \pa_{x_n} / \im $ and $ D_x^\alpha := 
D_{x_1}^{\alpha_1} \cdots D_{x_d}^{\alpha_d} $.  Then, recalling \eqref{B},
$$
R_{2}(a,b) u = \frac14 \int_0^1(1-t) \sum_{|\alpha|=2} 
\Op{ f^\alpha_t } u \, \di t  
$$
 where 
\begin{equation}\label{deffxxi}
f^\alpha_t(x,\xi):= 
\sum_{n,j} \wh {D_x^\alpha b_\chi}\Big(j,\xi  +\frac{j}{2} \Big)     \pa_\xi^\alpha 
\wh a_\chi\Big(n-j, \, \xi  + \frac{n + tj}{2} \Big)  e^{\im n \cdot x} .
\end{equation}
We claim that $ {f_t^\alpha} (x, \xi) $ is spectrally localized, namely 
\begin{equation}
\label{cutcut}
\exists \delta \in (0,1) \colon \ \ \  |n | \leq \delta \la \xi \ra , \quad \forall (n, \xi) \in \supp \wh{f_t^{\alpha}} \, . 
\end{equation}
In fact on the support of $\wh b_\chi\big(j, \xi  +\frac{j}{2} \big)$ we have, for some
$ \delta' \in (0,1) $,
\begin{equation}\label{pricut}
 |j|  \leq \delta' \langle \xi \rangle \, , 
 \end{equation}
whereas, on the support of
$  \pa_\xi^\alpha \wh a_\chi\big(n -j, \xi  + \frac{n + tj}{2} \big)$, $ t \in [0,1] $,  
\begin{align}
& | n - j | \leq \delta \langle \xi  \rangle + \delta \langle n \rangle + \delta \langle j \rangle 
\stackrel{\eqref{pricut}} \leq (\delta  + \delta \delta' ) 
\langle \xi  \rangle + \delta \langle n \rangle \, . \label{cutcut2}
\end{align}
The estimates \eqref{pricut}-\eqref{cutcut2}  then give
$ | n | \leq | j | +| n - j | \leq \delta' \langle \xi \rangle + (\delta  + \delta \delta' ) 
\langle \xi  \rangle + \delta \langle n \rangle  $, 
which implies  \eqref{cutcut}.

In order to apply Theorem \ref{thm:cont2} it remains to prove that,   
for any $ N \geq 3d+4 $, 
\begin{equation}\label{ftad}
 | f_t^\alpha (x,\xi)|_{m+m'-\vr, L^\infty, d+1}\lesssim 
 |  b |_{m', W^{\vr,\infty}, N} \ | a |_{m, L^\infty, N}  \, ,
\end{equation}
which implies, 
for any $s\in \R$, $ u \in \dot H^s $,  $
 \| R_{2}(a,b)u\|_{{s-m-m'+\vr}}\lesssim  |  b |_{m', W^{\varrho,\infty}, N}\  | a |_{m, L^\infty, N}\, \| u\|_{s} $. 
Thus $R_2(a,b)$ satisfies the estimate \eqref{comp020}.

In order to prove \eqref{ftad} note that, differentiating  \eqref{deffxxi}, 
for any $ \beta \in \N_0^{d} $,  
 \begin{align}\label{deffxxib}
\pa_\xi^\beta f^\alpha_t(x,\xi) & = 
\sum_{\beta_1+\beta_2 = \beta} C_{\beta_1, \beta_2}
\sum_{n,j} \wh {\pa_{\xi}^{\beta_1} 
D_x^\alpha b_\chi}\Big(j,\xi  +\frac{j}{2} \Big)     \pa_\xi^{\alpha+\beta_2} 
\wh a_\chi\Big(n-j, \, \xi  + \frac{n + tj}{2} \Big)  e^{\im n \cdot x}  \nonumber \\
& = 
\sum_{\beta_1+\beta_2 = \beta} C_{\beta_1, \beta_2}
\int_{\T^{2d}} K_t^{\beta_1, \beta_2} (x, y, z) \, \di y\, \di z 
\end{align}
where $ C_{\beta_1, \beta_2} $ are binomial coefficients and 
\begin{equation}\label{cappacappa}
K_t^{\beta_1, \beta_2} (x,y,z):= \frac{1}{(2\pi)^{2d}} \sum_{n,j}  (\pa_\xi^{\beta_1} D_x^\alpha b_\chi)\Big(x-z-y,\xi  +\frac{j}{2} \Big)   \pa_\xi^{\alpha+\beta_2} a_\chi\Big(x-z, \, \xi  + \frac{n + tj}{2} \Big)  e^{\im (n \cdot z + j \cdot y)} \, . 
 \end{equation}
 By \eqref{cutcut} and \eqref{pricut}
the sum over $n$ in \eqref{deffxxi} is restricted to  indexes satisfying 
$$ 
|n| \ll \la \xi \ra \, , \ |j|\ll \langle \xi \rangle \, , 
 \qquad {\rm and \ therefore} \qquad
\langle  \xi+\frac{j}{2}\rangle \sim\langle  \xi + \frac{n + tj}{2}\rangle \sim \langle \xi\rangle 
\, .
$$ 
We deduce that the sum in \eqref{cappacappa} is bounded by 
\begin{align}
\notag
\big| K_t^{\beta_1, \beta_2}(x,y,z) \big| &\lesssim  \langle \xi \rangle^{2d + m + m' - |\beta| -\vr} 
\, 
| \pa_\xi^{\beta_1} D_x^\alpha b_\chi |_{m'-|\beta_1|+|\alpha|-\vr, L^\infty, 0} 
\, 
 | \pa_\xi^{\alpha+\beta_2} a_\chi |_{m -|\alpha|-|\beta_2|, L^\infty, 0} \\
&
 \stackrel{\eqref{sim1}, \eqref{achi.est3}, \eqref{achi.est}} \lesssim
 \langle \xi \rangle^{2d + m + m' - |\beta|- \vr} 
 \, |b |_{m', W^{\vr, \infty}, |\beta|} \,  | a |_{m, L^\infty, 2+|\beta|}  \, , \label{cappa4} 
 \end{align}
 recalling that $|\alpha| = 2 $.
 We also estimate $K_t^{\beta_1, \beta_2} (x, y, z)$ applying Abel resummation formula \eqref{Abeld} in the sum \eqref{cappacappa},  in the  index $n$ and in the index $j$ separately,  obtaining, using \eqref{achi.est3}, \eqref{achi.est},  \eqref{finitediff} and \eqref{sim1},
\begin{equation}
\begin{aligned}
\big| K_t^{\beta_1, \beta_2}(x,y,z) \big| & \lesssim \langle \xi \rangle^{2d+m+m'-|\beta|-\vr} \, 
| b|_{m', W^{\vr,\infty} , 2d+1+|\beta|}
 | a |_{m, L^\infty, 2d+3+|\beta|} \\
& \quad \times  \min_{1\leq h \leq d} \Big( \abs{  \langle \xi \rangle 2 \sin \frac{y_h}{2} }^{-(2d+1)}, \  \abs{  \langle \xi \rangle 2\sin \frac{z_h}{2} }^{-(2d+1)} \Big) \, .  \label{cappanew} 
 \end{aligned}
\end{equation}
In view of  \eqref{cappa4}-\eqref{cappanew} and $ |\beta| \leq d+ 1 $, 
we apply Lemma \ref{lem:K}  with $d \leadsto 2d$, choosing
$A=\langle \xi \rangle $, 
$ B=\langle \xi \rangle^{m+m'-|\beta|-\vr} \,  | b|_{m', W^{\vr,\infty} , 2d+1+|\beta|} \, 
 | a |_{m, L^\infty, 2d+3+|\beta|} $
 and we obtain 
\begin{equation*}
\|  \pa_\xi^{\beta} f^\alpha_t(\cdot,\xi) \|_{L^\infty} \lesssim 
\int_{\T^{2d}} |K_t^{\beta_1, \beta_2} (x, y, z)| \, \di y\, \di z  
\lesssim
\langle \xi \rangle^{m+m'-\vr-|\beta|} \, 
| b|_{m', W^{\vr,\infty} , 3d+2} \, 
 | a |_{m, L^\infty, 3d+4} 
 \end{equation*}
proving \eqref{ftad}. 
\\
\indent
The proof that $R_3(a,b)$ and $R_4(a,b)$  satisfy the estimate \eqref{comp020}
 follows similarly. 
  \end{proof}

\end{document}